\documentclass[11pt]{scrartcl}

\usepackage{fullpage}
\usepackage{comment}

\usepackage[T1]{fontenc}
\usepackage[utf8]{inputenc}
\usepackage{lmodern}
\usepackage{mathtools}
\usepackage{tikz}

\usepackage{url}

\usepackage[english]{babel}
\usepackage{csquotes}

\usepackage{setspace}

\allowdisplaybreaks[1]
\usepackage{amsmath}
\usepackage{amsfonts}
\usepackage{amssymb}
\usepackage{amsthm}
\usepackage{stix}

\numberwithin{equation}{section}

\usepackage{thmtools}

\declaretheorem[
	name=Theorem,
	numberwithin=section
	]{thm}

\declaretheorem[
	name=Conjecture,
	sibling=thm
	]{conj}
\declaretheorem[
	name=Lemma,
	sibling=thm,
	]{lem}
\declaretheorem[
	name=Proposition,
	sibling=thm,
	]{prop}
\declaretheorem[
	name=Corollary,
	sibling=thm,
	]{cor}

\declaretheorem[
	name=Remark,
	style=remark,
	sibling=thm,
	]{rem}

\def\cT{\mathcal{T}}

\def\cR{\mathcal{R}}

\def\cL{\mathcal{L}}

\def\cF{\mathcal{F}}

\def\cC{\mathcal{C}}

\def\P{\mathbb{P}}
\def\Q{\mathbb{Q}}

\def\He{\mathbb{H}}
\def\D{\mathbb{D}}
\def\DD{\mathbb{D}}

\def\E{\mathbb{E}}
\def\C{\mathbb{C}}
\def\R{\mathbb{R}}
\def\Z{\mathbb{Z}}

\def\T{\mathbb{T}}

\def\Gdim{{\Gamma_\delta}}
\def\Gtree{\Omega^\delta}

\def\ph{\varphi}
\def\eps{\varepsilon}

\DeclareMathOperator{\dist}{dist}

\DeclareMathOperator{\Int}{Int}

\DeclareMathOperator{\LE}{\mathsf{LE}}

\newcommand{\dd}{d}

\DeclareMathOperator{\Cross}{\mathsf{Cross}}

\def\la{\langle}
\def\ra{\rangle}

\DeclareMathOperator{\diam}{Diam}
\DeclareMathOperator{\crad}{crad}

\newcommand{\old}[1]{{}}

\title{Near-critical dimers and massive SLE}
\author{Nathanaël Berestycki, Levi Haunschmid-Sibitz}

\def\Dprime{{\Omega'}}
\def\aprime{{a'}}
\def\sigmaprime{\sigma'}

\usepackage{hyperref}
\hypersetup{
    colorlinks=true,
    linkcolor=blue,
    citecolor=red,
    urlcolor=blue,
    pdfborder={0 0 0}
}

\begin{document}

\maketitle
\begin{abstract}
We consider the dimer model on the square and hexagonal lattices with doubly periodic weights. 
The purpose of this paper is threefold: (a) we establish a rigourous connection with the massive SLE$_2$ constructed by Makarov and Smirnov \cite{MakarovSmirnov} (and recently revisited by Chelkak and Wan \cite{ChelkakWan}); (b) we show that the convergence takes place in \emph{arbitrary} {bounded} domains subject to Temperleyan boundary conditions, and that the scaling limit is universal; and (c) we prove conformal covariance of the scaling limit. For this we introduce an inhomogeneous near-critical dimer model, corresponding to a drift for the underlying random walk which is a smoothly varying vector field or alternatively to an inhomogeneous mass profile. When the vector field derives from a log-convex potential  we prove that the corresponding loop-erased random walk has a universal scaling limit. Our techniques rely on an exact discrete Girsanov identity on the triangular lattice which may be of independent interest. We complement our results by stating precise conjectures making connections to a generalised Sine-Gordon model at the free fermion point. 

\end{abstract}

\tableofcontents

\hypersetup{pageanchor=true}
\renewcommand{\thepage}{ \arabic{page} }

\section{Introduction}

Makarov and Smirnov initiated in \cite{MakarovSmirnov} a programme to describe near-critical scaling limits of planar statistical mechanics models in terms of massive SLE and/or Gaussian free field. To quote from their paper:

\emph{The key property of SLE is its conformal invariance, which is expected in 2D
lattice models only at criticality, and the question naturally arises:
Can SLE success be replicated for off-critical models?
In most off-critical cases to obtain a non-trivial scaling limit one has to adjust some
parameter [...], sending it at an appropriate speed to the critical value. Such limits
lead to massive field theories, so the question can be reformulated as whether one
can use SLEs to describe those. Massive CFTs are no longer conformally invariant,
but are still covariant when mass is considered as a variable covariant density [...].}

\medskip As part of this programme, Makarov and Smirnov introduced a massive version of SLE$_2$, which will be defined more precisely in Section \ref{SS:massiveSLE}. As established rigourously recently by Chelkak and Wan \cite{ChelkakWan}, this can be seen as the scaling limit of the loop-erasure of a massive random walk, i.e., a random walk which has a fixed probability of being killed at every step, and which is conditioned to leave the domain before being killed. Makarov and Smirnov also listed a number of fascinating questions, many of which remain open today.

In this paper we carry out part of this programme for the near-critical dimer model. The dimer model is one of the most classical models of statistical mechanics, and is equivalent to random matchings on a planar bipartite graph. That is, given such a (finite) graph $G$, we associate to every dimer covering (or perfect matching) $\mathbf{m} $ (a subset of the edges such that every vertex is covered exactly once) the Gibbs weight
$$
\P( \mathbf{m}) = \frac1Z \prod_{e \in \mathbf{m}} w_e,
$$
where $w_e>0$ are given edge weights and $Z$ is a normalisation constant (partition function). The model is also equivalent to tilings (in particular to lozenge tilings if the underlying graph is the hexagonal lattice; see \cite{Gorin} for a recent superb introduction).
 The study of the dimer model goes back to the pioneering work of Temperley and Fisher \cite{FisherTemperley} and Kasteleyn \cite{Kasteleyn1961TheLattice}, who computed its partition function, and noted that it is equal (up to a sign or more generally a complex number of modulus one) to the determinant of a matrix now called the Kasteleyn matrix, which is a suitably weighted adjacency matrix. This identity is the starting point of a far-reaching theory which eventually led Kenyon to prove convergence (subject to so-called Temperleyan boundary conditions, described below) of the associated height function to a Gaussian free field in a sequence of two landmark papers \cite{Kenyon_confinv}, \cite{KenyonGFF} when all edge weights are equal. This was the first proof of conformal invariance for a planar model of statistical mechanics.

\subsection{Off-critical dimer model.} \label{SS:offcriticalintro} In this paper we are concerned with an off-critical model, which can be defined either on the square lattice or on the hexagonal lattice when the edge weights are assumed to be doubly periodic, in the following sense. We start with the square lattice. Let $s_0, \ldots, s_3>0$. We divide the square lattice into the usual black and white vertices in checkboard fashion, and the black vertices are themselves divided into two alternating classes $B_1$ and $B_2$ (as in \cite{Kenyon_confinv}). We declare that around every $B_1$ vertex, the edge weights are respectively $s_0, \ldots, s_3$ as we move in the clockwise direction starting from the east (thus $s_k$ corresponds to the direction $\mathbf{i}^k = e^{\mathbf{i} k \pi/2}, k = 0, \ldots, 3$; here $\mathbf{i}= \sqrt{-1}$). All other edge weights are set to 1. See Figure \ref{F:double} for an illustration. We will further specify the weights $s_k$ so as to be in the near-critical regime in \eqref{eq:weights_square}.

\begin{figure}
  \begin{center}
    \includegraphics[width = .3 \textwidth]{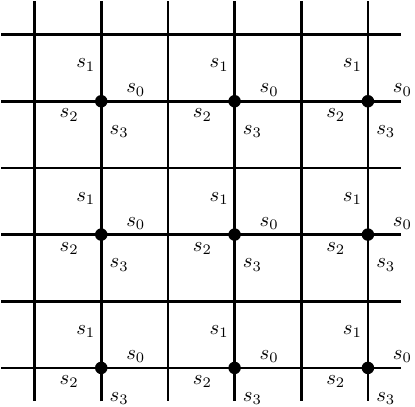}\quad
    \includegraphics[width = .3 \textwidth]{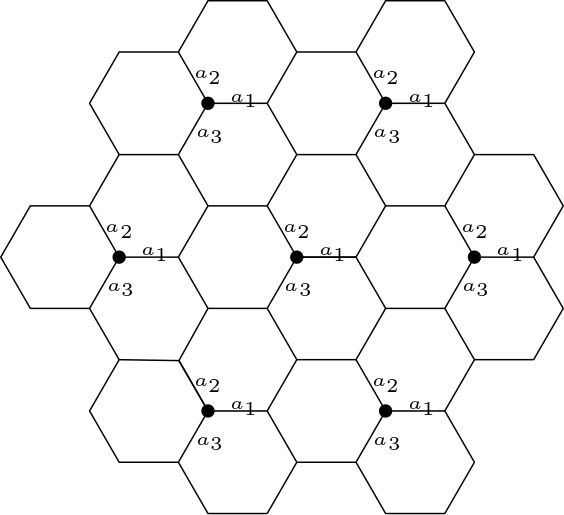}
  \end{center}
  \caption{Doubly periodic weights on the square and hexagonal lattices. Black vertices of type 1 are marked with disks. The weights $s_0, \ldots, s_3$ or $a_0,\dots,a_2$ are periodically repeated around every black vertex of type 1. Every other edge weight is equal to 1.}
  \label{F:double}
\end{figure}

A similar construction can be applied to the hexagonal lattice. Consider the usual black and white colouring of the vertices of the hexagonal lattice $\He$. Black vertices at distance two apart in $\He$ form a triangular lattice, which is a tripartite graph. So all black vertices in $\He$ belong one of three possible classes, $B_1, B_2, B_3$, say. We declare that the edge weights around a $B_1$ vertex are respectively $a_0,a_1$ and $a_2$ going counter-clockwise starting from the east direction See Figure \ref{F:double}.

This model was first considered in the work of Chhita \cite{Chhita2012} in the case of the square lattice, who called it the ``drifted'' dimer model, for reasons that will become clear later. 
Suppose $s_k = 1+ c_k \delta$, where $\delta$ tends to zero (we will later identify $\delta$ with the mesh size). This scaling will be enforced throughout the paper. As already noted in \cite{Chhita2012}, this choice of scaling essentially corresponds to studying the \textbf{liquid-gas boundary} of the dimer phases. 
When applying the treatment of Kenyon \cite{Kenyon_confinv} to this model, if $K$ denote the associated Kasteleyn matrix then one can easily check that $L = K^*K$, viewed as an operator on the black vertices, is approximately the negative of a massive Laplacian: indeed, on the $B_1$ vertices, the diagonal entry is of the form $s_0^2 + \ldots +s_3^2$, while the sum of the off-diagonal entries is $-2s_1s_3 - 2s_0 s_2$. (The reason why this is only an approximation is because terms of the form $L(b_1, b_2)$ are not all exactly zero when $b_1 \in B_1, b_2 \in B_2$; they are simply lower order than $L(b_1, b'_1)$ for $b_1, b'_1 \in B_1$). (In fact, after a suitable transformation, the inverse Kasteleyn matrix can be related to a modified Kasteleyn matrix which corresponds exactly to the Green function of a massive random walk, see Section 3 of \cite{Chhita2012}).

From this it is perhaps natural to conjecture that the height function, suitably rescaled, converges to the \textbf{massive Gaussian free field}, which is (informally) the Gaussian field whose covariance matrix is the massive Green function. Surprisingly, however, \cite{Chhita2012} showed that while there is a scaling limit for the height function as $\delta \to 0$ in the full plane, the limit cannot be the massive Gaussian free field since its moments do not even satisfy the Wick relation, hence it is not even Gaussian.

\medskip The purpose of this paper is threefold:

\begin{itemize}
  \item First, we extend the results of \cite{Chhita2012} in several different ways: we consider not only the square lattice but also the hexagonal lattice; furthermore our results are not only valid in the whole plane but in arbitrary simply connected domains subject to Temperleyan boundary conditions (these are perhaps the nicest boundary conditions from the combinatorial point of view and are defined immediately below in Section \ref{sec:temperleyan}).

  \item Second, we show for the first time a connection to massive models and more specifically to the massive SLE$_2$, constructed by Makarov and Smirnov \cite{MakarovSmirnov} and revisited recently by Chelkak and Wan \cite{ChelkakWan}.

  \item Finally, we show that the scaling limit of the height function obeys a certain conformal covariance rule. This is reminiscent of other near-critical scaling limits previously obtained e.g. for percolation \cite{GarbanPeteSchramm_nearcritical}. Interestingly however, the covariance rule involves not only the modulus of the derivative of the conformal map but also its argument.

\end{itemize}

Last but not least, this will be complemented by some novel conjectures attempting to make a connection with a generalised Sine-Gordon model (which will be introduced below) at its free fermion point. Along the way we identify a larger and more interesting family of near-critical dimer models which give an intuitively transparent explanation for why and how the Sine-Gordon model is connected to near-critical dimers; these models are characterised by the fact that the mass (or equivalently the drift) is inhomogeneous. 

At the technical level a key contribution of this paper will be an exact discrete Girsanov identity on the triangular lattice as well as a proof that the loop-erasure of a random walk with drift which may vary with the position has a scaling limit. 

\subsection{Temperleyan boundary conditions.}\label{sec:temperleyan}

To make the connection to massive SLE and state our results, we will now define precisely the type of boundary conditions we impose on the model, which in the case of the square grid are known as Temperleyan. We recall the definition in this case first. Let $\Omega \subset \mathbb{C}$ be a bounded simply connected domain of the complex plane. Let $\Gdim  = (v(\Gdim), E(\Gdim))$ be a sequence of graphs in $\delta \Z^2$approximating $\Omega$ in the following sense.
That is, $\Gdim$ is a planar graph with vertex set $v(\Gdim) \subset \Omega \cap (\delta \Z^2)$ and edge set $E(\Gdim)$ such that if $x,y \in v(\Gdim)$ and $x \sim y$ in $\delta \Z^2$, then $(x,y) \in E(\Gdim)$ if and only if $[x,y] \subset \Omega$.
Furthermore, assume that $\Gdim$ is connected and that $\Gdim$ is \textbf{Temperleyan}: namely, all corners (be them convex or concave) are of type $B_2$, and one further such corner has been removed. See e.g. Figure \ref{F:temp}. Equivalently, along the vertex boundary, all black vertices are of type $B_2$, i.e, the boundary alternates between $B_2$ and white vertices (except at the removed corner).
Further we will assume that the associated graph $\Gtree$, defined in the next section, converges to $\Omega$ in the Carathéodory topology.
Note that $\Gdim$ being connected and Temperleyan can be obtained by choosing $v(\Gdim)$ correctly.
In particular, it is not equal to $\Omega \cap (\delta \Z^2)$ in general.

We make a similar definition in the hexagonal case. We say that the domain $\Gdim$ whose vertices are in $\delta \He$ is Temperleyan if the boundary does not contain any $B_1$ vertices (i.e., consists only of $B_2$ and $B_3$ and white vertices), and a vertex of type $B_2$ or $B_3$ has been removed. Figure \ref{F:temp} shows examples of a Temperleyan domain on both the square and  hexagonal lattices.

\begin{figure}
  \begin{center}
    \includegraphics[width = .3 \textwidth]{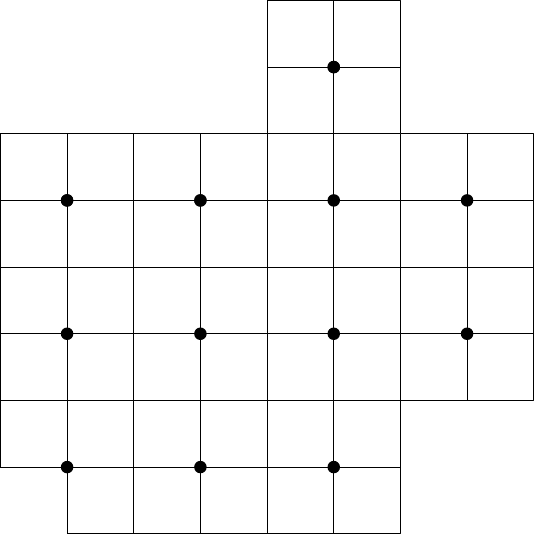}\quad
    \includegraphics[width = .3 \textwidth]{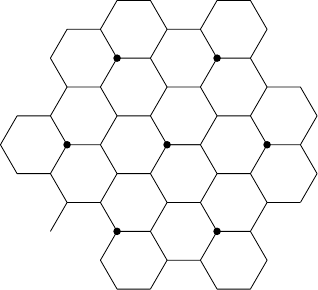}
  \end{center}
  \caption{A Temperleyan domain on the square lattice and a Temperleyan domain on the hexagonal lattice. In both the black vertices of type $B_1$ have been highlighted and a (non-$B_1$) black vertex on the lower left boundary has been removed.}
  \label{F:temp}
\end{figure}

\subsection{Temperley's bijection.}
\label{sec:Temperleybijintro}
Temperley's bijection is a powerful tool which relates the dimer model on the Temperleyan graph $\Gdim$ to a pair of spanning trees on a different graph. As it turns out, the Temperleyan boundary conditions described above are such that both dimer models (i.e., on the square and hexagonal lattices respectively) are equivalent to a certain spanning tree on a (possibly directed) graph $\Gtree$ whose vertices are the $B_1$ vertices of $\Gdim$ (or, equivalently, to a pair of dual spanning trees on $\Gtree$ and its planar dual). In the square lattice (and for rectangles) this goes back to the original paper of Temperley and Fisher \cite{FisherTemperley}. This was considerably generalised and strengthened in many subsequent works, in particular, \cite{KenyonProppWilsonGeneralizedTemperley}. That paper included the perhaps less-well known case of the hexagonal lattice, which we will use in this paper and will be recalled in more detail in Section \ref{SS:Tbij_hex}; in that case, the corresponding graph $\Gtree$ of the spanning tree is the directed triangular lattice with mesh size $\delta$.

As developed in the sequence of papers \cite{BLR_DimersGeometry, BLR_torus, BLRriemann2}, Temperley's bijection 
can be used to describe the scaling limit of the height function fluctuations via a random geometric approach. Essentially these papers reduce the problem of finding the scaling limit of the dimer height function to the (easier) problem of finding a scaling limit for the associated Tempereleyan tree in the Schramm topology: in other words, to the question of the scaling limit of a single branch of that tree. In turn, by Wilson's algorithm, this boils down to the scaling limit of the loop-erasure of the random walk on the (possibly directed) graph $\Gtree$.

\subsection{Massive SLE\texorpdfstring{$_2$}{2}.} \label{SS:massiveSLE} As already mentioned, the construction of massive SLE$_2$ was sketched by Makarov and Smirnov in \cite{MakarovSmirnov} and recently revisited by Chelkak and Wan \cite{ChelkakWan} (see also \cite{Bauer2008LERWSLEs} for a mathematical physics perspective). We will describe it in the radial case for ease of comparison with the situation which is of interest to us, though one should note that Chelkak and Wan's paper actually deals with the chordal case. 

A massive random walk (on the square lattice, say) is a walk which has a chance of order $\delta^2$ to be killed at every time step (the constant of proportionality is by definition $m^2/2$, where $m\ge 0$ is the mass), and otherwise moves like ordinary walk. Massive radial SLE$_2$ describes the scaling limit of the loop-erasure of a massive random walk from $o$ to $a$ (where $a$ is on the boundary of a simply connected domain $\Omega$, and $o$ is in the interior of $\Omega$), conditioned on not getting killed before reaching $a$. In fact, it is more convenient to define massive SLE$_2$ by its associated Loewner flow, which in the radial case is defined by Loewner's equation (parametrised by capacity)
$$
\frac{d g_t (z)}{dt} = - g_t(z) \frac{g_t(z) + \zeta_t}{g_t(z) - \zeta_t} ; z \in \Omega_t
$$
where $\varphi$ is a fixed conformal map sending $\Omega $ to $\mathbb{D}$ and $o$ to $0$, $\Omega_t$ denotes the slit domain $\Omega \setminus \gamma([0,t])$ (since $\kappa =2$ we do not need to remove more than that), $g_t$ is the Loewner map from $\Omega_t$ to $\mathbb{D}$, and if we write the driving function in the form $ \zeta_t = e^{i \xi_t}$, then $\xi$ solves the Stochastic Differential Equation:
\begin{equation}\label{eq:massiveSLE}
d\xi_t = \sqrt{2} dB_t + 2 \lambda_t dt; \lambda_t = \frac{\partial}{\partial g_t(a_t)}\log \frac{P^{(m)}_{\Omega_t} (o, a_t)}{P_{\Omega_t} (o,a_t)}.
\end{equation}
Here $a_t=\gamma(t)$, and $P^{(m)}_{\Omega_t}$ and $P_{\Omega_t}$ are the so-called invariant Poisson kernels for the Brownian motion with mass $m$, and regular Brownian motion respectively, in $\Omega_t$. By Brownian motion with mass $m$ we mean the law of a standard Brownian motion, killed at rate $m^2$ (which defines a subprobability measure)\footnote{Note that the probability with which the above discrete random walk is killed at each step (namely, $(1/2)m^2 \delta^2$) is chosen so that in the scaling limit, we obtain a Brownian motion killed at rate $m^2$. Indeed, if $X^\delta$ is an ordinary random walk on $\delta \Z^2$, then $(X_{2\delta^{-2} t})_{t\ge 0}$ converges weakly uniformly on compacts to a standard planar Brownian motion. In \cite{ChelkakWan} the killing probability is chosen to be $ \delta^2 m^2$ instead of $\delta^2m^2/2$ but this appears to be a typo.}. The notion of invariant Poisson kernel is quite delicate, especially in the massive case, and/or if the domain $\Omega_t$ has a rough boundary (which is a.s. the case here, as $\Omega_t$ is the complement in $\Omega$ of a curve which is absolutely continuous with respect to SLE$_2$). We explain this carefully in Section \ref{S:LERW_general}. For now, we simply briefly explain what we mean in the non-massive case (i.e., for standard Brownian motion). If $\Omega$ is smooth, the density of harmonic measure in $\Omega$ viewed from $x \in \Omega$, with respect to arclength on $\partial \Omega$, defines a conformally covariant function $h_\Omega (x,a)$. The ratio $h_\Omega(x,a) / h_\Omega(o, a)$ (for a fixed reference point $o \in \Omega$) thus defines a conformally invariant function, which is one of the definitions of the invariant Poisson kernel. (The word invariant refers to the above conformal invariance). Since it is conformally invariant, it also makes sense for domains $\Omega$ that are rough (i.e., where arclength is not even defined) so long as $\Omega$ is simply connected; the boundary point $a$ is then viewed as a prime end  (or equivalently a point on the Martin boundary) of $\Omega$.

The above expression for $\lambda_t$ is that in \cite{ChelkakWan} which is somewhat similar to the one appearing in Makarov and Smirnov \cite[(9)]{MakarovSmirnov}. However we feel it deserves a few explanations, notably considering the meaning of the differentiation $\tfrac{\partial}{\partial g_t(a_t)}$. This should be understood as a spatial differentiation with respect to $x \in \R$ after mapping $\Omega_t$ to $\D$, setting $a_t = e^{ix}$, and evaluating the result at $x = \xi_t$. That is, let $\rho = m^2$, and let $\rho_t$ denote the (squared) mass profile in $\D$ that corresponds to the constant mass $m $ in $\Omega_t$; i.e., after mapping to $\D$ (and applying the relevant time change) we obtain a Brownian motion killed at rate $\rho_t(x)$). Then 
$$
\lambda_t = \frac{\partial}{\partial x} \left.\log \frac{P_{\D}^{(\rho_t)} (0, e^{ix})}{P^{(0)}_\D (0, e^{ix})} \right|_{x = \xi_t}.
$$
By choice of normalisation, the denominator in the fraction is simply equal to 1, so that this logarithmic derivative can also be written in the form: 
$$
\lambda_t = \left. \frac{\tfrac{\partial}{\partial x}  P_{\D}^{(\rho_t)} (0, e^{ix})}
{P^{(\rho_t)}_\D (0, e^{ix})} 
\right|_{x = \xi_t}.
$$
The spatial derivative of the invariant Poisson kernel, i.e., the numerator of this fraction, is a quantity which can be shown to correspond to what Chelkak and Wan \cite{ChelkakWan} denote by $Q^{(\rho_t)}_\D (x, a_t)$. Their result in fact establishes convergence with 
\begin{equation}\label{eq:lambda_t_actual}
\lambda_t = 
\frac{Q_{\D}^{(\rho_t)}(0, e^{i \xi_t})}{P_\D^{(\rho_t)} (0, e^{i \xi_t})} = 
\frac{Q_t^{(\rho)} (o, a_t)}{P_t^{(\rho)} ( o, a_t)}.
\end{equation}
Proving this chain of identities would require some arguments. This is circumvented by defining $\lambda_t $ (i.e., the right hand side of \eqref{eq:massiveSLE}) as in the right hand side of \eqref{eq:lambda_t_actual}.


The description above is then a theorem proved in the chordal case and on the square lattice by \cite{ChelkakWan} (the radial case is briefly discussed as being analogue to, and in fact a little simpler than, the chordal case). 
See Theorem 1.1 in \cite{ChelkakWan} for a precise statement, and see \cite{Lawler2008ConformallyPlane} as well as \cite{Berestycki2014LecturesEvolution} for general references on SLE.

\subsection{Main results}

Our first result below concerns the branches of the Temperleyan tree for an off-critical dimer model on a graph (defined more precisely below, which may be a piece either of the square lattice or of the hexagonal lattice, scaled by $\delta$) with Temperleyan boundary conditions, as explained above. The result shows that the scaling limit exists, and furthermore gives a connection to massive models. On the square lattice, suppose that the weights $s_0, \ldots, s_3$ satisfy
\begin{equation} \label{eq:weights_square}
  s_k = 1 + c_k \delta \quad (k = 0, \ldots, 3) 
\end{equation}
counterclockwise from the east direction, while on the hexagonal lattice we assume that the weights $a_0, \dots, a_2$ satisfy
\begin{equation}\label{eq:weights_hex}
  a_k = 1+ c_k \delta \quad (k = 0, \ldots, 2)
\end{equation}
also counterclockwise from the east direction. We consider the associated rescaled drift vector $\alpha$ defined respectively by
\begin{equation}\label{E:alpha_intro}
\alpha = \frac12\sum_{k=0}^3c_k\mathbf i^k ; \alpha = \frac23\sum_{k=0}^2c_k\tau^k,
\end{equation}
where ${\mathbf i} = \sqrt{-1}=e^{\mathbf{i} \pi/2}$ and $\tau = e^{2\mathbf{i} \pi/3 }$ are the fourth and third roots of unity, respectively. The scaling factors in front of these expressions are chosen to guarantee that in the scaling limit, a random walk with the above weights converges to Brownian motion with drift $\alpha$: that is, if $X^\delta$ denotes this random walk on either $\delta \Z^2$ or $\delta \T$, then $(X^\delta_{2\delta^{-2} t})_{t\ge 0 }$ converges weakly, uniformly on compacts, to $(B_t + \alpha t)_{t\ge 0}$, where $B$ is a standard planar Brownian motion.

We also assume
\begin{equation}\label{E:equalaxes}
c_0 + c_2 = c_1 + c_3
\end{equation}
in the square lattice case.
See Remark \ref{R:equalitycondition} for a discussion of this condition. 

We suppose we are given a Temperleyan lattice domain $\Gdim$ as in Section \ref{sec:temperleyan} and a dimer model on $\Gdim$. Applying the Temperleyan bijection leads to a pair of dual trees respectively on $\Gtree$ and its dual, where $\Gtree$ is a subgraph of either $\delta\mathbb T$ or $\delta\mathbb Z^2$.
Note that there is natural edge boundary $\partial\Gtree$ on $\Gtree$, corresponding to pair of vertices $(y_1, y_2)$ of the lattice  
such that $y_1$ and $y_2$ are neighbours in the lattice, at least one of $y_1$ or $y_2$ is a vertex of $\Gtree$ but not both. With a slight abuse of notation we will still refer to $y_\delta$ as a point on the boundary, identify it in calculations with $y_2$ and say that $y_\delta\to y$ if $y_2$ converges to $y$ as $\delta\to 0$ (or equivalently $y_1$). Likewise, we will often consider the random walk $(X_n, \ge 0)$ on $\Gtree$. 
With an abuse of notation we will refer to the first time $\sigma$ that the walk leaves $\Gtree$ as the smallest  $n\ge 1$ such that $(X_{n-1}, X_n)$ is a boundary edge. We will also identify, with an abuse of notation, the position $X_\sigma$ with the boundary edge $(X_{\sigma -1}, X_\sigma)$, and denote it by $Y_\delta$ in the following.

In fact, the distribution of $Y_\delta$ converges weakly to a distribution $\mu^{(\alpha)}_z $ on $\partial \Omega$, which is the exit law from $\Omega$ of Brownian motion with unit covariance matrix and drift vector $\alpha$. We therefore obtain the following result.

\begin{thm}\label{T:constantdrift_tree}
  \label{T:convergenceTree} Let $\cT_\delta$ denote the Temperleyan tree associated with the dimer configuration in $\Gdim$ (either in the hexagonal or square lattice case). Then as $\delta \to 0$, the tree $\cT_\delta$ converges in the Schramm sense to a continuum limit tree $\cT$. Each branch of this tree from a point $z \in \Omega$ has the following law: sample $a$ according to $\mu^{(\alpha)}_z$; given $a$, the branch of $\cT$ from $z$ to $a$ has the law of massive radial SLE$_2$ with mass $m = \|\alpha\|/\sqrt{2}$ (associated with a Brownian motion killed at rate $\|\alpha\|^2 / 2$).
\end{thm}

A key result from \cite{BLR_DimersGeometry} (see also \cite{BLR_torus}) is that the convergence of the Temperleyan tree implies the convergence of the dimer height function. This requires only a uniform crossing estimate and some basic estimates such as polynomial decay on the probability for the loop-erasure to visit a small ball, and control on the moments of winding close to a point (these estimates are a fairly simple consequence of our work, and are written explicitly at the end of Section \ref{sec:drift} in a more general context).
We obtain the following corollary:

\begin{cor}\label{C:height}
In the setup of Theorem \ref{T:convergenceTree}, the centered height function $h^\delta - \E( h^\delta)$ converges to a limit as $\delta \to 0$ whose law depends only on the vector $\alpha$ defined in \eqref{E:alpha_intro}.
\end{cor}

\subsection{Exact Girsanov identity}

To establish these results, we observe that the law of a branch in the Temperleyan tree may be described via Wilson's algorithm as the loop-erasure of a random walk on $\Gtree$ with near-critical weights defined by \eqref{eq:weights_square} on the square lattice and \eqref{eq:weights_hex} on the directed triangular lattice respectively. The random walk corresponding to these weights is one which has a drift: as the mesh size $\delta \to 0$, the random walk converges to a Brownian motion with drift vector $\alpha$ defined in \eqref{E:alpha_intro}. 
 Furthermore, using a discrete Girsanov transform (which will be detailed below), we relate the corresponding random walks to massive ones on the same lattices; the above result then intuitively follows by the known convergence of the massive LERW to the massive SLE$_2$ of Makarov and Smirnov (proved rigourously by Chelkak and Wan recently in \cite{ChelkakWan}).

 We now describe our Girsanov identity. As this holds independent of any scaling limit consideration we formulate it in the unscaled triangular lattice $\mathbb{T}$ and square lattice $\mathbb{Z}^2$. The Girsanov identity takes a slightly different form in each case. Although both are exact formulas, the connection between massive and drifted walk is only exact on the triangular case (Corollary \ref{conditionedequality}) whereas it is approximate in the case of the square lattice (Corollary \ref{P:massdrift_square}). On the other hand, the application of the results of Chelkak and Wan \cite{ChelkakWan} in the directed triangular case needs additional arguments because of the lack of reversibility. As we believe this result is of independent interest, we state it below on the triangular lattice where the statement is the simplest.  We consider a Markov chain on the (directed) triangular lattice $\mathbb{T}$ where the jump probabilities are allowed to depend on the position of the vertex $v$ of the triangular lattice $\mathbb{T}$. That is, suppose given for any $v \in \mathbb{T}$, a collection of parameters $(\alpha_0 (v), \alpha_1 (v), \alpha_2 (v)) \in \R^3$, and
let $\Q$ denote the law of a Markov chain such that if the walk is at the vertex $v$, then the jump probabilities are given by
\begin{equation}\label{eq:driftweights}
\Q(v,v+ \tau^{k} ) = \frac{e^{\alpha_k(v)}}{a(v)}, \quad k = 0, \ldots, 2, \quad \text{ with } a(v)  = e^{\alpha_0 (v)}+ \ldots + e^{\alpha_2(v)}.
\end{equation}
Let also $Y$ denote the position of the random walk when it hits $\partial \Omega$ and let $\Q (\cdot | Y = y)$ denote the conditional law given the exit point is $y$. We also let $\P = \P^{(0)}$ denote the law of the usual simple random on the directed triangular lattice $ \mathbb{T}$.

Fix $\gamma = (x_0, \ldots, x_{n}) $ a given path on the triangular lattice, starting from some point $x_0 = z \in \Omega$ of some length $n= N(\gamma)$. Let $\dd x_s = x_{s+1} - x_s \in \{1, \tau, \tau^2\} \subsim \R^2$, for $s = 0, \ldots, n-1$ denote the discrete derivative of $\gamma_s$ at time $s$.
Define  $ \beta(v)>0$ by
\begin{equation}\label{E:tribeta}
 \exp(- \beta (v)^2)= (\tfrac{a(v)}{3})^{-3} \prod_{k=0}^2 e^{\alpha_k(v)},
\end{equation}
which is well-defined by the arithmetic-geometric mean inequality. Let $\alpha=\alpha(v) = \frac23(\alpha_0+\alpha_1\tau+\alpha_2\tau^2)$, which is a complex number (identified with a vector in $\R^2$) associated to every vertex $v$ of the triangular lattice $\mathbb{T}$.

Note that while $\alpha$ does not uniquely determine the $\alpha_i$, it does determine $\Q$, since the transition probabilities in \eqref{eq:driftweights} do not change under a shift $(\alpha_1,\alpha_2,\alpha_3)\to(\alpha_1+x,\alpha_2+x,\alpha_3+x)$ and all $(\alpha_1,\alpha_2,\alpha_3)$ corresponding to a specific $\alpha$ are related to one another by such a shift. Thus given a vector $\alpha (v)\in \R^2$, there is a unique choice of $\alpha_0(v), \alpha_1( v), \alpha_2(v)\in \R$ summing to zero such that $\alpha(v) = \tfrac23 (\alpha_0(v) + \alpha_1(v) \tau + \alpha_2(v) \tau^2)$.
The following gives us an exact value for the global Radon--Nikodym derivative of the law $\Q$ compared to $\P^{(0)}$.
\begin{thm}
\label{L:girsanov_discrete_tri}
\begin{equation}\label{eq:MnVn}
\frac{\mathbb Q_x(\gamma)}{\mathbb P_x(\gamma)}=\exp(M_{n}-\tfrac12V_{n}),
\end{equation}
where
\begin{equation}
M_n=\sum_{s=0}^{n-1}\langle \alpha(x_s),\dd x_s\rangle
\quad\text{and}\quad
V_n=\frac{2}{3} \sum_{s=0}^{n-1}\beta^2(x_s).
\end{equation}
\end{thm}

Of particular relevance in this article will be the case where the drift vector $\alpha = \alpha (v), v \in \Omega^\delta$ derives from a \textbf{potential function} $\Phi : \mathbb{T} \to \R$, i.e., when 
\begin{equation}\label{eq:discretegradient}
\alpha (v) = \nabla^{\mathbb{T}} \Phi (v) : = \frac23\sum_{i=0}^2 ( \Phi (v+ \tau^i) - \Phi(v)) \tau^i;
\end{equation}
in other words, $\alpha_{i+1}(v) = \Phi ( v+ \tau^i) - \Phi (v)$ for $0\le i \le 2.$ 
If $\alpha$ is of this form, the Radon--Nikodym derivative in Theorem
\ref{L:girsanov_discrete_tri} 
takes a particularly nice form:

\begin{cor} \label{cor:girsanovpotential}
Suppose $\alpha$ derives from a potential function $\Phi$ as above. Then 
$$
\frac{\mathbb Q_x(\gamma)}{\mathbb P_x(\gamma)} = \exp \left( \Phi (x_n) - \Phi (x_0) - A_n\right); 
$$
where
$$
A_n = \sum_{s=0}^{n-1} \Delta^{\mathbb{T}} \Phi (x_s) + \frac13 \beta^2 (x_s).
$$
Here $\Delta^{\mathbb{T}} \Phi  (x) = \frac13 \sum_{i=0}^2 \Phi (x+ \tau^i) - \Phi (x)$ is the usual graph Laplacian on the directed triangular lattice $\mathbb{T}$.
\end{cor} 


To understand the formulas in Theorem \ref{L:girsanov_discrete_tri} and Corollary \ref{cor:girsanovpotential}, we now explain how both should be viewed as the discrete analogues of Girsanov's theorem followed by an application of It\^o's formula. Indeed, in the continuum, if $\Q$ is the law of the solution of the stochastic differential equation (SDE)
\begin{equation}
\dd X_t = \dd B_t + \alpha (X_t) \dd t \, ; \quad \text{where} \quad \alpha (x) = \nabla \ph(x)
\end{equation}
and where $\ph$ is a smooth Lipschitz function on $\R^2$, then 
\begin{align}
\left.\frac{d\mathbb{Q}}{d\mathbb{P}}\right|_t & = 
\exp \left( \int_0^t \alpha (X_s )\cdot dX_s - \frac12 \int_0^t \|\alpha (X_s)\|^2 ds \right) \label{eq:massGirsanov}\\
& = \exp\left( \ph(X_t) - \ph (X_0) - \frac12\int_0^t \Delta \ph(X_s)+\|\nabla \ph(X_s)\|^2 ds \right)\,. \label{eq:massGirsanovIto}
\end{align}
Thus the two terms $M_n$ and $V_n$ in \eqref{eq:MnVn} are the discrete analogues of the two terms on the right hand side of \eqref{eq:massGirsanov}. 
The term $A_n$ in Corollary \ref{cor:girsanovpotential} is the direct discrete analogue of the integral in \eqref{eq:massGirsanovIto}.

\subsection{Conformal covariance; loop-erased random walk with drift}\label{s:conformalcovarianceintro}

A fundamental feature of critical models in two-dimensional models of statistical mechanics is that they display conformal invariance. In the near-critical regimes that are under consideration in this paper, we cannot of course expect conformal invariance but rather a change of conformal coordinates rule known as \textbf{conformal covariance} which, roughly speaking, says that the transformation needs to be corrected by suitable powers of the derivative of the conformal map. This has been established in particular in the case of near-critical percolation in the paper \cite{GarbanPeteSchramm_nearcritical} (where this follows from analogous covariance rules for the limit of the uniform measure on pivotal points proved earlier in the remarkable work \cite{GarbanPeteSchramm_pivotal}). To state such a result we need to extend the setup slightly, by allowing the drift vector $\alpha$ to depend continuously on the point $z \in \Omega$.

Thus, let us fix $\alpha: \Omega \to \R^2 \simeq \C$ a locally Lipschitz, bounded \textbf{vector field} (identified with a complex-valued function) on $\Omega$; for $z\in \Omega$, $\alpha (z) \in \R^2$ will represent the drift at position $z$. Our results pertain only to the case where $\alpha$ \textbf{derives from a potential}, i.e. there exists a $\cC^1$ function $\ph: \bar \Omega \to \R$ such that $\alpha = \nabla \ph $.  (We will also make additional assumptions on $\ph$ in the theorem.)

Given such a bounded, continuous vector field, we associate weights on the (scaled) directed triangular lattice $\Omega^\delta$ as follows: 
\begin{equation}
    \label{eq:weights_hex_phi}
    \alpha^\delta_{i} (v) = \ph ( z+ \delta \tau^i) - \ph (z) ; \quad  i = 0,\ldots, 2
\end{equation}
and, as before, these parameters define a Markov chain on $\Omega^\delta$ (which we will refer to later as random walk on $\Omega^\delta$ with drift $\alpha$) given by:
\begin{equation}\label{E:weights3}
\P^{(\ph)}(v,v+ \delta \tau^{i} ) = \frac{e^{\alpha^\delta_{i}(v)}}{a(v)}, \quad i = 0, \ldots, 2, \quad \text{ with } a(v)  = e^{\alpha^\delta_0 (v)}+ \ldots + e^{\alpha^\delta_2(v)}.
\end{equation}

Thus the weights are defined by the gradient of the potential $\ph$, computed locally at each point $z \in \Gtree$. An easy application of the Stroock--Varadhan theorem shows that as $\delta \to 0$, the position of a random walk starting from $o^\delta \to o \in \Omega$, after scaling time by $2\delta^{-2}$, converges to the solution of the Stochastic Differential Equation
$
\dd X_t = \dd B_t + \alpha(X_t) \dd t.
$
Since $\alpha$ derives from a potential, the previous SDE takes the form
\begin{equation}\label{Langevin_intro}
\dd X_t = \dd B_t +\nabla \ph(X_t) \dd t,
\end{equation}
known as a \textbf{Langevin} SDE or diffusion.

We will show that Theorem \ref{T:convergenceTree} can be generalised to this more general setup both for the case of a general drift vector field. The first step is the construction of a scaling limit for the loop-erased random walk with drift (i.e., with weights as above) when the drift vector field derives from a potential satisfying a certain condition; this is the most difficult result of this paper.

\begin{thm}\label{T:LERWgeneral_drift}
Let $\Omega$ be a simply connected domain and $\alpha:\Omega\to\R^2$ be given. Fix $o \in \Omega$ and let $o^\delta \in \Omega^\delta$ such that $o^\delta \to o$ as $\delta \to 0$. Let $a\in \partial \Omega$ and let $a^\delta$ be a sequence of vertices on the boundary of $\Omega^\delta$ such that $a^\delta \to a$.

Suppose the vector field $\alpha$ derives from a smooth potential $\ph : \bar \Omega \to \R $ and suppose also that 
\begin{equation}\label{eq:masscondition}
\rho (x) = \frac12\Delta \ph (x) + \frac12\|\nabla \ph (x) \|^2 \ge 0; x \in \Omega. 
\end{equation}
Let $(X^\delta_t, t= 0,1, \ldots)$ be a random walk on $\Omega^\delta$ with drift $\alpha = \nabla \ph$, i.e., a sample from $\P_{o^\delta}^{(\ph)}$ defined in \eqref{E:weights3}. Let $\sigma^\delta$ denote the first time at which  $X^\delta$ leaves $\Omega$ and consider the
 the loop erasure  $\LE (X^\delta)$ of the walk up until this time. Then conditionally on $X^\delta_{\sigma^\delta} = a^\delta$,
 $\LE (X^\delta)$ converges weakly to a radial Loewner evolution $\gamma$ starting from $\gamma_0 = a$, whose driving function $\zeta_t = e^{ i \xi_t}$ (when parametrised by capacity) satisfies the stochastic differential equation
\begin{equation}\label{E:Loewner_off}
    d\xi_t=\sqrt{2}dB_t+\lambda_tdt, \quad \lambda_t=\frac{\partial}{\partial g_t(a_t)}\log\left(\frac{P^{(\rho)}_{\Omega_t}(o,a_t)}{P_{\Omega_t}(o,a_t)}\right),
\end{equation}
where $a_t=\gamma(t)$, $\Omega_t={\Omega}\setminus\gamma([0,t])$ is the slitted domain at time $t$, $g_t$ is the Loewner map from $\Omega_t$ to ${\DD}$ and $P^{(\rho)}_{\Omega_t}$ is the \textbf{massive invariant Poisson kernel} with (squared) mass profile $\rho = \frac{1}{2}\Delta\ph+\frac12\|\nabla\ph\|^2$, and $P_{\Omega_t}$ is the invariant Poisson kernel of regular Brownian motion in $\Omega_t$.
This convergence is in the sense of uniform convergence after reparametrization.
\end{thm}

The existence and construction of the massive invariant Poisson kernel is nontrivial and is part of the theorem; see Section \ref{S:LERW_general} and in particular Theorem \ref{T:PKDelta}. Furthermore, as before, the drift term in \eqref{E:Loewner_off} has to be understood appropriately, and will really be defined as $Q_t^{(\rho)}(o)/ P_t^{(\rho)} (o)$ where the numerator will also be defined carefully in Section \ref{S:LERW_general}. 

That there is also a unique strong solution to the SDE \eqref{E:Loewner_off} is also not obvious; this will follow from the Novikov criterion and the estimate in Lemma~\ref{lem:novikov}.

\medskip \textbf{Discussion of the assumption \eqref{eq:masscondition}.} The theorem above relies on the condition \eqref{eq:masscondition} which plays a technical but important role. We do not believe this assumption is necessary, but it greatly simplifies the analysis leading to the result. Essentially, our discrete Girsanov theorem allows us to relate random walk with drift to \textbf{random walk with variable mass}. The corresponding limiting (squared) mass function is then given by the formula 
\begin{equation}
\rho(x) : = \frac12\Delta \ph(x) + \frac12\|\nabla \ph(x) \|^2.
\end{equation}
Thus our assumption \eqref{eq:masscondition} amounts to requiring the killing rate to be nonnegative. One can already intuit the emergence of this function from \eqref{eq:massGirsanovIto}.

Although this condition could appear somewhat artificial, we note that this condition is actually \textbf{invariant under conformal transformations}. More precisely, fix $T:\Omega \to \tilde \Omega$ a conformal isomorphism of simply connected domains, and let $X$ be a solution of the Langevin SDE \eqref{Langevin_intro}, where we assume $\frac12\Delta \ph + \frac12\|\nabla\ph \|^2 \ge 0$. Then $T (X_t)$ is, up to a time-change, a solution of the SDE:
\begin{equation}
    dY_t = d\tilde B_t + \nabla \tilde \ph (Y_t) dt.
\end{equation}
This is also a Langevin SDE \eqref{Langevin_intro}, where the new potential $\tilde \ph$ is simply given by 
$$
\tilde \ph (y) = \ph (T^{-1}(y)).
$$
From there it is not hard to see that 
$$
\Delta \tilde \ph (T(x)) = \Delta \ph (x) \cdot |T'(x) |^2; 
$$
(this is best seen by computing the Laplacian via Wirtinger derivatives). Since 
$$\|\nabla \tilde \ph (T(x))\|^2 = |T'(x)|^2 \| \nabla \ph(x)\|^2,$$ 
we deduce that the associated mass function $\tilde \rho $ satisfies 
\begin{equation} \label{eq:masscovariance}
\tilde \rho (y) = |T'(x)|^2 \rho (x); \quad y = T(x). 
\end{equation}
Thus $\rho \ge 0 $ if and only if $\tilde \rho \ge 0.$

\medskip Physically, the assumption \eqref{eq:masscondition} corresponds to a potential that tends to push the diffusion towards the boundary. In particular, if $\ph$ is convex then this condition is satisfied. (Note that our $\ph$ follows an unusual sign convention: the Langevin diffusion is pushed towards higher potential instead of the more commonly adopted convention of lower potentials). In fact, note that $\Delta (e^\ph) = (\Delta \ph + \| \nabla \ph\|^2 ) e^\ph$ so our assumption \eqref{eq:masscondition} is precisely that $e^\ph$ is convex. In other words $\ph $ is a \textbf{log-convex} function. (We warn the reader that with the more usual sign conventions then $\ph$ would be the negative of a log-convex function but not log-concave.)

\begin{rem}
The relation \eqref{eq:masscovariance} is a \textbf{conformal covariance relation for the mass functions}. 
\end{rem}

We now address the consequence of Theorem \ref{T:LERWgeneral_drift} for the dimer model. Let $\alpha =  \nabla \ph$ be a vector field deriving from a smooth potential $\ph:\bar \Omega \to \R$ satisfying \eqref{eq:masscondition}. To the weights $e^{\alpha_{k}(v)}$ in \eqref{E:weights3} we can associate edge weights on $\Gdim$ in a bipartite fashion similar to \eqref{F:double}. The only difference with what was discussed in Section \ref{SS:offcriticalintro} is that now the weights $a_0, a_1, a_2$ depend on the point $v$. We call this the \textbf{inhomogeneous massive dimer model}.

Nevertheless, Temperley's bijection still applies: thus dimer configurations on $\Gdim$ are in (measure-preserving) bijection with wired spanning trees on $\Gtree$. Using results from \cite{BLR_DimersGeometry}, we deduce from Theorem \ref{T:LERWgeneral_drift} that the height function $h_\delta^{(\alpha)}$ of the corresponding dimer model converges to a scaling limit (this generalises Corollary \ref{C:height} to the variable drift setting).
Furthermore, under the additional technical restriction that the conformal map $T: \Omega\to\tilde \Omega $ extends analytically to a neighbourood of $\Omega$, the limit is conformally covariant. 

\begin{thm}
  \label{T:covariance} Fix $\Omega$ and $\alpha = \nabla \ph$ satisfying \eqref{eq:masscondition} as above. The height function $h_\delta^{(\alpha)}$, of the corresponding biperiodic dimer model just described, has a scaling limit which we denote by $h^{(\alpha);\Omega}$.
  This convergence is in the sense of joint convergence of the field integrated  against an arbitrary number of test functions. 
 Furthermore, let $T: \Omega \to \tilde \Omega$ denote a conformal isomorphism of bounded simply connected domains, and suppose that $
 T$ extends analytically to a neighbourhood of $\Omega$.  Then we have the identity in law,
  $$
  h^{(\alpha);\Omega } \circ T^{-1}= h^{(\tilde \alpha); \tilde \Omega}
  $$
  where at a point $w \in \tilde \Omega$,
\begin{equation}
\tilde \alpha (w) = \overline{(T^{-1})'(w)} \cdot \alpha (T^{-1}( w)).
\end{equation}
The product above refers to the multiplication of complex numbers; and this drift vector field $\tilde \alpha$ derives from the potential $\ph \circ T^{-1}$.
\end{thm}


To explain the theorem, we point out that the new drift vector field $\tilde \alpha$ in $\tilde \Omega$
 has an amplitude which, compared to that of $\alpha$ in $\Omega$, has been scaled by 1 over the modulus of the derivative of the conformal map going from $\Omega$ to $\Omega'$, and the vector has been rotated (in the positive direction) by the argument of its derivative. This is the desired conformal covariance rule. Once again, we point out that this formula may be simply understood in terms of conformal covariance of Langevin diffusions. Simply put, the above rule describes (by It\^o's formula and the Cauchy--Riemann equations) the change of coordinates for a Brownian motion with drift $\alpha = \nabla \phi$.

\subsection{Comments and open problems}

\begin{enumerate}

\item The limiting height function $h^{(\alpha); \Omega}$ is determined implicitly from the scaling limit of the associated Temperleyan tree. A natural question would be to identify its law explicitly. For this the \textbf{Coleman correspondence} (see \cite{BauerschmidtWebb} which establishes a rigorous version) is a natural starting point. Briefly speaking, the Coleman correspondence can be viewed as a massive extension of the boson-fermion correspondence, embodied (in the critical case) by the convergence of the dimer (= fermionic) height function to the Gaussian (= bosonic) free field. This suggests that the height function $h^{(\alpha)}$ should be related in the scaling limit to the so-called \textbf{Sine-Gordon model} at the free fermion point, from quantum field theory. The latter is one of the most canonical quantum (yet not conformal) field theories. Despite its non-conformal nature, it enjoys a great deal of integrability. Informally, the sine-Gordon field is defined (in the whole plane) by the law 
\begin{equation}\label{eq:SG1}
\P^{\text{SG}} ( \dd h) \propto \exp \left( z \int_\C \cos ( \sqrt{\beta } {h (x)}) dx \right) \P^{\mathrm{GFF}\#} (\dd h),
\end{equation}
where $\P^{\mathrm{GFF}\#} (\dd h)$ corresponds to the law of a Gaussian free field in $\Omega$ (with Dirichlet boundary conditions) but normalised so that the whole plane Green function satisfies $G^\#_\C(x,y) = - (2\pi)^{-1} \log |x-y| $.

The above expression is however purely formal, as the cosine of (multiples) of the GFF is ill-defined. While this can be made sense of using the theory of imaginary chaos (\cite{JunnilaSaksmanWebb}) for all $\beta< 4\pi$, the free fermion point (corresponding to $\beta = 4\pi$) falls just outside the regime where this theory yields a nontrivial object. 

We conjecture however that the Sine-Gordon field above describes the limit of the dimer height function only in the case of constant mass/drift (and assuming also that the drift vector field points to the right, or that $\Omega$ is the full plane). 
 More generally, we conjecture the following description for the limiting height function for the inhomogeneous massive dimer model (with weights \eqref{E:weights3}), given any vector field deriving from a smooth potential $\ph: \bar \Omega \to \R$. 

 \begin{conj}\label{conj:SG} Let $\P^{(\alpha); \Omega}$ denote the law of the field $h^{(\alpha); \Omega}$ in Theorem \ref{T:covariance}. Then 
\begin{equation}\label{eq:SG2}
\P^{(\alpha); \Omega} (\dd h) \propto \exp \left( z_0 \int_\Omega \left\langle e^{\mathbf{i} {h(x)}/{\chi} } , \alpha(x) \right\rangle dx \right) \P^{\mathrm{GFF}} (\dd h).
\end{equation}
\end{conj}

Again this expression is informal and assigning it a meaning is itself nontrivial. The factor $z_0$ in front of the integral comes from conventions such as the normalisation of the Laplacian and that of the limiting drift $\alpha$.
But note how the expression \eqref{eq:SG2} reduces to that in \eqref{eq:SG1} at the free fermion point when $\Omega$ is replaced by the whole plane. Indeed, first of all the normalisation of the GFF in $\P^{\mathrm{GFF}}$ and $\P^{\mathrm{GFF}\#}$ differ by a factor of $\sqrt{2\pi}$. Thus $h^\# = (2\pi)^{-1/2} h$, so that $\sqrt{\beta} h^\# = (1/\chi) h$ when $\beta = 4\pi$ and $\chi = 1/\sqrt{2}$ is the imaginary geometry constant associated to $\kappa =2$.

Furthermore, when $\Omega = \C$ then by rotational invariance, then $\langle e^{\mathbf{i} {h(x)}/{\chi} } , \alpha(x) \rangle$ has the same law as $\| \alpha (x) \| \cos (\sqrt{2} h(x) ) $. Thus taking $\alpha(x) = \alpha$ to be constant the expression \eqref{eq:SG2} indeed boils down to \eqref{eq:SG1} with $z = z_0\| \alpha\|$. Since $z$ is the mass parameter of the Sine-Gordon model, this is entirely consistent with our Theorem \ref{T:convergenceTree} (and with $z_0 = 1/\sqrt{2}$ in our choice of conventions for the normalisation of the Laplacian).

The above conjecture is informally supported by the imaginary geometry approach to the dimer model (\cite{BLR_DimersGeometry}). Informally, this conjecture says that massive SLE$_2$ is (in some sense) a \textbf{flow line} of the Sine-Gordon field at the free fermion point. We do not know whether this should hold away from the free fermion point, but it is tempting to conjecture so. (Recall that for $\beta<4\pi$ the Sine-Gordon field is absolutely continuous with respect to a GFF so that the notion of flow line is at least well defined).

\item A possible approach to the above (which is also of independent interest in its own right) is the following: can an axiomatic characterisation of this field be given in the manner of \cite{BPR, BPR2, AruPowell}? (This last question is due to Christophe Garban who asked it in a slightly different form.)

\item A separate line of enquiry concerns the possible implications of our results to the study of the Ising model. By bosonization, it is known that the critical Ising model is related to the critical dimer model (\cite{Dubedat_bosonization}). This correspondence remains at least partly valid in the near-critical regime studied here, but we do not know whether the corresponding Ising model is near-critical in the sense of commonly studied perturbations of the critical Ising model (see in particular, \cite{DC_Garban_Pete}, \cite{ChelkakIzyurovMahfouf}, \cite{Park}, \cite{CamiaJiangNewman_review} and references therein).

\item Finally we have developed a near-critical dimer theory on the square and hexagonal lattices using the symmetries of these lattices, but it would be of considerable interest to have a theory in some more general setting, e.g., for double isoradial graphs (i.e., superposition of an isoradial graph and its dual) since we know for instance that the Temperleyan bijection extends to this setting (\cite{KenyonProppWilsonGeneralizedTemperley}).

\end{enumerate}

\emph{Updates}. Since the paper was first put on arXiv we can report on a few developments in the direction of the above conjectures. 

\begin{enumerate}
    \item
On the one hand, Mason \cite{Mason} showed that in the full plane and in the case of constant drift, the two-point correlation of the limiting massive dimer height function coincides with that of the free fermion Sine-Gordon field.

\item Separately, Papon considered the case $\kappa = 4$ of Makarov and Smirnov's programme (\cite{papon2023massiveHarmonic, papon2023massiveCLE}). Roughly speaking, she shows in these articles the following results: convergence of the massive harmonic explorer to massive SLE$_4$, and a conformal invariance property analogous to Theorem~\ref{T:covariance}. Furthermore, level lines of the massive GFF are given by massive variant of CLE$_4$, and the occupation field of a massive Brownian loop soup coincides with the square of the massive GFF. The pairwise relations between these three objects hold simultaneously, as in the work of Qian and Werner for the non-massive case (\cite{qianWerner2019decomposition}). 

\item Finally, Rey \cite{rey2024} has developed a Girsanov identity for isoradial graphs and applied it to massive dimer models, thereby generalising the results of this paper. 

\end{enumerate}


While it is not the purpose of this paper to give an extensive overview of recent works on near-critical models, we feel it is appropriate to conclude this introduction by mentioning some which are at least in spirit motivated by similar questions albeit for different models.  These include, beyond the already mentioned works on near-critical percolation \cite{GarbanPeteSchramm_nearcritical} and the near-critical Ising model \cite{ChelkakIzyurovMahfouf}, \cite{Park}, \cite{CamiaJiangNewman_review} and \cite{DC_Garban_Pete}, the work of Duminil-Copin and Manolescu on scaling relations in the random cluster model \cite{DC_Manolescu_scaling}), the work of Benoist, Dumaz and Werner \cite{BenoistDumazWerner} on near-critical spanning forests, and Camia's work on off-critical Brownian loop soup \cite{Camia_off}.

\subsection{Notation and Scaling}\label{sec:notation}
The triangular lattice $\mathbb T$ always refers to the \emph{directed} triangular lattice, in which each edge has been directed in the respective direction $1, \tau$ or $\tau^2$, i.e. when we speak of a random walk on this lattice only steps in those three directions are allowed.

We recall that $
\Omega^\delta$ is the graph on which all random walk paths will leave (a scaled copy of the triangular lattice or the square lattice, approximating $\Omega$), which is often identified with its vertex set.  In this paper several measures on such lattice paths appear.
For the convenience of the reader we collect the most important ones here.

\begin{itemize}
    \item The simple random walk measure $\P^{(0)}$, which takes all possible steps with equal probability. Note that if $X^\delta$ has law $\P^{(0)}$ then $(X_{2\delta^{-2}t})_{t\ge 0}$ converges to a standard planar Brownian motion (this holds both on $\delta \T$ and on $\delta \Z^2$). 
    
    \item Given a function $\rho^\delta: \Omega^\delta \to [0,1]$, the massive random walk measure $\P^{(\rho^\delta)}$ is the massive random walk, dying at each step with probability $\rho^\delta(v)$ if it is in position $v \in \Gtree$ (and otherwise jumping to one of its neighbours with equal probability).
    
    \item Consider the triangular lattice case. Given $\alpha^\delta: \Gtree \to \R^2$ a discrete vector field, the random walk with variable drift $\P^{(\alpha^\delta)}$ takes steps according to \eqref{eq:driftweights}, i.e., the walk jumps from $v$ to $v+ \delta \tau^k$ with probability proportional to $e^{\alpha_k(v)}$ $(k=0, 1, 2)$, where $\alpha_0(v), \alpha_1(v), \alpha_2(v)\in \R$ are uniquely defined by the requirements $\sum_{k=0}^2 \alpha_k(v) = 0$ and $\alpha(v) =\tfrac23 \sum_{k=0}^2 \alpha_k(v)  \tau^k$.

\end{itemize}

  Here we have defined the law of random walk with drift $\alpha^\delta$, $\P^{(\alpha^\delta)}$, only in the case of the triangular lattice. Obviously, an analogous definition can be given in the case of the square lattice too; this will be made explicit in Section \ref{S:square} when it is needed.
  
  The weights $\alpha^\delta(v)=\alpha^\delta(v)$ typically depend on both $\delta$ and $v$. When there is no risk of ambiguity we will sometimes write $\P^{(\alpha)}$ for $\P^{(\alpha^\delta)}$, and likewise we will write $\P^{(\rho)}$ for $\P^{(\rho^\delta)}$ if there is no risk of confusion. 
  
  Usually we view them as measures on the canonical path space, whose corresponding random variable is denoted by $X^\delta_t,t=0,1,\dots$. Sometimes however, given a discrete path $\gamma^\delta=(x_0,\dots,x_n)$ we write $\P(\gamma^\delta)$ for $\P ( (X^\delta_s)_{s= 0, \ldots, n} = (\gamma^\delta_s)_{s= 0, \ldots, n})$, where $\P$ is any of the laws above.

Note that $\P^{(0)}$ is a special case of all of these measures, setting the respective parameters to $0$.

\medskip For the appropriate weights these random walks have scaling limits.
For instance, if $\alpha^\delta(v)=\delta F(v) +o(\delta)$ for some bounded Lipschitz-continuous $F$, then this random walk converges to the solution of the SDE
    \[
    \dd X_t=F(X_t)\dd t+ \dd B_t\,.
    \]
let $\ph:\Omega \to \R$ be a smooth function. Noting the fact that $\nabla^{\delta\mathbb T}\ph(v)=\delta \nabla\ph(v)+o(\delta)$ (recall our conventions for the discrete gradient in \eqref{eq:discretegradient}), this implies that if $\alpha^\delta(v) = \nabla^{\delta\mathbb{T}}( \ph)$, then
the random walk corresponding to $\P^{\alpha^\delta}$ converges in the scaling to the Langevin diffusion
    \[
    \dd X_t=\nabla\ph(X_t)\dd t+ \dd B_t\,.
    \]
That is, $(X^\delta_{2\delta^{-2} t })_{t\ge 0}$ converges weakly under $\P^{(\alpha^\delta)}$ to the above Langevin diffusion.
    
Likewise, in the massive case,  suppose that $\rho^\delta(v)={\delta^2}\rho(v)/2+o(\delta^2)$.
Then the random walk corresponding to $\P^{(\rho^\delta)}$,  converges (under the same scaling) to massive Brownian motion with profile $\rho$, i.e. its law converges to the measure $\P_x^{(\rho)}$ whose Radon--Nikodym derivative with respect to Brownian motion is given by
\[
\left.\frac{\dd \P^{(\rho)}_x}{\dd \P_x}\right|_t=\exp\left(-\int_0^t\rho(X_s)\dd s\right)\,.
\]

\paragraph{Organisation of the paper.} In Section \ref{section:DriftMass} we state and prove the discrete Girsanov identities and explain the implication for the connection between drifted and massive walks which lies at the heart of this paper. In Section \ref{section:convergence} we extend Chelkak and Wan's result about the convergence of the massive LERW to massive SLE$_2$ to the directed triangular case; the additional difficulty compared to their setup is the lack of reversibility. At this stage Theorem \ref{T:convergenceTree} are proved. 

In Sections \ref{S:LERW_general} and \ref{sec:drift} we show how to get the existence of a scaling limit for loop-erased random walk on graphs where the drift is a variable function of the vertices given by the gradient of a potential (in particular, the scaling limit of the random walk is given by a Langevin diffusion).
Finally in Section \ref{sec:drift} we transfer results about convergence of trees to convergence of height function (which implies in particular the conformal covariance of Theorem \ref{T:covariance}).

\paragraph{Acknowledgements.} Both authors thank Dima Chelkak and Marcin Lis for a number of illuminating conversations, as well as Christophe Garban for a number of comments on a draft. Fredrik Viklund made a number of useful suggestions which among other things helped improve the presentation. Research of the first author is supported by FWF grants 10.55776/P33083, ``Scaling limits in random conformal geometry'' and SFB grant 10.55776/F1002, ``Discrete Random Structures: enumeration and scaling limits''.

Part of the paper was written while the first author was in residence at the Mathematical Sciences Research Institute in Berkeley, California, during
the Spring 2022 semester on \emph{Analysis and Geometry of Random Spaces}, which was supported by the National Science Foundation under Grant No. DMS-1928930.

\section{Girsanov identity; proof of Theorem  \ref{T:convergenceTree}}\label{section:DriftMass}

In this section we start with a proof of Theorem \ref{T:convergenceTree}, which we prove separately in the case of the square and hexagonal lattices. As mentioned the result will follow from applying a form of Temperley's bijection and studying the scaling limit of the corresponding loop-erased random walk (which describe branches in the spanning tree by Wilson's algorithm). Since Temperley's bijection is not so well known in the case where $\Gdim$ is a subgraph of the hexagonal lattice, we start by explaining the bijection in this case, which can also be found (albeit somewhat informally) in Section 2 of \cite{KenyonProppWilsonGeneralizedTemperley}; see in particular their Figure 2.

\subsection{Temperley's bijection on the hexagonal lattice}\label{SS:Tbij_hex}

As mentioned in Section \ref{sec:Temperleybijintro}, Temperley's bijection relates rooted spanning trees on a graph $\Gtree$ (already discussed in Section \ref{sec:temperleyan}) to dimers on the graph $\Gdim$, a Temperleyan subgraph of the hexagonal lattice. We start by describing how $\Gtree$ and $\Gdim$ are related to one another. The bijection itself will be stated in Theorem \ref{T:Tbij} and is illustrated in Figure \ref{fig:pretty}.

\begin{figure}
    \centering

    \includegraphics[height=6cm]{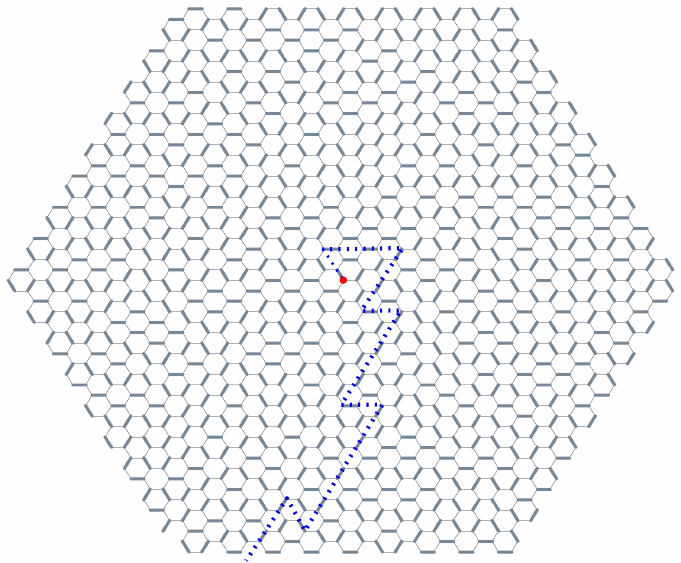}
    \includegraphics[height=6cm]{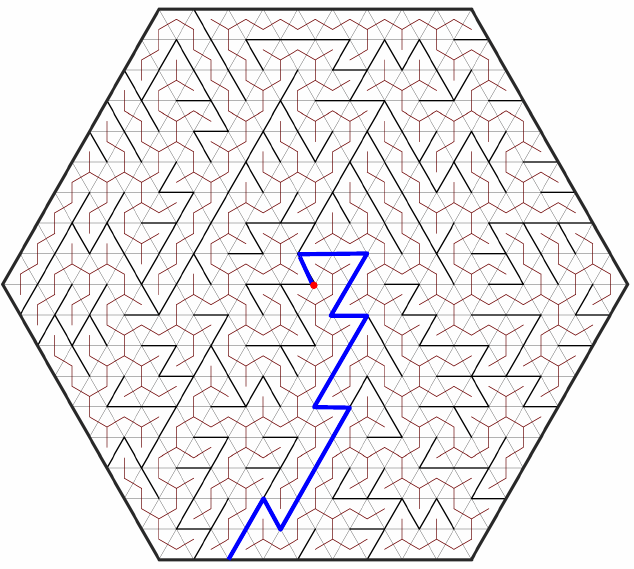}
    \caption{A dimer configuration on the hexagonal lattice, and its associated
   pair of dual spanning trees from Tempereley's bijection. The outer (i.e., boundary) vertex is represented as a black hexagon for convenience. The unique path connecting a vertex $v$ (at the centre of the hexagon) to the boundary has been highlighted on the tree; the corresponding path on the dimer graph appears as a dotted line. Each dimer on this path can be viewed as the first half of the corresponding tree edge. Conversely, we obtain the tree by multiplying by two each dimer emanating from a $B_1$ vertex, in the direction from black to white.}
    \label{fig:pretty}
\end{figure}

Consider the triangular lattice, that is the graph whose vertices are given by $a+b\tau$, where $a, b \in \delta \Z$ are integers (times $\delta$) and $\tau = e^{ 2 {\mathbf i} \pi /3}$ is the third root of unity, and where each pair of vertices at distance $\delta$ is connected by an edge.
We will give each edge an orientation, such that it is oriented in direction $1$, $\tau$ or $\tau^2$ and a weight, which is $a_1,a_2$ or $a_3$ accordingly.
This gives a directed graph in which each vertex has three outgoing and three incoming edges.
We will call this graph the directed triangular lattice and denote it by $\mathbb T$ throughout this article.

Let us now choose a simply connected set of vertices of $\mathbb{T}$ and identify all other vertices as a single \textbf{outer vertex}. We call the resulting graph $\Gtree$.
A \textbf{spanning tree of $\Gtree$ rooted at the outer vertex} is a spanning set of edges containing no cycle. By orienting the edges to wards the root of the tree (the outer vertex), any such tree is equivalent to a spanning \emph{arborescence}, i.e., a collection of directed edges such that there is exactly one outgoing edge from each non-root vertex (and none at the root), and having no cycle (irrespective of the orientation). This point of view is useful in Temperley's bijection since edges come with a natural orientation.

(Sometimes such a tree is called an arborescence).
By definition we assign a weight to a rooted spanning tree given by the product of the weights of the edges in the tree.

Now consider the superposition graph $H^*$ obtained in the following way.
The vertices of $H^*$ are the vertices, edges and faces of $\Gtree$.
To avoid terminological confusion, call the vertices of $H^*$ nodes and call them vertex-nodes, edge-nodes and face-nodes depending on their counterpart in $\Gtree$.
The edges of $H^*$ are called links and are defined as follows: connect a vertex-node $v$ and an edge-node $e$ if $e$ is an outgoing edge of $v$ in $\Gtree$ and give this link the same weight as $e$ in $\Gtree$.
Also connect an edge-node $e$ and a face-node $f$ if $e$ is adjacent to $f$ in $\Gtree$, and assign weight $1$ to such links.
Finally obtain $\Gdim$ from $H^*$ by deleting the vertex-node corresponding to the outer vertex and one face-node for a face adjacent to the the outer vertex. Note that the vertex nodes of $H^*$ are the $B_1$ vertices of the hexagonal lattice, while edge-nodes are white. (The face nodes of $H^*$ are either of type $B_2$ or $B_3$.) 
For an illustration of this procedure see Figure \ref{fig:pretty}.

The graph $\Gdim$ obtained this way is exactly a Temperleyan domain of the hexagonal lattice as defined in Section \ref{sec:temperleyan}, and by choosing $\Gtree$ as the directed triangular lattice formed by the $B_1$ vertices in such a domain, it is also clear that each Temperleyan subgraph of the hexagonal lattice can be obtained in this way.
The weights on this graph are as in Figure \ref{F:double}. The relevant version of Temperley's bijection is then the following:
\begin{thm}[\cite{KenyonProppWilsonGeneralizedTemperley}]\label{T:Tbij}
There is a weight preserving bijection between spanning trees of $\Gtree$ rooted at the outer vertex (i.e., spanning arborescences rooted at the outer vertex) and dimer configuration on $\Gdim.$
\end{thm}

The bijection is easier to describe in the direction ``dimers'' to ``trees'': given a dimer configuration $\mathbf{m} $ on $\Gdim$, define a collection $\cT$ of oriented edges in $\Gtree$ as follows: for every dimer occupying a link between a vertex-node $v \in v(\Gtree)$ and an edge node $e \in E( \Gtree)$, include the outgoing edge $e$ from $v$ to $\cT$. One can check that the resulting collection of edges $\cT$ is a spanning tree in the sense above. (Essentially, to every vertex $v \in v( \Gtree)$ there is a unique outgoing edge containing $v$ in $\cT$ by definition of the dimer model and of $\cT$; following the outgoing edges from a given vertex $v \in v(\Gtree)$ may not result in a cycle by duality considerations, and thus necessarily ends at the outer vertex -- this is the unique path to the outer vertex in the definition).   Once again, we refer to Figure \ref{fig:pretty} for illustration.

\subsection{Proof of Theorem \ref{L:girsanov_discrete_tri} and relation to massive walk}

We consider first the case of the triangular lattice and give the proof of Theorem \ref{L:girsanov_discrete_tri}, and recall that here we work  on the unscaled lattice $\mathbb{T}$ rather than the scaled lattice $\delta \mathbb{T}$.

\begin{proof}[Proof of Theorem \ref{L:girsanov_discrete_tri}]
Let $n_0 = n_0(v),n_1 = n_1(v)$ and $n_2 = n_2(v)$ be the number of steps taken by $\gamma$ from $v$ in the directions $1,\tau$ and $\tau^2$ respectively. Then

\begin{align*}
    {\mathbb P_x^{(\alpha)}(\gamma)} 
    &=  \prod_{v \in v(\Omega)} \prod_{k=0}^2(\frac{e^{\alpha_k}}{a})^{n_k}
    = 3^{-n} \prod_{v\in v(\Omega)}  \left[ ( (a/3)^{-(n_0+n_1+n_2)}\prod_{k=0}^2 (e^{\alpha_k})^{\frac{n_0+n_1+n_2}{3}}\prod_{k=0}^2(e^{\alpha_k})^{n_k-\frac{n_0+n_1+n_2}{3}} \right]\\
     & = 3^{-n}  \prod_{v \in v(\Omega)} e^{ - \beta(v)^2 \tfrac{n_0+n_1+n_2}{3}} \exp \left(  \sum_{k=0}^2 \alpha_k (n_k - \tfrac{n_0 + n_1 + n_2}{3} )\right)\\
     &= 3^{-n} e^{-\tfrac12 V_n } \exp \left(  \sum_v  \alpha_0 ( \tfrac{2n_0 - n_1 - n_2}{3} + \alpha_1 ( \tfrac{2n_1 - n_0 - n_2}{3} ) + \alpha_2 ( \tfrac{2 n_2 - n_0 - n_1}{3} )\right)\\
    &=  3^{-n} e^{-\tfrac12 V_n } \exp \left( \tfrac{2}{3}  \sum_v \langle \alpha_0 + \alpha_1 \tau + \alpha_2 \tau^2 , n_0 + n_1\tau + n_2\tau^2 \rangle\right)
\end{align*}
where we have used in the last line that $\langle 1 ,\tau \rangle = \langle 1 ,\tau^2 \rangle = \langle \tau, \tau^2 \rangle = - 1/2$. To conclude, simply observe that each $\dd x_s$ contributes exactly $1$, $\tau$ or $\tau^2$ exactly $n_0$, $n_1$ or $n_2$ times respectively. Therefore,
\begin{equation}\label{e:n_to_x}
\frac23  \sum_v \langle \alpha_0 + \alpha_1 \tau + \alpha_2 \tau^2 , n_0 + n_1\tau + n_2\tau^2 \rangle = \sum_{s=0}^{n-1} \langle \alpha(x_s), \dd x_s\rangle = M_n,
\end{equation}
so that
$$
\frac{\mathbb P_x^{(\alpha)}(\gamma)} {\mathbb P_x^{(0)}(\gamma)} = \exp ( M_n - \tfrac12 V_n),
$$
as desired.
\end{proof}

Before we proceed with the case of constant drift let us first consider the case of drift of gradient type.
\begin{proof}[Proof of Corollary~\ref{cor:girsanovpotential}]
    Consider a single summand of $M_n$.
    Let $j$ be such that $x_{s+1}-x_s=\tau^j$
    \begin{align}
        &\langle\nabla^\mathbb T\Phi(x_s), x_{s+1}-x_s\rangle=
        \frac23\sum_{i=0}^2(\Phi(x_s+\tau^i)-\Phi(x_s))\langle\tau^i,\tau^j\rangle=\\
        &\frac23\left(\Phi(x_s+\tau^j)-\Phi(x_s)
        -\frac12(\Phi(x_s+\tau^{j+1})-\Phi(x_s))
        -\frac12(\Phi(x_s+\tau^{j+2})-\Phi(x_s))\right)=\\
        &\Phi(x_{s+1})-\Phi(x_s)-\Delta^\mathbb T\Phi(x_s).
    \end{align}
    Telescoping the first term gives the desired result.
\end{proof}

\begin{rem}
    Since Theorem \ref{L:girsanov_discrete_tri} and Corollary \ref{cor:girsanovpotential} are just statements about the random walk on the triangular lattice, i.e. independent of the embedding of this graph, we chose to state it for the unscaled lattice $\T$.
    However, for convenience let us describe what these results become when we scale the triangular lattice, as this will be the situation of interest in the rest of the article.
    Thus, let us assume that we are given $\alpha^\delta:\delta\T\to\R$, such that $\alpha^\delta(v)=\delta\alpha(v)+o(\delta)^2$. Let $\alpha_0^\delta,\alpha_1^\delta,\alpha_2^\delta$ be  associated weights such that $\alpha^\delta=\frac23\sum_{k=0}^2\alpha_k^\delta\tau^k$ (these are defined only up to a common additive constant, as $1+ \tau + \tau^2 = 0$).

    Then, as will be checked in Lemma \ref{L:massdrift_tri}, the corresponding factor $\beta(v)$ (which does not depend on the choice of the above constant) will be of order $\delta^2$.
    The statement of Theorem \ref{L:girsanov_discrete_tri} remains unchanged except that one has
    \begin{equation}\label{eq:Mnscaled}
    M_n=\sum_{s=0}^{n-1}\langle \delta^{-1}\alpha^\delta(x_s),\dd x_s\rangle\,.
    \end{equation}
    The additional factor $\delta^{-1}$ compared to  \eqref{e:n_to_x}, comes from the scaling of the triangular lattice: in \eqref{e:n_to_x} we had used that $\dd x_s\in\{1,\tau,\tau^2\}$,
    but on the scaled triangular lattice one has instead $\tfrac{\dd x_s}{\delta} \in \{1, \tau, \tau^2\}$, so we need to add a factor of $\delta^{-1}$ to compensate.
    The fact that the terms in the sum defining $M_n$ in \eqref{eq:Mnscaled} are each of order $\delta$, while the summands in the sum defining $V_n$ are of order $\delta^2$ is consistent with the fact that $M_n$ converges to a stochastic integral and $V_n$ converges to a finite variation integral (with $n$ of order $\delta^{-2}$ in both cases).

    A similar remark applies to Corollary \ref{cor:girsanovpotential} when we scale the triangular lattice. The assumption $\alpha(v)=\nabla^\T\Phi(v)$ becomes
    \[
    \alpha^\delta(v)=\nabla^{\delta\T}\ph(v):=\frac{2}{3}\sum_{s=0}^{2} (\ph(v+\tau^s)-\ph(v))\tau^s=\delta\nabla\ph(v)+o(\delta^2)\,,
    \]
    for $v\in\delta\T$, and smooth $\ph:\Omega\to\R$.
    Using
    \[
    \Delta^{\delta\T}\ph(v):=\frac13\sum_{s=0}^2\ph(v+\tau^s)-\ph(v)=\frac{\delta^2}{4} \Delta \ph(v) +o(\delta^2)
    \]
    in place of $\Delta^\T$ the statement remains unchanged.
\end{rem}

\subsection{Statement of the theorem about LERW}

We may now state the theorem needed for the proof of Theorem \ref{T:convergenceTree}. Let $\Gdim$ be as in Theorem \ref{T:convergenceTree} and let $\Gtree$ denote the embedded graph on which the tree obtained from the Temperleyan bijection lives; thus $\Gtree$ is either a portion of the scaled square lattice or of the (directed) triangular lattice, and is embedded within the domain $\Omega$. With an abuse of notation, we often identify the vertex set $v(\Gtree)$ of $\Gtree$ with $\Gtree$ itself. Consider the random walk on $\Omega^\delta$ arising from the weights \eqref{eq:weights_hex} (resp. \eqref{eq:weights_square}). Observe that in either case, the corresponding law is of the form $\P^{(\alpha^\delta)}$ with $\alpha^\delta = \delta \alpha + o(\delta)$ does not depend on $v$, where $\alpha$ is as in \eqref{E:alpha_intro}. For instance, in the case of the triangular lattice, we define $\alpha_k^\delta $ ($0\le k \le 2$) by
\begin{equation}\label{E:trialpha}
\exp(\alpha^\delta_k)=a_k=1+c_k\delta
\end{equation}
so  $\alpha^\delta=\tfrac23\sum_{k=0}^2\alpha_k^\delta\tau^k =  \tfrac{2}{3}\delta \sum_{k=0}^2 c_k \tau^k + o(\delta) = \delta \alpha + o(\delta)$.

\begin{figure}
    \centering
    \includegraphics[width=0.4\textwidth]{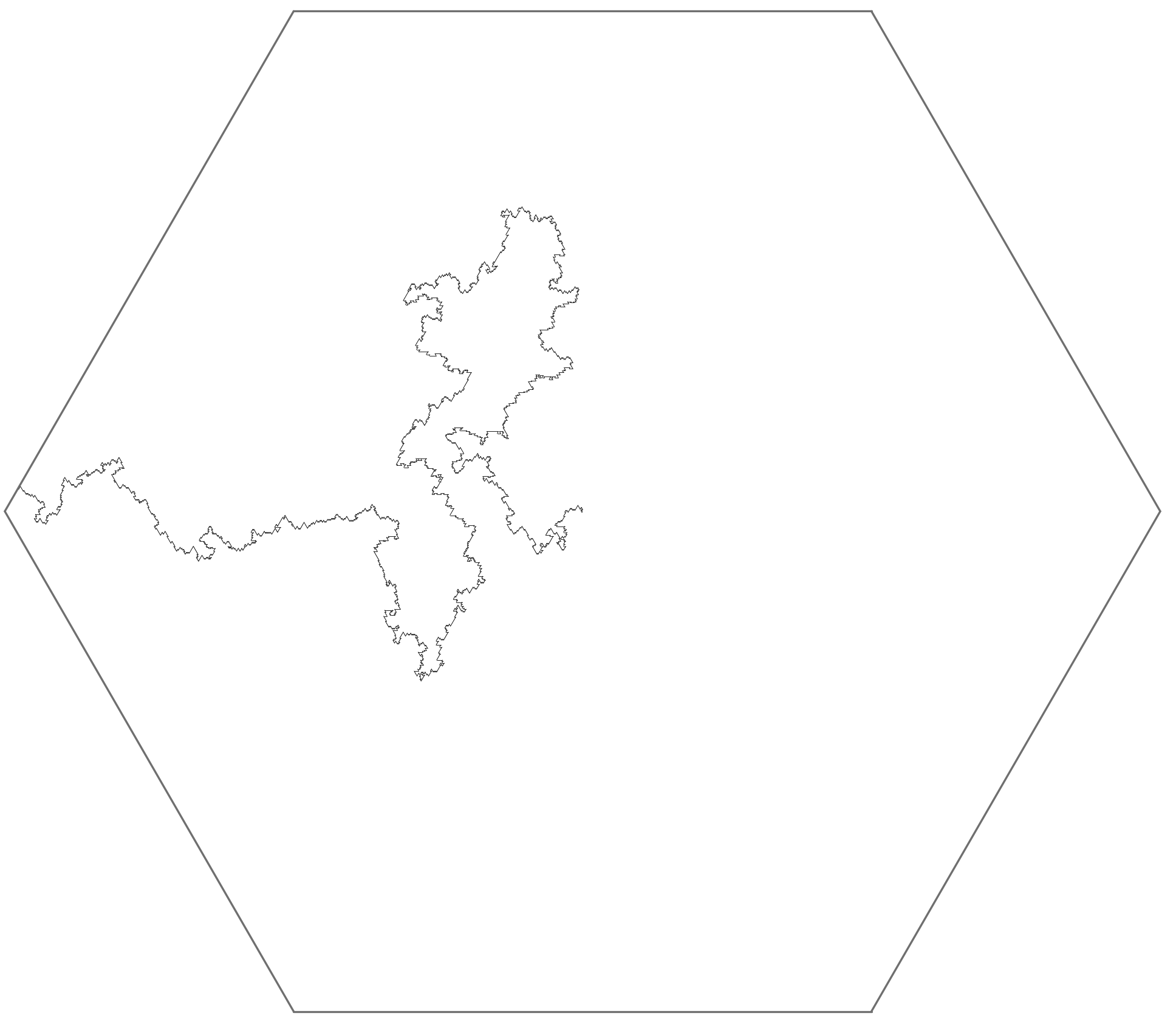}
    \includegraphics[width=0.4\textwidth]{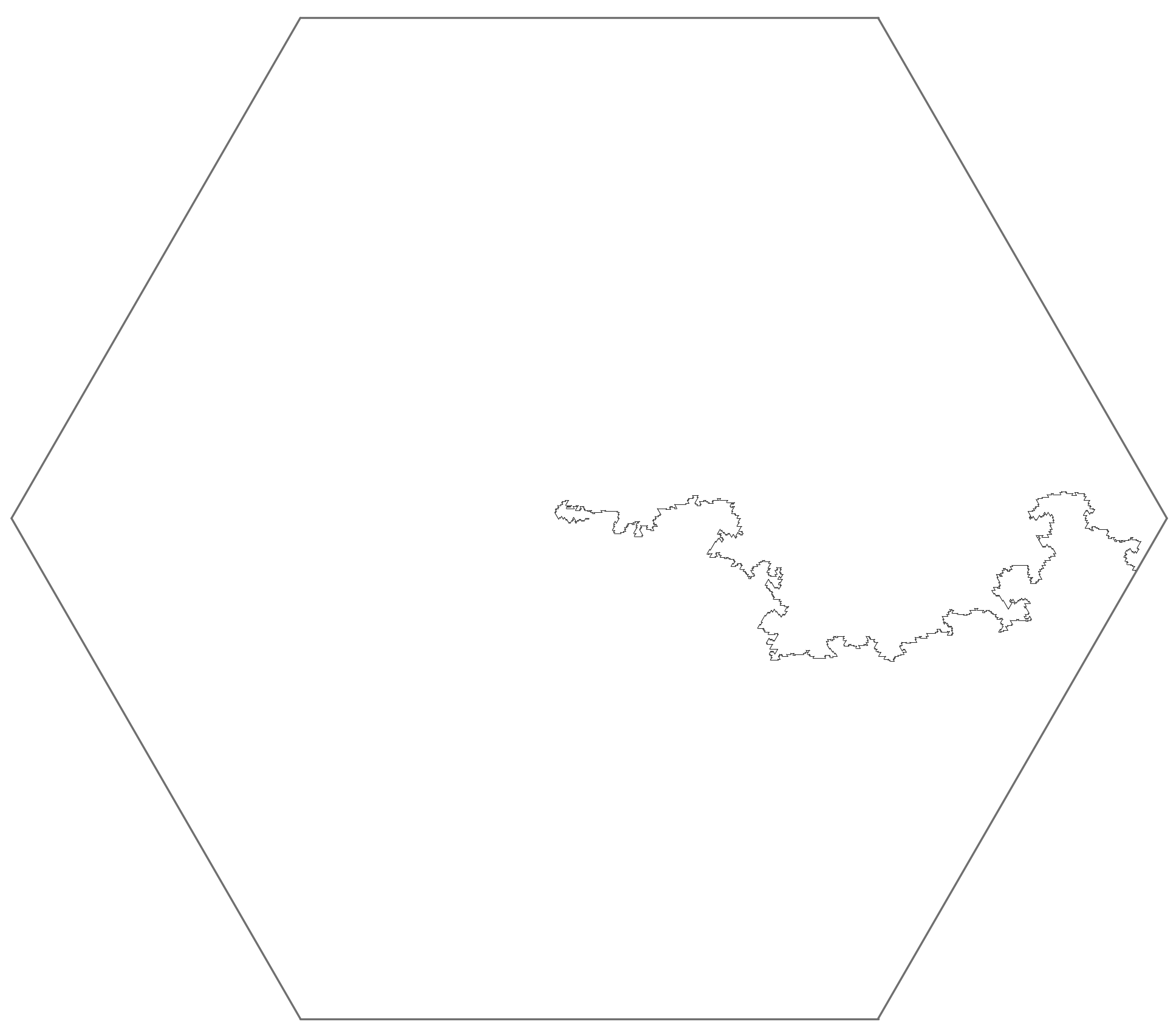}
    \caption{Two samples of loop-erased random walks on the triangular lattice in a hexagon of side-length 500.
    Left: no drift. Right: small drift to the right.}
    \label{fig:walks}
\end{figure}

\begin{thm}\label{T:LERW_constant} Suppose $\Omega$ is bounded and suppose $\alpha^\delta: \Omega \to \R^2$ is independent of $v$ and satisfies $\alpha^\delta = \delta \alpha + o(\delta)$ for some fixed $\alpha \in \R^2$.
Let $o \in \Omega$ and let $o^\delta$ denote a lattice point on $\Gtree$ which converges to $o$ as $\delta \to 0$. Let $(\gamma^\delta_0, \ldots, \gamma^\delta_T)$ denote the loop-erasure of a random walk sampled from $\P^{(\alpha^\delta)}_{o^\delta}$, starting from $o^\delta$ killed when leaving $\Gtree$, and identify $\gamma^\delta$ with its linear interpolation to get a continuous path on $[0,T]$. Then as $\delta \to 0$,
$$
\gamma^\delta \to \gamma^0,
$$
where $\gamma^0$ has the following law: first, its endpoint $a$ has the law $\mu^{(\alpha)}_o$ which is the hitting distribution of $\partial \Omega$ by a Brownian motion with drift $\alpha$ starting from $o$; furthermore, conditionally given $a$, $\gamma^0$ is a massive radial SLE$_2$ from $a$ to $o$ in $\Omega$ with mass $\|\alpha\|/\sqrt{2}$. Here the convergence is in the sense of uniform convergence up to reparametrisation.
\end{thm}

Note that Theorem \ref{T:constantdrift_tree} follows directly from Theorem \ref{T:LERW_constant} and Temperley's bijection (Theorem \ref{T:Tbij}). The rest of Section \ref{section:DriftMass} will be devoted to a proof of Theorem \ref{T:LERW_constant}. We will separate the case of the square and triangular lattices as the proofs are a little different in each case. We first outline the main ideas. Essentially, we are able to relate at the discrete level the loop-erasure of random walk on $\Gtree$ with that of a massive random walk. The relation is exact in the case of the triangular lattice and approximate in the case of the square lattice. On the square lattice, we know by the results of Makarov and Smirnov \cite{MakarovSmirnov} (as clarified by the more recent work of Chelkak and Wan \cite{ChelkakWan}) that the massive LERW converges to massive SLE$_2$. Combined with the above-mentioned approximate relation on the square lattice, this gives a proof of Theorem \ref{T:LERW_constant} in this case.
The theorem of Chelkak and Wan is however only stated for the square lattice and we will verify that their approach can be extended to cover the directed triangular lattice as well. The lack of reversibility is a difficulty in that case.

Let us now begin the proof of Theorem \ref{T:LERW_constant} for the triangular lattice, with a proof of the fact that  the loop-erased random walk has the same law as the loop-erasure of a massive walk, once we condition on the endpoint.


Fix $\alpha^\delta$ as in the theorem, and write $a_k =a_k^\delta = e^{\alpha_k^\delta}$ ($k = 0, \ldots, 2$), and $a = a_0 + a_1 + a_2$. Let $\beta(v)=\beta^\delta(v)$ be implicitly defined by \eqref{E:tribeta}, which as we will soon see is of order $\delta^2$, and clearly does not depend on $v$.
We will want to compare our walk
$\P^{(\alpha)}$ with an appropriate massive walk. Let $m = m^\delta >0$ be defined by
\begin{equation}\label{eq:m}
 \frac{1}{3} (1- \frac{m^2 \delta^2}{2}) =  \frac{\sqrt[3]{a_0 a_1  a_2 }}{a }.
 \end{equation}
(Note that $m$ is well defined by the arithmetic-geometric mean inequality.) 
The mass $m$ can also be related to the factor $\beta^2$ previously introduced in \eqref{E:tribeta}: that is,
$$
1-\frac{m^2\delta^2}{2} = \exp ( - \beta^2  /3).
$$


We now show that the mass $m = m^\delta$ is non degenerate in the limit, and in fact simply equals the norm of the drift vector $ \alpha$ (up to a factor 1/2).

\begin{lem}\label{L:massdrift_tri}
Let $\alpha^\delta$ be as in Theorem \ref{T:LERW_constant} and $\beta $ be as above. Then 
\begin{equation}\label{eq:beta_asymp}
\quad\beta^2(v)=\frac34\delta^2\|\alpha\|^2+o(\delta^2),
\end{equation}
Equivalently, if $m = m^\delta$ be as in \eqref{eq:m} then
$m^\delta$ converges as $\delta \to 0$ to $\frac{\|\alpha\|}{2}$.
\end{lem}
\begin{proof}
In fact we will directly prove the result on $m$. This will come from a careful second order expansion (note however that our assumption about $\alpha^\delta$ implies only $\alpha^\delta = \delta \alpha + o(\delta)$). For $k = 0,1,2,$ let us write $a_k = e^{\alpha_k^\delta}= 1+ c_k \delta$, and let $s = c_0 + c_1 + c_2$, so that with these notations $a = a_0+ a_1 + a_2 = 3 + s \delta$.
Then starting from the identity
$$
\frac{\sqrt[3]{(1+ c_0\delta)(1+ c_1 \delta) (1+ c_2\delta)}}{ 3 + s \delta} = \frac13( 1-\frac{m^2\delta^2}{2}),
$$
and expanding the product before doing a Taylor expansion of the left hand side as $\delta \to 0$, we find
$$
\frac{1 + \tfrac{s\delta}3 + (\tfrac{c_0 c_1 + c_1 c_2 + c_2 c_0}{3} - \tfrac{1}{9}s^2 )\delta^2 + o(\delta^2) }{3 + {s \delta} } = \frac13( 1-\frac{m^2\delta^2}{2})
$$
in other words, writing $\kappa = \tfrac{c_0 c_1 + c_1 c_2 + c_2 c_0}{3} - \tfrac{1}{9}s^2$,
$$
\frac13 + \frac{\kappa}{3} \delta^2 + o(\delta^2) = \frac13( 1-\frac{m^2\delta^2}{2})
$$
from which it follows that
$$
m^2 =- 2\kappa+ o(1).
$$
Let us call $p = c_0 c_1 + c_1 c_2 + c_2 c_0$, so that
\begin{align*}
\kappa &= \frac{p}{3} - \frac{s^2}{9}  = \frac{p}{3} - \frac{1}{9}(\sum_{k=0}^2 c_k^2 + 2p) = \frac{p}{9} - \frac{1}{9} \sum_{k=0}^2 c_k^2
\end{align*}
Now observe that since $\la1, \tau\ra  = \la \tau, \tau^2 \ra = \la 1, \tau^2\ra = -1/2$,
\begin{align*}
\|\alpha\|^2 & =  \frac{4}{9}\la c_0 + c_1 \tau + c_2 \tau^2, c_0 + c_1 \tau + c_2 \tau^2\ra \\
& =-\frac{4}{9}p + \frac{4}{9}\sum_{k=0}^2 c_k^2 = -\kappa.
\end{align*}
Therefore,
$$
m^2 =  \frac{1}{2}\|\alpha\|^2 + o(1),
$$
as desired.
\end{proof}

Let $\rho = \rho^\delta = m^2 \delta^2 /2$, and let $\P^{(\rho)} = \P^{(\rho^\delta)}$ denote the law of massive random walk, which has jump probabilities
$$
\P^{(\rho)}( v, v+ \delta \tau^k ) = \frac{1}{3}(1-\frac{m^2\delta^2}{2}) = \frac{\sqrt[3]{a_0 a_1 a_2}}{a}; \quad k = 0, \ldots, 2
$$
and which jumps to an additional ghost or cemetery vertex with probability $\rho^\delta = m^2 \delta^2/2$ (in which case say that the path has died). Let $\P^{(\rho)} ( \cdot | Y^\delta = a^\delta)$ denote the conditional law of massive random walk, given that the walk does not die before leaving $\Gtree$ and that the exit point is $a^\delta$.

From Theorem \ref{L:girsanov_discrete_tri} we get the following corollary:

\begin{cor}\label{conditionedequality}
For each $\delta>0$, for each $o^\delta \in \Gtree$ and  $a^\delta \in \partial \Gtree,$ we have
$$
\P^{(\alpha^\delta)}_{o^\delta} (\cdot | Y_\delta =a_\delta) = \P^{(\rho^\delta)}_{o^\delta}(\cdot | Y_\delta = a_\delta).
$$
\end{cor}

\begin{proof}
Since $\alpha^\delta$ does not depend on $v$, the discrete stochastic integral can be written as
\begin{equation}\label{eq:martingalepart}
M_n = \sum_{s=0}^{n-1} \langle \delta^{-1}\alpha^\delta (x_s) ,\dd x_s \rangle =  \langle \delta^{-1}\alpha^\delta, a_\delta - o_\delta\rangle
\end{equation}
and so does not depend on the path $\gamma^\delta$ subject to the condition $Y^\delta = a^\delta$. Furthermore the mass $m$ has been chosen so that the quadratic variation part cancels the mass term exactly: that is,
$$
\exp( - \tfrac12 V_n ) = (1-\frac{m^2\delta^2}{2})^n.
$$
Hence the ratio of the left hand side to the right hand side is a constant, independent of the path $\gamma^\delta$, therefore this constant is one since both probability measures sum up to one when we sum over all paths.
\end{proof}

\begin{cor}\label{derivative}
Let $ \alpha = \alpha^\delta$ be as in Theorem \ref{T:LERW_constant} and consider the mass $m$ as in
\eqref{eq:m} and $\rho = \rho^\delta = m^2 \delta^2/2$.
Suppose $o^\delta \to o \in \Omega, a^\delta \to a
\in \partial \Omega$,
$$
    \frac{\mathbb{P}^{(\alpha)}_{o^\delta}}{\mathbb{P}^{(\rho)}_{o^\delta}}(\gamma^\delta)  \rightarrow\exp(\left\la \alpha, a-o\right\ra).
$$
\end{cor}
\begin{proof}
This follows from our exact expression for $  ({\mathbb{P}^{(\alpha)}_{o^\delta}} / {\mathbb{P}^{(0)}_{o^\delta}})(\gamma^\delta)  $, \eqref{eq:martingalepart}, the already observed fact that the quadratic variation part cancels exactly with the mass, and the fact that $\delta^{-1}\alpha^\delta$ converges to $\alpha$.
\end{proof}

Note that this is stated without conditioning the massive walk to hit the boundary before dying.
(This conditioning would simply add a term to the Radon-Nikodym derivative of the previous lemma, corresponding to the probability to hit the boundary before dying.)

From Corollary \ref{derivative}, in particular we see that when $\gamma^\delta$ is a loop then $\P^{(\alpha)} (\gamma^\delta) = \P^{(\rho)} (\gamma^\delta)$. Although we do not need this here, this implies that the loop measures associated with the drifted walk $\P^{(\alpha)}$ and $\P^{(\rho)}$ are identical. Since these loop measures can be used to identify the law of loop-erased random walk (see, e.g., \cite[Chapter 9.5]{LawlerLimic}), we can use this observation to deduce that the expression obtained in Lemma \ref{derivative} can be transferred at the level of the loop-erasure. In fact this can be proved directly as follows. If $X$ is a lattice path, let $\LE (X)$ denote the chronological loop-erasure of $X$ considered up until its hitting time of $\partial \Gtree$ (if the path never reaches $\partial \Gtree$ -- for instance if it dies before reaching the boundary -- then $\LE(X)$ is by convention the empty path).

\begin{lem}\label{L:derivativeLE} Let $\gamma^\delta$ denote a fixed (sequence of) \textbf{simple} lattice paths from $o^\delta \in \Gtree$ to $a^\delta \in \partial \Gtree$, with $o^\delta \to o
\in \Omega$, $a^\delta \to a \in \partial \Omega$. Then
$$
\frac{\P^{(\alpha)} (\LE(X) = \gamma^\delta)}{\P^{(\rho)} (\LE(X) = \gamma^\delta)} \to \exp(\left\la \alpha, a-o\right\ra).
$$
as $\delta \to 0$.
\end{lem}

\begin{proof}
This follows directly from Corollary \ref{derivative} by summing over all ways to obtain $\gamma^\delta$ as a loop-erasure, and noting that the expression for the Radon-Nikodym derivative in Corollary \ref{derivative} depends only on the endpoints of the path, and not the rest of the path itself.
\end{proof}

As was mentioned in the introduction, the scaling limit of massive LERW is rather well understood, at least on the square lattice. Although the existing proofs of convergence to massive SLE$_2$ do not cover the case of the triangular lattice, it is possible with a bit of effort to extend these methods to cover this case (the main issue is the lack of reversibility which is needed to establish the crucial ``resolvent identity'' at the discrete level). We state the result here, but defer its proof until later, and see how this can be used to deduce Theorem \ref{T:LERW_constant}.

\begin{thm}
\label{T:massiveLERWtriangular}
Let $\Gtree \subset\delta\mathbb{T}$ approximate $\Omega$ with $ o^\delta\in\Omega^\delta\to o\in\Omega$, and let $a^\delta$ be a boundary point of $\Gtree$ such that $a^\delta \to a\in\partial\Omega$.
Consider the loop-erasure of a random walk sampled from $\P^{(\rho^\delta)}_{o^\delta}$, started at $o^\delta$ and conditioned to hit the boundary at $a^\delta$ before dying, with mass $\rho^\delta = m^2 \delta^2/2$, where $m = m^\delta \to m$, converges in law to radial massive SLE$_2$ from $a$ to $o$ with mass $m$. 
\end{thm}

Since the exit distribution of massive LERW from $\Gtree$, conditional on exiting this domain before dying, has a limit as $\delta \to 0$ (the ``massive harmonic measure'' on $\partial \Omega$), and since the law of radial massive SLE$_2$ from $a \in \partial \Omega$ to $o \in \Omega$ is continuous with respect to $a$, we deduce from this theorem that the scaling limit holds even if we do not condition on the exit point $a^\delta$ of the random walk, and simply condition on not dying before reaching the boundary.

\medskip The proof of Theorem \ref{T:massiveLERWtriangular} is deferred to Section \ref{section:convergence}. For now, we see how this immediately implies Theorem \ref{T:LERW_constant} for the triangular lattice.

\begin{proof}[Proof of Theorem \ref{T:LERW_constant}] This will follow rather simply from Lemma \ref{L:derivativeLE}, the fact that the expression for the Radon-Nikodym derivative is well-behaved, and the fact that the extinction probability for the massive walk converges to some nontrivial probability bounded away from zero and one as $\delta \to 0$. Indeed, since $\Omega$ is bounded, the function $\exp ( \la  \alpha, a - o \ra)$, viewed as a function of the endpoint $a \in \partial \Omega$, is a bounded continuous functional on path space. Therefore, if $F$ is another arbitrary such functional, then letting $\sigma^\delta$ be the hitting time of $\partial \Gtree$ by $X^\delta$,
\begin{align*}
\E_{o^\delta}^{(\alpha^\delta)} [ F ( \LE (X^\delta))] & = \E^{(\rho^\delta)}_{o^\delta} \left[ \left. F ( \LE(X^\delta)) \exp ( \la  \alpha + o(1), Y_\delta - o^\delta \ra) \right| \sigma^\delta < \infty\right] \P_{o^\delta}^{(\rho^\delta)} ( \sigma^\delta < \infty)\\
& \to \E^{\text{mSLE$_2$}}_o \left[ F(\gamma) \exp (\la  \alpha, Y - o \ra)\right] p(o),
\end{align*}
where $\E^{\text{mSLE$_2$}}_o$ denote the law of a massive radial SLE$_2$ started from massive harmonic measure on $\partial \Omega$ (a point which we denote by $Y$), towards $o$, and $p(o)$ is the survival probability for massive Brownian motion in $\Omega$ starting from $o$, killed at rate $\rho = m^2 = \|\alpha\|^2 /2$, i.e., $p(o) =  \P_{o}^{(\rho)} 
 ( \sigma< \infty) = \E^{\text{BM}}_o ( \exp ( - \|\alpha\|^2 \sigma/2))$ with $\sigma$ the exit time from $\Omega$. The rest of the result follows immediately by specifying $F$ to be a continuous function of the endpoint $Y_\delta$.
 \end{proof}

\subsection{Discrete Girsanov on the square lattice}\label{S:square}

Now let us consider the case of the square lattice, so $\Gtree$ is a portion of $\delta\Z ^2$ which approximates $\Omega$ in the sense discussed above. Our first task is to define precisely what we mean by $\P^{(\alpha^\delta)}$. 
%
%
Let $c_k\colon\Gtree\to\R$ be bounded functions for $k=0,\dots,3$, satisfying $c_0+c_2=c_1+c_3$ at each vertex $v \in \Gtree$.
Then $\P^{(\alpha^\delta)}$ is the law of the Markov chain on $\Gtree$ whose jump probabilities from the vertex $v\in \Gtree$ are given by
\begin{equation}\label{eq:jumpsquare}
\P^{(\alpha^\delta)}(v,v+\delta {\mathbf i}^{k})=\frac{a_k(v)}{a(v)} \quad k = 0, \ldots, 3\,,
\end{equation}
where
$$a(v)=\sum_{k=0}^3 a_k(v)\text{ and } a_k(v)=1+ c_k(v)\delta\text{, for }k=0,\dots,3, v \in \Gtree.$$
And define $\alpha = \alpha^\delta$ via
$$\alpha^\delta(v)=\frac12\sum_{k=0}^3a_k(v){\mathbf{i}}^{k} = \frac{\delta}2\sum_{k=0}^3c_k(v){\mathbf{i}}^{k}  \in \C \simeq \R^2.$$
Again $\alpha^\delta$ does not determine the $c_k$ uniquely, but only up to global shift, which does not influence the limit of the law, so that our notation $\P^{(\alpha^\delta)}$ is justified.

Together these assumptions guarantee that, if $c_k(v)$ is given by some fixed Lipschitz function $c_k:\Omega \to \R$ evaluated at $v \in \Gtree \subset \Omega$, then $\alpha^\delta = \delta \alpha + o(\delta),$ where $\alpha = (1/2) \sum_{k=0}^3 c_k \mathbf{i}^k$. (This is in particular the situation of interest for Theorem \ref{T:convergenceTree}, where $c_k$ are in fact constant). Thus $\alpha$ is itself a Lipschitz vector field defined on all of $\Omega$; this random walk converges to a Brownian motion with drift $\alpha$ under the same scaling as discussed in Section \ref{sec:notation}.

Again fix $\gamma^\delta=(x_0,\ldots,x_n)$ a given path, this time on the square lattice, starting from some point $x_0=o^\delta\in \Gtree$ of some length $n=N(\gamma^\delta)$.
Define $\alpha^\delta_k(v)\in\R, k=0,\ldots, 3$ by
\begin{equation}\label{E:sqalpha}
\exp(\alpha_k)= a_k = 1+ c_k \delta.
\end{equation}
Define also (for $i = 1,2$), $\beta_i = \beta^\delta_i(v)\ge 0$ by
$$
\exp(-\beta_i^2)=\frac{a_{i-1}a_{i+1}}{(a/4)^2}.
$$
Note that $\beta_1$ is well defined by the arithmetic-geometric mean inequality since for $i=1,2$,
$$\frac{a}{4}=1+\frac14\sum_{k=0}^3c_k\delta=1+\tfrac12(c_{i-1}\delta+c_{i+1}\delta)=\frac{a_{i-1}+a_{i+1}}{2},$$
where we used the assumption that $c_0+c_2=c_1+c_3$. We will then denote by $\beta \in \R^2$ the vector $\beta = ( \beta_1, \beta_2)$.

The next lemma gives the Girsanov identity in the case of the square lattice, which for conciseness we only give on the scaled lattice. 

\begin{lem}\label{L:discreteGsq}
On the square lattice, we have
\begin{equation}\label{E:discreteGfull}
\frac{\mathbb{P}_z^{( \alpha^\delta)}(\gamma^\delta)}{\P_z^{(0)} ( \gamma^\delta) }
= \exp (M_n - \tfrac{1}{2} V_n),
\end{equation}
where $M_n$ and $V_n$ can be written as
$$
M_n= \sum_{s=0}^{n -1} \delta^{-1} \langle \alpha (x_s); \dd x_s\rangle
\text{ and }
V_n  =  \sum_{s=0}^{n-1}    \delta^{-2}\| \beta(x_s) \odot \dd x_s \|^2
$$
where $a \odot b$ is the Hadamard product of the vectors $a= (a_1, a_2)$ and $b=(b_1, b_2)$, whose coordinate are $a_i b_i$ ($i=1, 2$). Explicitly, 
$V_n = \delta^{-2}\sum_{s=0}^{n-1}\beta_1(x_s)^2|\dd x_s^1|^2+\beta_2(x_s)^2|\dd x^2_s|^2.$
\end{lem}

\begin{proof}
Denote for a given path $\gamma^\delta$ of length $n$ whose starting point is $z$, by $n_0(v), \ldots, n_3(v)$ the number of steps of the walk from $v$ and going in the direct $1, {\mathbf i} , -1, -{\mathbf i}$ respectively.
 \begin{align}
      \mathbb{P}_z^{(\alpha)}(\gamma^\delta)
			&= \prod_{v \in v(\Gtree)} a_0^{n_0(v)} a_1^{n_1(v)} a_2^{n_2(v)} a_3^{n_3(v)} a^{-n} \nonumber \\
			&=4^{- n}\prod_{v\in v(\Gtree)}
			\left(\frac{a_0}{a_2}\right)^{\frac{n_0-n_2}{2}}
			\left(\frac{a_1}{a_3}\right)^{\frac{n_1-n_3}{2}}
			\left(\frac{a_0a_2}{(a/4)^2}\right)^{\frac{n_0+n_2}{2}}
			\left(\frac{a_1a_3}{(a/4)^2}\right)^{\frac{n_1+n_3}{2}} \nonumber \\
			&=4^{- n}\exp\left(\sum_{v \in v(\Gtree)}
			(\alpha_0- \alpha_2)\frac{n_0-n_2}{2}+
			(\alpha_1 - \alpha_3)\frac{n_1-n_3}{2} - 
			\beta_1^2\frac{n_0+n_2}{2}-
			\beta_2^2\frac{n_1+n_3}{2}
			\right) \nonumber \\
			&=4^{- n}\exp(M_n-\tfrac12V_n), \label{E:squareDensityGirsanov}
\end{align}
where in the last step we used that for each step of the walk in direction $1,i,-1$ or $-i$ the left two summands contribute $1/2$ times $\alpha_0-\alpha_2, \alpha_1 - \alpha_3, \alpha_2 -\alpha_0$ or $\alpha_3 - \alpha_1$ respectively, whereas the right two summands contribute $1/2$  times $\beta_1^2$ or $\beta_2^2$ 
depending on whether the displacement is horizontal or vertical. This leads to the expressions for $M_n$ and $V_n$ respectively.
The negative powers of $\delta$ in the expressions of $M_n$ and $V_n$ compensate the length of $\dd x_s$.
\end{proof}

We can again compare $\P^{(\alpha)}$ with an appropriate massive random walk.
Let $m = m^\delta = m^\delta(v)>0$ be defined by:
$$\tfrac14 (1-\frac{m^2\delta^2}{2})=\frac{\sqrt[4]{a_0a_1a_2a_3}}{a}$$
Note that this choice of $m$ also satisfies:
$$(1-\frac{m^2\delta^2}{2})=\exp(-\frac{\beta_1^2+\beta_2^2}{4}).
$$
We will see below that in the situation of interest for Theorem \ref{T:LERW_constant} (and so in particular for Theorem \ref{T:convergenceTree}) $\beta_i^2$ will indeed be of order $\delta^2$. 

Let $\rho = \rho^\delta(v) = m^2 \delta^2/2$ and let $\mathbb P^{(\rho)}$ be the law of the massive random walk, which has jump probabilities
$$\P^{(\rho)}(v,v+\delta {\mathbf i}^{k})=\tfrac14(1-\frac{m^2\delta^2}{2}),\text{ for }k=0\dots,3$$
and which jumps to an additional ghost or cemetery vertex with probability $\frac{m^2\delta^2}{2}$ (in which case we say that the path has died).
Let $\P^{(\rho)}(\cdot|Y_\delta = a^\delta))$ denote the conditional law of the massive random walk, given that the walk does not die before leaving $\Gtree$ and that the exit point is $a^\delta$.

\begin{lem} \label{L:beta_asymp_square}
    Suppose $\alpha^\delta$ does not depend on $v$ and $\alpha^\delta = \delta \alpha + o(\delta)$,  or equivalently $c_k $ converges as $\delta \to 0$ for $k = 0, \ldots, 3.$ Then for $i = 1,2$
    $$
    \beta_i^2 = \tfrac14 (c_{i-1} - c_{i+1})^2 \delta^2 + o(\delta^2).
    $$
    In particular $\delta^{-2} \| \beta\|^2 $ converges as $\delta \to 0.$ Furthermore, 
    $$
    m^2 \to \tfrac{\| \alpha \|^2}{2}.
    $$
\end{lem}

\begin{proof}
    Note that by definition, for $i = 1,2,$
$$e^{-\beta_i^2} = \frac{a_{i-1}a_{i+1}}{(a/4)^2}=\frac{(1+c_{i-1}\delta)(1+c_{i+1}\delta)}{(1+\frac{c_{i-1}+c_{i+1}}{2}\delta)^2}=1-\tfrac14(c_{i-1}-c_{i+1})^2\delta^2+o(\delta^2),$$
as desired. Now, 
\begin{align*}
1- \frac{m^2 \delta^2}{2} &= \exp ( - \tfrac{\beta_1^2 + \beta_2^2}{4} )  = \exp\left( - \tfrac{\delta^2}{16} \Big( ( c_2 - c_0)^2 + (c_3 - c_1)^2\Big)\right)
\end{align*}
so that 
$$
m^2 = \tfrac{1}{8} \Big( ( c_2 - c_0)^2 + (c_3 - c_1)^2\Big) + o( \delta^2).
$$
On the other hand, $\alpha = \tfrac12 \sum_{k=0}^3 c_k \mathbf{i}^k = ( (c_2 - c_0)/2, (c_3 - c_1)/2)$ so that 
$$
\|\alpha\|^2 = \tfrac14 ( (c_2-c_0)^2 + (c_3-c_1)^2),
$$
which concludes the proof. 
\end{proof}

While an exact connection between massive random walk and the random walk with drift conditioned on the exit point holds only for the triangular lattice, a similar statement holds asymptotically also for the square lattice.
To establish the connection between the two random walks we first prove the following lemma:

\begin{lem}\label{L:ascancel}
Fix $\alpha^\delta = \delta \alpha + o(\delta)$, where $\alpha \in \R^2$ is fixed. Let $\sigma = \sigma^\delta$ be the first time the random walk leaves the domain $\Gtree$ and $\theta<1$.
Then uniformly over $z_\delta \in \Gtree$:
$$\P^{(\alpha^\delta)}_{z_\delta}(|V_\sigma-\sigma \frac{\|\beta\|^2}2|>\delta^\theta)\rightarrow 0,$$
as $\delta\to0$. In particular this holds also under $\P^{(0)}_{z_\delta}$.
\end{lem}

\begin{proof}
Since $(X^{\delta}_{2\delta^{-2} t})_{t\ge 0}$ converges to Brownian motion with drift $\alpha \in \R^2$, $\sigma = \sigma^\delta$ is of order $\delta^{-2}$ and fluctuates on that scale, i.e. the distribution of $\delta^2 \sigma^{\delta}$ has a nontrivial weak limit, which simply is the law of the exit time $\sigma$ of $\Omega$ by a Brownian motion with drift $\alpha$ (let $\P^{(\alpha)}$ denote its law).

Let $\epsilon>0$.
Choose $K$ large enough that $\P^{(\alpha)}_{z}(\sigma>K)<\epsilon$ uniformly over $z\in \Omega$.
For $\delta$ small enough it follows by compactness that
$$\P_{z_\delta}^{(\alpha^\delta)}(\sigma>K \delta^{-2})<2\epsilon.$$
At each step, the walk takes a horizontal or a vertical step, each with probability $\tfrac12$ (since $c_0 + c_2 = c_1 + c_3$).
Therefore $Q_n:=V_n-n\frac{\beta^2_1+\beta^2_2}{2} = V_n - n \|\beta\|^2/2$ is a martingale with increment jumps uniformly bounded by $O(\delta^2)$ by Lemma \ref{L:beta_asymp_square}. 

Hence $q_n=\delta^{-2}Q_n$ is a martingale with bounded increments, and we are interested in the terminal value of $Q_n$ at the stopping time $\sigma = \sigma^\delta$.
Using Freedman's martingale inequality (Proposition (2.1) in \cite{Freedman}), we conclude
\begin{align*}
\P^{(\alpha^\delta)}_{z_\delta}(|Q_{\sigma^\delta}|>\delta^\theta)
&\leq \P^{(\alpha^\delta)}_{z_\delta}(\sigma>K\delta^{-2})+\P^{(\alpha^\delta)}_{z_\delta}(|q_{\sigma}|>\delta^{\theta-2};\sigma \le K\delta^{-2})\\
&\leq 2\epsilon+\exp(-\frac{\delta^{2\theta-4}}{2( C\delta^{\theta-2}+K\delta^{-2} )})\\
&\leq 2\epsilon+\exp(-c\delta^{2\theta-2}),
\end{align*}
where $c$ depends only on $K$ and $\theta$ (and hence only on $\epsilon$ and $\theta$) but not on $\delta$.
The lemma follows since $\theta<1$.
\end{proof}

This allows us to prove the analogue of Corollary \ref{derivative}:

\begin{cor}\label{P:massdrift_square}
Let $\gamma^\delta$ be a path in $\Gtree$ from $o^\delta \in \Gtree $.
Then if $o^\delta\to o\in\Omega$ and $a^\delta\to a\in\partial\Omega$
$$
    \frac{\mathbb{P}^{( \alpha^\delta)}_{o^\delta}(\gamma^\delta)}{\mathbb{P}^{(\rho^\delta)}_{o^\delta}(\gamma^\delta)}\rightarrow\exp(\langle (a-o), \alpha\rangle ),
$$
in probability as $\delta$ goes to $0$, under either the law $\P^{(\rho^\delta)} ( \cdot | Y_\delta = a^\delta)$ or $\P^{(\alpha^\delta)} ( \cdot | Y_\delta = a^\delta)$.
\end{cor}

\begin{proof}
By \ref{L:discreteGsq} we have that the ratio satisfies:
$$\frac{\mathbb{P}^{( \alpha^\delta)}_{o^\delta}(\gamma^\delta)}{\mathbb{P}^{(\rho^\delta)}_{o^\delta}(\gamma^\delta)}=\exp(M_n-\tfrac12V_n+n\frac{\|\beta\|^2}{4})=\exp(M_n)\exp\left(\tfrac12(n\frac{\|\beta\|^2}{2}-V_n)\right).$$
Since we are only considering paths that do not die before reaching their endpoint, Lemma \ref{L:ascancel} applies and the second term converges in probability to $1$ with respect to $\P^{(\alpha^\delta)}$
since the term in the exponential converges to $0$ in probability. 


On the other hand, as in the triangular case, since $\alpha^\delta$ does not depend on $v$, $M_n=\langle\delta^{-1}\alpha^\delta,a^\delta - o^\delta\rangle$, which converges to $\langle \alpha, a- o \rangle$ under our assumptions. 
\end{proof}

\begin{rem}
\label{R:RNdeterministic}
Note in particular that if $G_\delta$ is the good event
$$
G_\delta : = \{|V_{\sigma}-\sigma \frac{ \|\beta\|^2}2|\leq\delta^\theta \},
$$
then we have learnt that on $G_\delta$ we may write
$$
    \frac{\mathbb{P}^{( \alpha^\delta)}_{o^\delta}(\gamma^\delta)}{\mathbb{P}^{(\rho^\delta)}_{o^\delta}(\gamma^\delta)}= (1+ o(1))\exp(\langle (a-o), \alpha\rangle ),
$$
where the $o(1)$ term is nonrandom.
Note that since $\mathbb P^{(\rho^\delta)}(G_\delta | Y_\delta = a^\delta)\rightarrow 1$ this implies that Lemma \ref{L:derivativeLE} also holds on the square lattice.
\end{rem}

With this proposition we can now conclude to the proof of Theorem \ref{T:LERW_constant} in the case of the square lattice.

\begin{proof}[Proof of Theorem \ref{T:LERW_constant}, square lattice case] Let $F$ be a bounded continuous functional on curves in $\Omega$ (for the topology of uniform convergence of paths up to reparametrisation). Let $o, a$ and $o^\delta, a^\delta$ be as in Corollary \ref{P:massdrift_square}. Let $\gamma^\delta$ denote the random walk with jump probabilities given by \eqref{eq:jumpsquare} and let $\sigma^\delta$ denote the first time $\gamma^\delta$ leaves $\Gtree$.
Let $\LE(\gamma^\delta)$ denote the chronological loop-erasure of $\gamma^\delta$. We want to show that
\begin{equation}\label{E:goalsquare}
\E^{( \alpha^\delta)}_{o^\delta} [ F(\LE (\gamma^\delta)) ] \to \int_{a \in \partial \Omega}\E^{\text{mSLE$_2$}}_{o;a} [ F(\gamma)] \mu^{(\alpha)}_o (\dd a)
\end{equation}
where $\E^{\text{mSLE$_2$}}_{o;a}$ is the law of massive radial SLE$_2$ between $a$ and $o$ in $\Omega$, with mass $m = \| \alpha \|/{\sqrt{2}}$, and $\mu^{(\alpha)}_o (dy)$ denote the hitting distribution of Brownian motion with drift $\alpha$ of $\partial \Omega$ from $o$. Then 
\begin{align*}
\E^{( \alpha^\delta)}_{o^\delta} [ F(\LE (\gamma^\delta)) ] & = \E^{(\alpha^\delta)}_{o^\delta} [ F(\LE (\gamma^\delta) 1_{G_\delta}] + o(1)\\
& =  \E^{(\rho^{\delta})}_{o^\delta} [ F(\LE (\gamma^\delta)) 1_{G_\delta \cap \{\sigma^\delta< \infty\}} (1+ o(1)) \exp ( \langle  \alpha,  \gamma^\delta_{\sigma^\delta} - o^\delta \rangle)  ] + o(1) \\
& = (1+ o(1)) \E^{(\rho^\delta)}_{o^\delta} [ F(\LE (\gamma^\delta) 1_{\{\sigma^\delta < \infty\}} \exp ( \langle  \alpha, \gamma^\delta_{\sigma^\delta} -o^\delta \rangle) ] + o(1)
\end{align*}
by Remark \ref{R:RNdeterministic}. Now, using Lemma \ref{L:beta_asymp_square}, by \cite[Theorem 1.1]{ChelkakWan}, and since $\gamma^\delta_{\sigma^\delta}$ is a bounded, a.s. continuous functional of $\gamma^\delta$ (when $\Omega$ is bounded), we find
$$
 \E^{( \alpha^\delta)}_{o^\delta} [ F(\LE (\gamma^\delta) )  ] \to p^{(\rho)}(o)\int_{a \in \partial \Omega}\E^{\text{mSLE$_2$}}_{o;a} [ F(\gamma) \exp ( \langle  \alpha, a-o \rangle)] \mu^{(\rho)}_o (\dd a),
$$
where $\mu^{(\rho)}_o$ is the law of $X_\sigma$ under $\P_o^{(\rho)}$, conditioned on $\sigma<\infty$ and $p^{(\rho)}(o)$ is the probability of this event.

Taking $F$ to be a function of $\gamma^\delta_{\sigma^\delta}$ only, we see that
$$
\int_{a \in \partial \Omega} F(a) \mu^{(\alpha)}_o (\dd a) = p^{(\rho)}(o)\int_{a \in \partial \Omega}  F(a) \exp ( \langle  \alpha, a -o \rangle) \mu^{(\rho)}_o (\dd a),
$$
so that
$$
\mu^{(\alpha)}_o (\dd a) 
= p^{(\rho)}(o)\exp ( \langle  \alpha, a -o \rangle)\mu^{(\rho)}_o (\dd a)
$$
almost everywhere with respect to $\mu^{(\rho)}_o$. This proves \eqref{E:goalsquare} and hence Theorem \ref{T:LERW_constant} in the case of the square lattice.
\end{proof}

\begin{rem} \label{R:equalitycondition}
If we had not assumed $c_0+c_2=c_1+c_3$ we could not write the ``quadratic variation term'' $V_n$ in the form of a sum along the path of positive terms of type $\beta_k^2$, $k =1,2$. Even if we don't insist on the positivity of these terms and try to analyse the limiting behaviour, we find that $V_n$ is the sum of terms of order $\delta$ rather than $\delta^2$. The first order contribution however cancels out on the large scale and we do get a term of order 1 when $n$ is of order $\delta^{-2}$,  but it does not seem that this term can  easily be interpreted as a massive term; in particular it seems it might not be concentrated at a fixed time $n \approx t\delta^{-2}$.
In other words, the Radon-Nikodym dervative of the random walk with drift with respect to the massive random walk picks up a non-trivial contribution due to the walk taking more horizontal or vertical steps, even though the proportion of those steps behaves like $\tfrac12+c\delta$.
\end{rem}

\section{Convergence of massive LERW on the triangular lattice}\label{section:convergence}

In \cite{Lawler2001ConformalTrees} Lawler, Schramm and Werner proved that the scaling limit of the loop-erased random walk in a simply connected domain on the square lattice converges to radial SLE$_2$.
While the proof is written for the LERW on the square grid, in the last chapter it is mentioned that the proof can be adapted to more general setups; the random walk on the directed triangular lattice is explicitly mentioned as an example of an \emph{irreversible} random walk to which the proof applies.
In \cite{MakarovSmirnov} Makarov and Smirnov proposed a strategy for proving convergence of the massive LERW to massive SLE$_2$ building in part on ideas coming from Conformal Field Theory (see \cite{BauerBernardCantini09,Bauer2008LERWSLEs}).
This strategy was then successfully followed by Chelkak and Wan in \cite{ChelkakWan}, using a framework for convergence to SLE developed by Kemppainen and Smirnov in \cite{Kemppainen2017RandomEvolutions} and a recent addition \cite{Karrila2018LimitsCurves} by Karilla. We show in this section that the arguments of Chelkak and Wan in \cite{ChelkakWan} can be adapted to the directed triangular lattice which will imply a proof of Theorem \ref{T:massiveLERWtriangular}. The additional difficulty here is the lack of reversibility, which is crucially used to derive a discrete ``resolvent identity'' and is the heart of the proof in \cite{ChelkakWan}; see in particular Proposition \ref{prop: green convergence} below.
We note that a more general proof (but requiring quite a bit more work) will be given in Section \ref{S:LERW_general}, so that this section could be skipped by the reader. 

In order to stay close to the notations of \cite{ChelkakWan} we will use, in this section only, the notation $\P^{(m)}$ (instead of $\P^{(\rho)}$) for the massive random walk which dies with probability $m^2 \delta^2/2 $ at each step; likewise partition functions will be denoted e.g. by $Z^{(m)}$, as we will see below.

\subsection{Convergence of domains and curves}\label{Kemp}
For each discrete domain $\Gtree \subset\delta\mathbb{T}$ we associate a polygonal domain $\hat\Gtree \subset\mathbb{C}$ which is the union of open hexagons with side length $\delta$ centered at vertices of $\Gtree$.
Notice that vertices of $\delta\mathbb{T}$ on the boundary of $\hat\Gtree$ are exactly vertices on the outer vertex boundary of $\Gtree$.

We will assume that $\hat\Gtree$ converges to $\Omega$ in the \emph{Carathéodory} topology and if this is the case write, that $\Gtree$ approximates $\Omega$.
This means that each inner point of $\Omega$ belongs to $\hat\Gtree$ for small enough $\delta$ and each boundary point of $\Omega$ can be approximated by boundery points of $\Gtree$, see, e.g., \cite{Pommerenke1992BoundaryMaps}.
Further, we assume that we are given $o^\delta\in\Gtree$ converging to $o\in\Omega$ and $a^\delta\in\partial\Gtree$ converging to $a\in\partial\Omega$.
Let $\psi_{\hat\Gtree}\colon\hat\Gtree\rightarrow\mathbb{D}$ be the unique conformal map such that $\psi_{\hat\Gtree}(o^\delta)=0$ and $\psi_{\hat\Gtree}(a^\delta)=1$.
Then it can be seen (see, e.g., \cite{Pommerenke1992BoundaryMaps}) that Carathéodory convergence is equivalent to the uniform convergence on compacts of $\psi_{\hat\Gtree}$ and $\psi_{\hat\Gtree}^{-1}$ to $\psi_{\Omega}$ and $\psi_\Omega^{-1}$ respectively.

The main theorem of \cite{Kemppainen2017RandomEvolutions} states that if a family $\Sigma$ of measures of random curves satisfies a certain annulus crossing condition, then the family is tight and furthermore, if $\mathbb{P}_n\in\Sigma$ is a weakly converging subsequence then its limit is a random Loewner chain.
Moreover if $(W^{(n)})_{n\in\mathbb{N}}$ are the driving processes of the random curves $(\gamma^{(n)})_{n\in\mathbb{N}}$ that satisfy the annulus crossing condition which are parametrized by capacity then:
\begin{itemize}
    \item $(W^{(n)})_{n\in\mathbb{N}}$ is tight in the space of continuous functions on $[0,\infty)$ with the topology of uniform convergence on compact subsets.
    \item $(\gamma^{(n)})_{n\in\mathbb{N}}$ is tight in the space of curves up to reparametrization with the supremum norm.
\end{itemize}
If the sequence converges in either of the topologies it also converges in the other one and the limit of the driving processes is the driving process of the limiting random curve.

That the annulus crossing condition is satisfied is checked for a chordal loop-erased random walk in \cite[Section 4.5]{Kemppainen2017RandomEvolutions} with a remark that the radial case is equivalent to calculations in \cite{Lawler2001ConformalTrees}.

\subsection{Absolute continuity with respect to classical SLE\texorpdfstring{$_2$}{}}\label{section:Absolutecont}
Let $0<\delta<m^{-1}\leq \infty$.
Here, $m$ is the mass, which we allow to be zero and $\delta$ is the scale.
We consider subgraphs $\Gtree$ of the scaled triangular lattice $\delta \mathbb{T}$, which approximate some domain $\Omega \in \mathbb{C}$.
Given such $\delta, m$ and $\Gtree$ as well as two vertices $x^\delta, y^\delta$ we define the partition function of the massive random walk:
\begin{equation}
    Z_{\Gtree}^{(m)}(x^\delta,y^\delta)\coloneqq\sum_{\pi^\delta\in S(x^\delta,y^\delta)} \left(\frac{1}{3}(1-\frac{m^2\delta^2}{2})\right)^{\#\pi^\delta},
\end{equation}
where the sum is over all possible paths $\pi^\delta$ from $x^\delta$ to $y^\delta$ remaining in $\Gtree$.
If $m=0$ this corresponds to the classical random walk and we drop the superscript $(m)$; thus
$$ Z_{\Gtree}(x^\delta,y^\delta) =  Z_{\Gtree}^{(0)}(x^\delta,y^\delta) .
$$
If $x^\delta$ is an interior vertex and $y^\delta$ is a vertex on the boundary, this is the probability that a random walk with killing rate $\frac{m^2\delta^2}{2}$ started at $x^\delta$  leaves the boundary at $y^\delta$ without any conditioning. More generally, $Z_{\Gtree} ( x^\delta, y^\delta)$ is the \textbf{discrete massive Green function}, i.e. the expected number of visits to $y^\delta$ starting from $x^\delta$ before hitting the boundary or being killed.
Note that, because of the directed edges in general $Z_{\Gtree}^{(m)}(x^\delta,y^\delta)\neq Z_{\Gtree}^{(m)}(y^\delta,x^\delta)$.
In the limit however, we will see (in section \ref{massiveconvergence}) that equality holds.

To apply the tightness results to the massive case we first need some estimates on this partition function, which are similar (but easier in some respects) as Lemma 2.4 and Proposition 2.5 in \cite{ChelkakWan}.
\begin{prop}\label{P:abscont}
For each domain $\Gtree$ with $\Gtree \subset B(0,1)$, for each $\eps>0$ there exists $c>0$ (depending only on $\eps>0$) such that the following holds.
For each
 interior point $x^\delta$ at distance at least $\eps>0$ from the boundary, and for each boundary point $a^\delta$, as long as $\delta\leq \frac 1 2 m^{-1}$, one has
\[
\frac{Z_{\Gtree}^{(m)}(x^\delta,a^\delta)}{Z_{\Gtree}(x^\delta,a^\delta)}\geq \exp(-cm^2).
\]
\end{prop}
\begin{proof}
We proceed as in the proof of Proposition 2.5 in \cite{ChelkakWan}. By Jensen's inequality (since $1-\frac{m^2\delta^2}{2} \ge 0$): 
\[
\frac{Z_{\Gtree}^{(m)}(x^\delta,a^\delta)}{Z_{\Gtree}(x^\delta,a^\delta)}=\mathbb{E}_{x^\delta\to a^\delta}\left((1-\frac{m^2\delta^2}{2})^{\sigma_{\Gtree}}\right)\geq (1-\frac{m^2\delta^2}{2})^{\mathbb{E}_{x^\delta\to a^\delta}(\sigma_{\Gtree})},
\]
where the expectation is for a classical random walk started at $x^\delta$ conditioned to leave $\Gtree$ at $a^\delta$ and $\sigma$ is the time at which this random walk exits $\Gtree$.
Therefore it suffices to show
\begin{equation}\label{E:radialestimate}
\mathbb{E}_{x^\delta\to a^\delta}(\sigma_{\Gtree})\leq \text{const}\cdot \delta^{-2}.
\end{equation}
This is a statement about the simple random walk and follows as in \cite[Corollary 2.8]{ChelkakWan}. This relies on \cite[Lemma 2.7]{ChelkakWan}, which uses the discrete Harnack inequality.
The discrete Harnack inequality is satisfied on the directed triangular lattice, see e.g. Lemma~\ref{lem:Harnack} below for the discrete Harnack inequality valid in the massive case, which includes in particular the non-massive case.
\end{proof}

From this (just as in \cite[Section 2.5]{ChelkakWan}) it follows that the densities of massive LERW with respect to classical LERW are uniformly bounded from above by $\exp(cm^2R^2)$ and thus the tightness of the law of massive LERW follows.
Also, (as in \cite[Section 2.6]{ChelkakWan}) it follows that each subsequential limit of $\P^{(m)}_{\Gtree}$ is absolutely continuous with respect to the SLE$_2$ on $\Omega$.
Thus we can use Girsanov's theorem to find the driving term of $\xi_t$ of the Loewner evolution under $\P^{(m)}_{\Gtree}$, see also Section~\ref{S:absolutecontinuity} for the analogous statement in a more general setting.

\subsection{Convergence of the Green function}\label{massiveconvergence}

In this section we prove the convergence of $Z^{(m)}_{\Gtree}(x^\delta,y^\delta)$ to a multiple of the massive Green function $G^{(m)}_\Omega(x,y)$.
To do so we will show that $G^{(m)}_{\Gtree} (x^\delta, \cdot)$ is precompact in a suitable space of functions, and we will show that any subsequential limit must
satisfy the following three properties:
\begin{align}
    G^{(m)}_\Omega(x,\cdot)&=0 \text{ on the boundary of $\Omega$,}\label{massiveboundary}\\
    (-\tfrac{1}{2}\Delta+m^2)G^{(m)}_\Omega(x,\cdot)&=0\text{ away from x, and }\label{massiveharmonic}\\
    G^{(m)}_\Omega(x,y)&= \frac{1}{\pi}\log(|x-y|^{-1})+O(1) \text{ as $y\rightarrow x$.}\label{massivelog}
\end{align}
As we will see, these three properties uniquely characterise the the (continuous) \textbf{massive Green function}; from this the desired convergence will follow immediately.
The second condition is that $G(u,\cdot)$ is a massive harmonic function. It will be useful to appeal to the discrete notion of massive harmonicity: given $m\ge 0$ we call a function $H$ massive discrete harmonic at $v\in\delta\mathbb{T}$ if
\begin{equation}
    H(v)=\frac{1}{3}(1-\frac{m^2\delta^2}{2})\sum_{w\in\delta\mathbb{T}:w\sim v}H(w).
\end{equation}

\begin{rem}
    Note that discrete massive harmonic functions with mass $m$ correspond massive harmonic functions in the sense of \eqref{massiveharmonic}.
    Indeed, the graph Laplacian approximates $\frac{1}{4}\Delta$ as $\delta\to0$ and thus a limit $h$ of massive harmonic functions $h^\delta$ on $\delta\T$ satisfies
    \[
    (\frac14\Delta-\frac{m^2}{2})h=\frac12 (\frac12\Delta-m^2)h=0\,.
    \]
    This is precisely the reason for the factor $\frac{1}{2}$ in the definitions of the massive random walk.
\end{rem}
$H$ being a discrete massive harmonic function is equivalent to being discrete harmonic on the augmented graph where every vertex is connected to an additional cemetery vertex, where the transition probability to the cemetery vertex is $\frac{m^2\delta^2}{2}$ from every point; and the value of $H$ at the cemetery vertex being 0. We immediately deduce:

\begin{lem}\label{L:Dirichlet}
Let $\Gtree$ be a bounded domain in $\delta\mathbb{T}$ and $(X_n)_{n\in\mathbb{N}}$ be a massive random walk with mass $\frac{m^2\delta^2}{2}$.
Let $H$ be a bounded real valued function defined on $\Gtree\cup\partial\Gtree$ and massive discrete harmonic at every point of $\Gtree$.
Denote by $\mathbb{P}^{(m)}_{x^\delta}$ the law of this walk started at $x^\delta$ and by $\mathbb{E}^{(m)}_{x^\delta}$ the corresponding expectation.
Let $\sigma_{\Gtree}$ be the hitting time of the boundary and let $\sigma^*$ denote the killing time, i.e. the hitting time of the cemetery vertex.
Then
\[
H(x^\delta)=\mathbb{E}^{(m)}_{x^\delta}\Big(H(X_{\sigma_{\partial\Gtree}})1_
{\{
\sigma^* > \sigma_{\Gtree}
\}}
\Big).
\]
\end{lem}
The above statement needs to be interpreted carefully as we defined the boundary $\partial \Gtree$ to be the edge boundary, that is, pairs $(y_1, y_2)$ of vertices such that exactly one of these vertices (say $y_1$) lies in $\Gtree$. In the above statement, we abusively identify $\partial \Gtree$ with the outer vertex boundary (i.e., the vertices of the form $y_2$ where $(y_1, y_2)$ is a boundary edge such that $y_1 \in \Gtree $ but $y_2 \notin \Gtree$).
Now we can prove the uniqueness of the Green function:
\begin{lem}\label{lem:Green uniqueness}
For each $x\in\Omega$ and $k\in\mathbb{R}^+$ there is exactly one function $G(x,\cdot)\colon\Omega\rightarrow\mathbb{R}$ that is massive harmonic away from $x$, $0$ on the boundary, and satisfies
\[
G(x,\cdot)= k\log(|x-y|^{-1})+o(\log|x-y||) \text{ as $y\rightarrow x$.}
\]
\end{lem}
\begin{proof}
Let $h$ and $g$ be two such functions.
Then $f\coloneqq h-g$ is a massive harmonic function that is massive harmonic away from $x$, $0$ on the boundary, and  $$
f(y) = o(\log(|y-x|))$$ as $y\to x$.
Fix $z \neq x \in\Omega $ and let $\P_z^{(m)}$ be the law of massive Brownian motion with mass $m$ started at $z$: thus if $\sigma^*$ denote an exponential random variable with rate $m^2$ then by definition
$$
\E^{(m)}_z ( f(B_t)) = \E_z ( f(B_t) 1_{\{ \sigma^* > t\} }).
$$
Since $f$ is massive harmonic,  $M_t = f(B_t)1_{\{\sigma^*>t\}}$ is a $\P^{(m)}_z$-local martingale.
Let $r>0$, $B(x,r)$ be the disk of radius $r$ with center $x$, $\sigma_r$ the hitting time of $B(x,r)$ and $\sigma_{\Omega}$ the hitting time of $\partial\Omega$.
It is a well known fact about Brownian motion that the probability that
$$
\P_z (\sigma_r<\sigma_{\Omega}) \lesssim 1/\log (1/r),
$$
as $r\to 0$.
(This can be seen by applying the optional stopping theorem to the $\P_x$-local martingale $\log|B_t - x|$, see for example \cite{LeGall2016BrownianCalculus}).
By applying the optional stopping theorem to $M$ under $\P_z^{(m)}$ (which is justified since $f$ is smooth and hence bounded away from $x$, as $\Omega$ is bounded) we obtain:
\begin{align*}
f(z)& =\mathbb{E}^{(m)}( M_{\sigma_r\wedge\sigma_{\Omega}}).
\end{align*}
The only contribution comes from the event $\sigma_r < \min ( \sigma^*, \sigma_\Omega)$ since if either of these two stopping times occur before $\sigma_r$ then the martingale is equal to zero. Hence
$$
f(z) = \E_z ( f ( B_{\sigma_r} )1_{\sigma_r < \min ( \sigma^*, \sigma_\Omega)} )
$$
But $f(B_{\sigma_r}) = o(\log(r))$ by assumption on $f$, and
$$
\P_z( {\sigma_r < \min ( \sigma^*, \sigma_\Omega)} ) \le \P_z( \sigma_r < \sigma_\Omega) \lesssim 1/\log (1/r).$$
 Hence letting $r \to 0$ we see that $f(z) = 0$.
Since $z$ was arbitrary, we deduce $f$ is $0$ everywhere and hence $g = h$, as desired.
\end{proof}

(The existence of a function satisfying \eqref{massiveboundary}, \eqref{massiveharmonic} and \eqref{massivelog} follows from the result in \cite{ChelkakWan}, or the convergence result below.) In order to prove convergence of the discrete Green function $Z^{(m)}_{\Gtree} (x^\delta, y^\delta)$ to $G^{(m)}(x, y)$ we will show precompactness and identify the limit ultimately via Lemma \ref{lem:Green uniqueness}. The following lemma will be useful for the existence of subsequential limits:

\begin{lem}\label{lem:Harnack}
There are constants $C$ and $\beta$ depending on $m$ such that for all positive massive harmonic functions $H$ defined in $B(x_0,2r)\cap\delta\mathbb{T}$ with $r\leq m^{-1}$ and for all $x_1,x_2\in B(x_0,r)\cap\delta\mathbb{T}$ one has:
\[
|H(x_1)-H(x_2)|\leq C(|x_2-x_1|/r)^\beta \max_{x\in B(x_0,2r)\cap\Gtree}(H(x)).
\]
\end{lem}
\begin{proof}
Essentially,  one  can follow the argument of \cite[Lemma 3.10]{ChelkakWan}. Its proof relies on the following estimate: for any annulus $A=A(x_0,r,2r)$, let $E(A)$ be the event that $X_n$ makes a non-trivial loop around in the annulus before leaving it and before dying, i.e. there are $0<s<t<\sigma_{ A}$ such that $X[s,t]$ disconnects $x_0$ from $\infty$; and $\sigma^* > \sigma_A$.
Then there exists a positive constant $c>0$ independent of $\delta, r, x_0$, and $x$ such that:
\begin{equation}
\mathbb{P}^{(m)}_x(E(A(x_0,r,2r)))\geq c,
\end{equation}
for all $8\delta<r\leq m^{-1}$ and all $v\in\delta\mathbb{T}$ such that $\frac 3 2 r - \delta \leq |x_0-v|\leq \frac 3 2 r + \delta$. This needs to be established in our directed context, which is not covered explicitly by \cite{ChelkakWan}. To see this, simply observe that we can restrict to $\sigma_A \le Mr^2 \delta^{-2}$ for some large $M$.
Using this we obtain
\begin{align*}
\mathbb{P}^{(m)}_x(E(A)) & \geq \mathbb{P}^{(m)}_x(E(A) ; \sigma_A \le M \delta^{-2} r^2) \\
& \geq  \mathbb{P}^{(0)}_x(E(A) ; \sigma_A \le M \delta^{-2} r^2) (1-\frac{m^2\delta^2}{2})^{M \delta^{-2} r^2}\\
& \ge \exp ( - (M/2) r^2 ) [ \P^{(0)}_x (E(A)) - \P^{(0)}_x( \sigma_A > M \delta^{-2} r^2)]\,.
\end{align*}
It is well-known and easy to see that $ \P^{(0)} (E(A)) $ is bounded away from 0 (by convergence to Brownian motion) and the second term can be made arbitrarily small by choosing $M$ sufficiently large. The result follows.
\end{proof}


We also need the following lemma about convergence of the conditioned (non-massive) random walk to a Brownian bridge:
\begin{lem}\label{bridge convergence}
Let $t>0$. Let $X_n^\delta$ be the simple random walk on $\delta\mathbb{T}$ started at $x^\delta$ converging to $x$.
Let $y^\delta\in\delta\mathbb{T}$ approximate $y$ such that for any $\delta>0$ small enough it is possible to go from $x^\delta$ to $y^\delta$ in $\lfloor\delta^{-2}t\rfloor$ steps with positive probability.
Then the law of $(X^\delta_{\lfloor\delta^{-2}s\rfloor})_{s\in[0,t]}$ conditioned on $X^\delta_{\lfloor\delta^{-2}t\rfloor}=y^\delta$ converges to the law of the Brownian bridge $(b_s)_{s\in[0,t]}$ from $x$ to $y$ of duration $t>0$.
\end{lem}
\begin{proof}
We interpolate linearly between vertices to consider $(X^\delta_{\lfloor\delta^{-2}s\rfloor})_{s\in[0,t]}$ as a continuous function on $[0,t]$.
Let $(S^\delta_s)_{0 \le s \le t}$ be this interpolation.
Fix $u \in[0,t)$, and let us show that $(S^\delta)_{0 \le s \le u}$ converges to $(b_s)_{0 \le s \le u}$.
Fix $F\colon C([0,u])\rightarrow\mathbb{R}$ be a bounded continuous functional.
Then the conditioning $S^\delta_t=y^\delta$ weights every path $(S^\delta_s)_{s\in[0,u]}$ by how likely it is to go to $y^\delta$ from $S^\delta_u$.
Thus the conditional expectation of the functional can be rewritten as:
\begin{equation}\label{E:RN}
\mathbb{E}_{x^\delta}(F((S^\delta_s)_{s\in[0,u]})|S^\delta_t=y^\delta)=
\mathbb{E}_{x^\delta}\left(F((S^\delta_s)_{s\in[0,u]})
\frac{\mathbb{P}_{x^\delta}(S^\delta_t=y^\delta|S^\delta_u)}{ \P_{x^\delta}(S^\delta_t=y^\delta)} \right).
\end{equation}
The probability in the enumerator can be written as $\P_{y^\delta} (S^\delta_{t-u } = y^\delta)$, with $y^\delta = S^\delta_u$. The ratio of probabilities therefore converges and the limit is
\[
\mathbb{E}(F((B_s)_{s\in[0,u]})\frac{\varphi(\frac{y-B_u}{t-u})}{\varphi(\frac{y}{t})})=\mathbb{E}(F((b_s)_{0 \le s \le u} )),
\]
where $\varphi$ is the density of a two-dimensional standard normal random variable, and $B$ is standard Brownian motion.
Since this holds for every $u<t$ it implies the statement. 

\end{proof}

We will use this to approximate the probability that a random walk conditioned on the point at time $n$ leaves a domain by the corresponding probability for the Brownian motion.
\begin{cor}\label{cor:uniform bridge}
Let $\Gtree$ approximate a domain $\Omega\in\mathbb{C}$ and $x^\delta,y^\delta$ approximate $x,y$ in $\Omega$. Let $\P_{x\to y ; t}$ denote the law of a Brownian bridge  of duration $t$ from $x$ to $y$.
For any $t>0$,
\[
\mathbb{P}_{x^\delta}(\sigma_{\Gtree} >t\delta^{-2}|X^\delta_{\lfloor t\delta^{-2}\rfloor}=y^\delta)\rightarrow P_{x,y}(t): = \mathbb{P}_{x\to y ; t}(\sigma_{\Omega}>t)
\]
\end{cor}

Suppose $x,y$ are fixed. When $t$ is small the Brownian bridge of duration $t$ is close to a straight line segment $[x,y]$. If the latter is contained in $\Omega$ then it is very likely that the bridge did not leave $\Omega$ by time $t$. This can be made rigorous through the following lemma.

\begin{lem}\label{lem:Pzero}
Let $P_{x,y}(t)$ be as above.
Assume that the line between $x$ and $y$ is in $\Omega$.
Then:
\[
\lim_{t\rightarrow 0}P_{x,y}(t)=1.
\]
Furthermore, $P_{x,y}$ is a continuous function of $t>0$.
\end{lem}
\begin{proof}
Let $(b_s)_{s\in[0,t]}$ be the Brownian bridge from $x$ to $y$ of duration $t$.
A well-known representation of the Brownian bridge is $b_s=x+(y-x)\frac{s}{t}+W_s-\frac{s}{t}W_t$, where $(W_s)_{s\in[0,t]}$ is a standard two dimensional Brownian motion started at $0$.
By rescaling the time to the interval $[0,1]$ we get $\hat b_t=b_{tc}$ for $t$ in $[0,1]$, which satisfies:
\[
\hat b_s=x+(y-x)s+W_{st}-sW_t.
\]
As $t\rightarrow 0$ the second term $W_{st}-sW_t$ converges to $0$ in probability uniformly in $s$.
Since $\Omega$ is an open set and hence also contains an open set around the line from $x$ to $y$ this implies that $P_{x,y}(t)$ converges to $1$.
Continuity in $t>0$ easily follows from the fact that as $t \to t_0>0$, the law of a Brownian bridge of duration $t$ (suitably extended or restricted) converges weakly for the uniform topology to that of a Brownian bridge of duration $t_0$, the portmanteau theorem, and the fact that a Brownian bridge of duration $t_0$ has probability zero to hit the boundary without leaving the domain $\Omega$.
\end{proof}

It is also useful to recall the following elementary estimate which can be obtained e.g. by Stirling's approximation (or from computing the Fourier transform):
\begin{lem}\label{L:HK2d}
	Let $x^\delta$ and $y^\delta\in\delta\mathbb T$ be sequences of lattice points.
	Then there exists a constant $C<
	\infty$ independent of $x^\delta,y^\delta,\delta$ and $n$ such that
	\begin{equation}
	\P_{x^\delta} (X_{n} = y^\delta) \leq \frac{C}n
	\end{equation}
	for some universal constant $C>0$.
\end{lem}

\begin{lem}\label{L:multi}
Let $x,y\in\mathbb C$ and $x^\delta,y^\delta=x^\delta + a\delta + b\delta\tau\in\Gtree$ such that $x^\delta\to x$ and $y^\delta\to y$ and that $n-a-b$ is divisible by $3$.
Then
\begin{equation}
3^{-n}\binom{n}{n-a-b,n-a+2b,n+2a-b}=\frac{\sqrt{27}}{2\pi n}\exp(-\frac{|x-y|^2}{\delta^2n})(1+O(\delta)).
\end{equation}
where the error is uniform in $\delta$ small enough $,x^\delta,y^\delta\in\Gtree$ and $n$ such that $|x-y|^2\delta^{-2}<n<M\delta^{-2}$ for some constant $M$.
\end{lem}
\begin{proof}

Since $y^\delta-x^\delta\to y-x$ we have that $a$ and $b$ are of order $\delta^{-1}$.
Because the domains are bounded they are uniformly of this order.
Therefore all entries in the multinomial coefficient are uniformly of order $\delta^{-2}$ and we can apply Stirling's approximation to all appearing factorials to obtain that the multinomial coefficient equals:
\begin{align*}
	&\frac{n^n\sqrt{2\pi n}} {(\frac{n-a-b}{3})^{\frac{n-a-b}{3}}(\frac{n+2a-b}{3})^{\frac{n+2a-b}{3}}(\frac{n-a+2b}{3})^{\frac{n-a+2b}{3}}(\sqrt{2\pi n/3})^33^n}(1+O(\delta^2)) \\
	=& \frac{\sqrt{27}}{2\pi n}
	\left((1+\frac{-a-b}{n})(1+\frac{2a-b}{n})(1+\frac{-a+2b}{n})\right)^{-\frac{n}{3}}\times\\
	&(1+\frac{-a-b}{n})^{-\frac{-a-b}{3}}(1+\frac{2a-b}{n})^{-\frac{2a-b}{3}}(1+\frac{-a+2b}{n})^{-\frac{-a+2b}{3}}(1+O(\delta^2))\\
	=&\frac{\sqrt{27}}{2\pi n}
	\left(1+\frac{-3(a^2-ab+b^2)}{n^2}+O(\delta^3)\right)^{-\frac{n}{3}}\times\\
	&(1+\frac{-a-b}{n})^{-\frac{-a-b}{3}}(1+\frac{2a-b}{n})^{-\frac{2a-b}{3}}(1+\frac{-a+2b}{n})^{-\frac{-a+2b}{3}}(1+O(\delta^2))\\
	=&\frac{\sqrt{27}}{2\pi n}
	\exp(\frac{a^2-ab+b^2}{n})\exp(-\frac{(-a-b)^2+(2a-b)^2+(2b-a)^2}{3n})(1+O(\delta))\\
	=&\frac{\sqrt{27}}{2\pi n}\exp(-\frac{|x-y|^2}{\delta^2n})(1+O(\delta)).
\end{align*}
In the last step we used that $\delta^2(a^2-ab+b^2)=|a\delta+b\delta\tau|^2=|x-y|^2+O(\delta)$.
\end{proof}

\begin{lem}\label{L:expexit}
Let $\Gtree\subset\delta\mathbb T$ be a sequence of lattice domains satisfying $\Gtree\subset B(0,R)$ for some $R>0$ independent of $\delta$.
Let $x^\delta$ and $y^\delta\in\Gtree T$ be a sequences of lattice points.
Then there exists a constants $c>0$ depending on $R$, but not on $\delta, n, x^\delta$ or $y^\delta$ such that for all $n\geq 1$:
\begin{equation}
\mathbb{P}^{(0)}_{x^\delta}(\sigma_{\Gtree}>n|X_n^\delta=y^\delta)<\exp ( - c n \delta^2).
\end{equation}
\end{lem}

\begin{proof} This can easily be deduced from the fact that the Radon--Nikodym derivative of the conditioned random walk compared to an unconditional random walk, restricted to $[0, n/2]$, is bounded (see, e.g., \eqref{E:RN}), and the analogous (and straightforward) bound for unconditional random walk.
\end{proof}

Now we state the main result of this section:

\begin{prop}\label{prop: green convergence}
Let $\Omega\subset\mathbb{C}$ be a bounded simply connected domain and $x,y\in\Omega$ be two distinct points of $\Omega$.
Assume that discrete domains $\Gtree\subset\delta\mathbb{T}$ approximate $\Omega$.
Then for all $x^\delta\to x, y^\delta\to y$
\[
Z^{(m)}_{\Gtree}(x^\delta,y^\delta)\rightarrow \sqrt{3}G^{(m)}_\Omega(x,y).
\]
\end{prop}

\begin{proof}

To begin we rewrite the Green function as
\begin{align}
Z^{(m)}_{\Gtree}(x^\delta,y^\delta) 
&=\sum_{n=0}^\infty  \P^{(0)}_{x_\delta}(X_n = y^\delta)\left(1-\frac{m^2\delta^2}{2}\right)^n\mathbb{P}^{(0)}_{x^\delta}(\sigma_{\partial\Omega}>n|X_n=y^\delta).
\end{align}
We split this sum into three parts:
First the sum from $n=0$ to $\lfloor|x-y|^2\delta^{-2}\rfloor$, then from $n = \lfloor|x-y|^2\delta^{-2}\rfloor +1$ to $\lfloor M\delta^{-2}\rfloor$ (where $M$ is a large constant chosen suitably later), then larger values of $n$.
We will call these sums $I,II$ and $III$ and estimate them separately.

\medskip \noindent \textbf{Bounding $I$.} To estimate the first part of the sum we compare $\P_{x^\delta}^{(0)} (X_n= y^\delta)$ with the same probability for points that are closer to $x$, as follows.
Depending on the residue of $n$ modulo $3$ a different set of vertices is reachable from $x^\delta$. Assuming that $z$ is reachable after $n$ steps and its distance to $x$ is at most half $|x-y^\delta|$, then the probability to be at $z$ after $n$ steps is bigger than the probability to be at $y^\delta$:

\begin{lem}
\label{L:comparison}
Fix $n\ge 0$. For any vertex $z$  such that $\P_{x^\delta}^{(0)} (X_n = z) >0$ and $|z-x^\delta|<\frac{1}{2}|y^\delta-x^\delta|$, it holds that
$$
\P_{x^\delta}^{(0)} (X_n = z)  \ge \P_{x^\delta}^{(0)} (X_n = y^\delta).
$$
\end{lem}

\begin{proof}
Since the number of steps $n$ is fixed this is just about comparing multinomial coefficients.
It is easy to check that for any $n$ and any $a_1,a_2,a_3$ such that $n=a_1+a_2+a_3$ and $a_1>a_2$ it holds that:
$$
\binom{n}{a_1,a_2,a_3}\leq\binom{n}{a_1-1,a_2+1,a_3}.
$$
Assume without loss of generality that $y^\delta-x^\delta=a_1+a_2\tau+a_3\tau^2$, such that $a_1+a_2+a_3=n$ and $a_1\geq a_2 \geq a_3$.
The above inequality implies that for any $z$ reachable from $y$ by repeatedly reducing one of the $a_i$ and increasing another $a_j$ subject to $a_i>a_j$ satisfies:
$\P_{x^\delta}^{(0)} (X_n = z)  \ge \P_{x^\delta}^{(0)} (X_n = y^\delta)$.
It is clear that in this way only points $z$ can be be obtained that are also reachable in $n$ steps from $x^\delta$.

\textbf{Claim} All $z$ reachable from $x^\delta$ in $n$ steps, which are in the quadrilateral descriped by the lines through $x^\delta$ in the directions $1$ and $\tau$ and through $y^\delta$ in the directions orthogonal to $1$ and $\tau$ are reachable through these operations.
See figure \ref{F:multinomialhex}.

\begin{figure}
\begin{center}
\includegraphics[width=.4\textwidth]{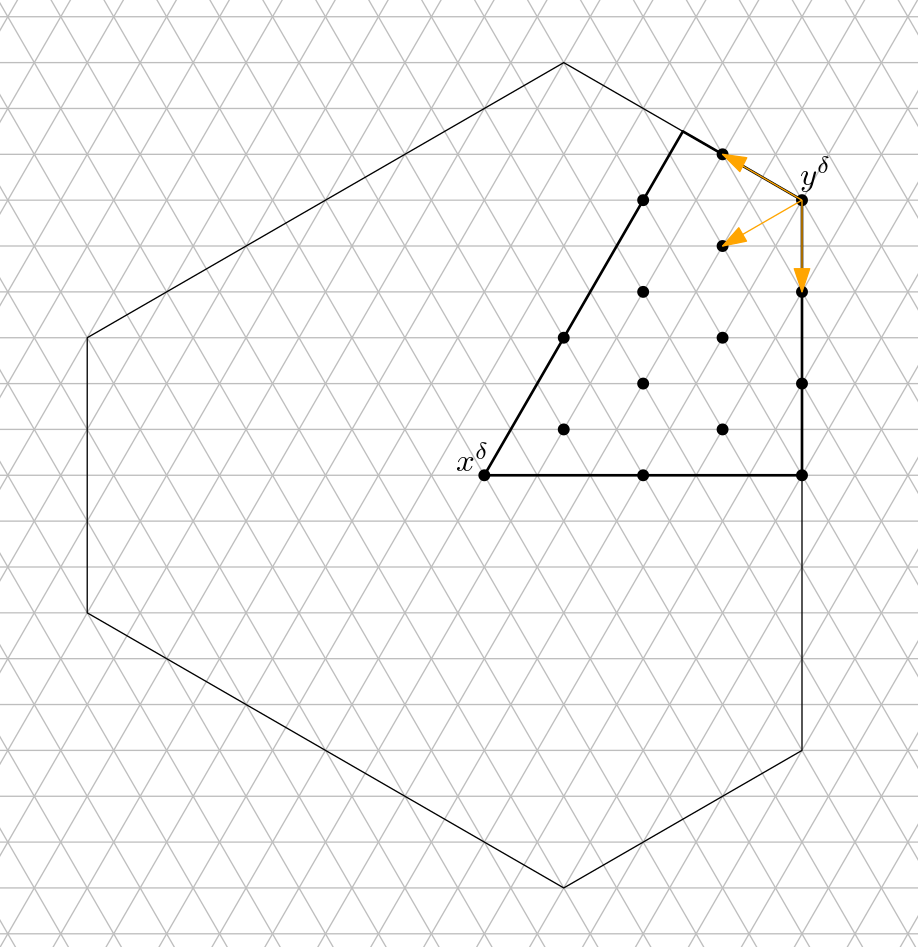}
\label{F:multinomialhex}
\end{center}
\caption{Situation of Lemma \ref{L:comparison}, the bold line marks the relevant quadrilateral, the marked points are the points reachable with $n$ steps from $x$ and the orange arrows are the three steps possible from $y^\delta$. }
\end{figure}
\textbf{Proof of the claim:}

By applying the step of reducing $a_1$ and increasing $a_2$ we see that all such points on the line through $y^\delta$ orthogonal to $1$ are reachable and the same by reducing $a_2$ and increasing $a_3$ for the line orthogonal to $\tau$.
By choosing the correct starting point on these lines any other point in the quadrilateral is reachable by applying the step of reducing $a_1$ and increasing $a_3$.
This proves the claim.

By mirroring this quadrilateral on the lines in directions $1,\tau$ and $\tau^2$ through $x^\delta$ we obtain that also all $z$ in the resulting hexagon satisfy $\P_{x^\delta}^{(0)} (X_n = z)  \ge \P_{x^\delta}^{(0)} (X_n = y^\delta)$.
The vertices of this hexagon are the reflections of $y^\delta$ along those lines.
The points on the boundary of this hexagon which are closest to $x^\delta$ are the projections of $y^\delta$ onto the lines through $x$ in directions $1$ and $\tau$ (and their respective reflections).
Since the angle between those lines is $\frac13\pi$ both of those points have distance from $x^\delta$ of at least $\frac12|x^\delta-y^\delta|$, therefore the disk of radius $\frac12|x^\delta-y^\delta|$ is contained in the hexagon.
This proves the Lemma.
Note that the extreme case of this being the largest disk that fits inside the hexagon is obtained exactly when $y^\delta-x^\delta$ is a multiple of $1, \tau$ or $\tau^2$.
\end{proof}

There are $C |x-y|^2\delta^{-2}(1 + o(\delta^{-1}))$ points verifying the conditions of Lemma \ref{L:comparison}, where $C=\frac{1}{18\sqrt{3}}\pi$.
Consequently we have:
\begin{align*}
 \lfloor|x-y|^2\delta^{-2}\rfloor &=\sum_z\sum_{n=0}^{\lfloor|x-y|^2\delta^{-2}\rfloor}\P_{x^\delta}^{(0)} (X_n = z) \\
 &\geq \sum_{z:|z-x^\delta|\leq\frac12|x^\delta-y^\delta|}\sum_{n=0}^{\lfloor|x-y|^2\delta^{-2}\rfloor}\P_{x^\delta}^{(0)} (X_n = z)\\
  &\geq \sum_{z:|z-x^\delta|\leq\frac12|x^\delta-y^\delta|}\sum_{n=0}^{\lfloor|x-y|^2\delta^{-2}\rfloor}\P_{x^\delta}^{(0)} (X_n = y)\\
& = (C + o(1)) |x-y|^2\delta^{-2} I
\end{align*}
Which implies that $I< 1/C+o(1)$ and thus $I$ is bounded uniformly in $x$ and $y$.

\medskip \noindent \textbf{Bounding $III$.}  From Lemma \ref{L:expexit}, we see that
$$
\mathbb{P}^{(0)}_{x^\delta}(\sigma_{\partial\Omega}>n|X_n=y^\delta) \le \exp ( - c n \delta^2) .
$$
By Lemma \ref{L:HK2d},
 $\P^{(0)}_{x_\delta}(X_n = y^\delta) \le C/n$.
 Hence, crudely bounding
 $(1-\frac{m^2\delta^2}{2})^n$ by 1 in the sum $III$ we get
\begin{align*}
III & \le \sum_{n \ge M \delta^{-2} } \P^{(0)}_{x_\delta}(X_n = y^\delta) \mathbb{P}^{(0)}_{x^\delta}(\sigma_{\partial\Omega}>n|X_n=y^\delta)\\
& \le \sum_{n \ge M \delta^{-2}} \frac{C}n \exp ( - cn \delta^2) \le \sum_{k \ge 1} \sum_{n = 2^k M \delta^{-2}}^{2^{k+1} M \delta^{-2} -1 } \frac{C}{ n} \exp ( - c M 2^{k}) \\
& \le \sum_{k \ge 1} C\exp ( - c M 2^k),
\end{align*}
which is bounded uniformly in $x^\delta$ and $y^\delta$.

\medskip \noindent \textbf{Estimating $II$.}   For $II$ we estimate the number of lattice paths using Stirling's formula.
Assume without loss of generality that $y^\delta=x^\delta+a+be^{2 \pi i/3} = a+ b \tau$ with $a =a^\delta,b = b^\delta\in \{0,1, \ldots\}$ (other cases are similar), then the number of paths from $x^\delta$ to $y^\delta$ is $0$ if $n-a-b$ is not divisible by $3$.
If $n-a-b$ is divisible by $3$, the number of paths is given by the multinomial coefficient: $$ \binom{n}{\frac{n-a-b}3,\frac{n+2a-b}3,\frac{n-a+2b}3}.$$ Now, in the regime $II$,
Applying Lemma \ref{L:multi} we find
\begin{align}
    \P_{x^\delta}^{(0)} (X_n = y^\delta)
    &
    = \frac{\sqrt{27}}{2\pi n} \exp ( - \frac{|x-y|^2}{ \delta^2 n}) (1+O(\delta))
\end{align}

Recall that $P_{x,y}(t)=\mathbb{P}_{x\to y; t }(\sigma_{\partial\Omega}>t)$.
By Corollary \ref{cor:uniform bridge} we get:
\[
\mathbb{P}_x(\sigma_{\partial\Omega}>n|X_n^\delta=y)=P_{x,y}(n\delta^2)(1+o_\delta(1)),
\]
where $o_\delta(1) \to 0$ when $\delta \to 0$, uniformly in $n$ such that $|x-y|^2\delta^{-2} \leq n\leq M\delta^{-2}$ .
Using this we get:
\begin{align}
    II 
    &=\frac{\sqrt{27}}{2\pi} \times \frac13 \times \sum_{n=\lfloor|x-y|^2\delta^{-2}\rfloor+1}^{\lfloor M\delta^{-2}\rfloor}\frac{1}{n} \exp ( - \frac{|x-y|^2}{ \delta^2 n}) P_{x,y}(n\delta^2)(1-\frac{m^2\delta^2}{2})^n(1+O(\delta))(1+o_\delta(1))\\
    =&\frac{\sqrt{3}}{2\pi}(1+o_\delta(1))\sum_{n=\lfloor|x-y|^2\delta^{-2}\rfloor+1}^{\lfloor c\delta^{-2}\rfloor}\frac{1}{n} \exp ( - \frac{|x-y|^2}{ \delta^2 n}) (1-\frac{m^2\delta^2}{2})^nP_{x,y}(n\delta^2),
\end{align}
where the fact $1/3$ in the first line comes from the fact that only one in three terms contribute to the sum (owing to periodicity).

This can be transformed into a Riemann sum, from which we deduce:
\begin{align}
II &
= \frac{\sqrt{3}}{2\pi}(1+o_\delta(1)) 
\int_{|x-y|^2}^M\frac{P_{x,y}(s)\exp ( - \frac{|x-y|^2}{s}) \exp(-m^2s)}{s}ds\label{IIestimate}
\end{align}
The convergence of the Riemann sum is guaranteed by the continuity of the integrand over the relevant interval.

From \eqref{IIestimate} and our bounds on $I$ and $III$ note that $Z^{(m)}_{\Gtree}(x^\delta,\cdot)$ is uniformly bounded in $\delta$ on compacts of $\Omega\setminus \{ x \}$. Using Lemma \ref{lem:Harnack} we deduce that $Z^{(m)}_{\Gtree}(x^\delta,\cdot)$ has subsequential limits in every compact of $\Omega \setminus \{ x \}$. By considering a countable number of such compacts (e.g.
$\Omega_n = \{ y \in \Omega : d(y,x) \wedge d( y, \partial \Omega ) \ge 1/n\}$)
and a standard diagonalisation argument we may assume that there are subsequential limits in all of these compact domains simultaneously, which are necessarily consistent with one another. Let $h(x,\cdot)$ denote any such limit. We aim to identify $h$ uniquely.

As we are interested in the behaviour when $y$ is close to $x$ we can assume that the straight line from $x$ to $y$ is in $\Omega$ and therefore Lemma \ref{lem:Pzero} applies and $P_{x,y}(s)$ approaches $1$ as $s$ goes to $0$.
Elementary computations give the asymptotic behaviour of this integral as $-2\log|x-y|+o(\log|x-y|)$ when $|x-y|\rightarrow 0$.

It is elementary to check that $Z^{(m)}_{\Gtree}(x^\delta,\cdot)$ is a discrete massive harmonic function in the sense of Lemma \ref{L:Dirichlet}. Since the convergence to the limit in the chosen subsequence is uniform, it is not hard to see that we can pass to the limit in the solution of the massive Dirichlet problem of Lemma \ref{L:Dirichlet}, and deduce that $h(x, \cdot)$ is massive harmonic away from $x$. Furthermore, from our estimates above it follows that
\begin{equation}
   h(x,y)=-\frac{\sqrt{3}}{\pi}\log|x-y|+o(\log|x-y|).
\end{equation}
Thus $h$ is the unique function satisfying the desired properties.
Therefore all subsequential limits are the same which proves the desired convergence of the discrete massive Green functions.
\end{proof}

\begin{rem}
The factor $\sqrt{3}$ can be explained as follows:
Just as in the discrete case, the expected time spent by Brownian motion in a disk $B$ is given by the integral of the Green's function.
The random walk considered in this section converges to Brownian motion under the scaling $X_{\lfloor 2t\delta^{-2}\rfloor}$.
Thus, the expected amount of time spent in $B$ of the discrete walk on the scaled lattice should satisfy:
\[
\tfrac12\delta^2\mathbb{E}(|\{n:X_n\in B\}|)=\tfrac{1}{2}\delta^2\sum_{y^\delta\in B\cap\delta\mathbb{T}}Z^m_{\Gtree}(x^\delta,y^\delta)\rightarrow\int_B G^m_\Omega(x,y)dy\,.
\]
This is indeed the case, since the density of points in the square lattice is $\frac{2}{\sqrt{3}}$ and thus the sum converges to the integral after cancelling the $\frac{1}{2}$ from the time change with $\frac{2}{\sqrt{3}}$ from the lattice and the $\sqrt{3}$ from the statement of Proposition~\ref{prop: green convergence}.
Therefore the factor $\sqrt{3}$ in the right hand side of Proposition \ref{prop: green convergence} is consistent with the above.
\end{rem}

\subsection{Convergence
of discrete massive Poisson kernel ratio}\label{section:continuous Martingales}

Given a domain $\Omega$, with a marked interior point $o\in\Omega$, a marked boundary point $a \in \partial \Omega$ (thought of as a prime end of $\Omega$) and an interior point $x \in \Omega$, we define the \textbf{continuous massive invariant Poisson kernel} (with constant mass) as:
\begin{equation}
    P_{\Omega}^{(m)}(x,a)\coloneqq P_{\Omega}(x,a)-m^2\int_{\Omega_t}G_{\Omega}^{(m)}(x,y)P_{\Omega}(y,a)dy.
\end{equation}
Here $P_{\Omega}(y,a) = P_{\D} (\phi_\Omega(y), \phi_\Omega(a))$, where $\phi_\Omega$ is the unique conformal isomorphism sending $\Omega$ to $\D$, $o$ to $0$, and $a$ to 1, and $P_{\D} (x, 1) =\frac{1}{2\pi}\mathrm{Re}\left(\frac{1+x}{1-x}\right)$.
Thus $P_{\Omega}(y, a)$ the (non-massive) continuous invariant Poisson kernel (note however that the discussion in this section will be superseded by the results in Section \ref{S:LERW_general}, where we also define the invariant Poisson kernel carefully).


This definition is motivated by the following crucial identity for the discrete massive Green function
This is the discrete counterpart of the \emph{resolvent identities} to which we will return in Section \ref{subsec:resolvent}, which relate massive and non-massive harmonic functions, as already observed in the work of Makarov and Smirnov \cite{MakarovSmirnov}):
\begin{lem}
\begin{equation}\label{eq:mystery}
    (1-\frac{m^2\delta^2}{2})Z^{(m)}_{\Gtree}(x^\delta,z^\delta)=Z_{\Gtree}(x^\delta,z^\delta)-
   \frac{m^2\delta^2}{2}\sum_{y^\delta\in \Int \Gtree}
    Z_{\Gtree}^{(m)}(x^\delta,y^\delta)Z_{\Gtree}(y^\delta,z^\delta),
\end{equation}
\end{lem}
\begin{proof}
We prove this by splitting each trajectory in the definition of $Z_\Omega$ into two parts, and summing over all possible ways to do so:
\begin{align*}
    \sum_{y^\delta\in \text{Int} \Gtree}
    Z^{(m)}_{\Gtree}(x^\delta,y^\delta)Z_{\Gtree}(y^\delta,z^\delta)
    &= \sum_{y^\delta\in \text{Int} \Gtree}\sum_{k\geq 0}\sum_{\substack{\pi:w\rightarrow z,\\ \pi_k = v}}(\frac{1}{3}(1-\frac{m^2\delta^2}{2}))^k(\frac{1}{3})^{(\#\pi)-k}\\
    &=\sum_{\pi:w\rightarrow z}(\frac{1}{3})^{\#\pi}\sum_{k=0}^{\#\pi}(1-\frac{m^2\delta^2}{2})^k\\
    &=\sum_{\pi:w\rightarrow z}(\frac{1}{3})^{\#\pi}\frac{1-(1-\frac{m^2\delta^2}{2})^{(\#\pi)+1}}{\frac{m^2\delta^2}{2}}\\
    &= \frac{Z_{\Gtree}(x^\delta,z^\delta)-(1-\frac{m^2\delta^2}{2})Z^{(m)}_{\Gtree}(x^\delta,z^\delta)}{\frac{m^2\delta^2}{2}}.
\end{align*}
Rearranging the terms gives the desired result.
\end{proof}


The importance of the invariant Poisson kernel stems from the well-known martingale observable of Lawler, Schramm and Werner \cite{Lawler2001ConformalTrees}. Namely,
let $\gamma^\delta$ be a massive LERW between $o^\delta$ in $\Gtree$ and $a^\delta\in\partial\Gtree$. We parametrise $\gamma^\delta$ from $a^\delta$ to $o^\delta$.
For a vertex $x^\delta \in \Gtree$, define the \textbf{massive martingale observable} as:
\begin{equation}\label{MMO}
    M_{n}^{(m)}(x^\delta)\coloneqq \frac{Z_{\Gtree\setminus\gamma^\delta[0,n]}^{(m)}(x^\delta,\gamma^\delta(n))}
    {Z_{\Gtree\setminus\gamma^\delta[0,n]}^{(m)}(o^\delta,\gamma^\delta(n))}.
\end{equation}
Since $\gamma^\delta(n)$ is on the boundary of $\Gtree\setminus\gamma^\delta[0,n]$, this is equal to the ratio of hitting probabilities of $\gamma^\delta(n)$ from $x^\delta$ vs. $o^\delta$. Proceeding exactly as in \cite[Remark 3.6]{Lawler2001ConformalTrees}, one can check that for every $\delta>0$ and every fixed vertex $x^\delta$, the sequence $(M_n^{(m)}(x^\delta))_{0 \le n \le \sigma}$ gives a martingale (see also \cite[Lemma 7.2.1]{Lawler2013Loop-ErasedWalk}).


The strategy of the proof of convergence of this martingale observable to its continuum limit in Chelkak and Wan \cite{ChelkakWan} is to:
\begin{itemize}

\item first, prove the convergence of the non-massive martingale observable in the non-massive case
(something which was in fact already proved in the radial case by Lawler, Schramm and Werner \cite{Lawler2001ConformalTrees} and generalised by Yadin and Yehudayoff \cite{YadinYehudayoffLERWconvergence}, but in the chordal context of \cite{ChelkakWan} requires some additional justifications); this was proved in Proposition 3.5 and Corollary 3.6 in \cite{ChelkakWan} (and put in the correct chordal framework in Proposition 3.14)

\item second, prove that the ratio of massive Green function to non-massive Green function converges to its continuum limit, which is Proposition 3.15 in \cite{ChelkakWan}.

\end{itemize}

The first step follows directly from the work of Yadin and Yehudayoff \cite{YadinYehudayoffLERWconvergence}, which holds for arbitrary planar graphs subject to convergence of random walk to Brownian motion, which holds on the directed triangular lattice. Therefore only the second step needs to be justified, this is the content of the next lemma (which is the analogue of Proposition 3.15 in \cite{ChelkakWan}).

\begin{lem}\label{PoissonYehudahof}
In the setup above for any $x\in \Omega_t$ and $x^\delta\rightarrow x$ as $\delta\rightarrow 0$, one has:
\[
\frac {Z^{(m)}_{\Gtree_t}(x^\delta,a_t^\delta)}{Z_{\Gtree_t}(x^\delta,a_t^\delta)}\rightarrow \frac{P^{(m)}_{\Omega_t}(x,a_t)}{P_{\Omega_t}(x,a_t)}=
1-m^2\int_{\Omega_t}\frac{P_{\Omega_t}(y,a_t)}{P_{\Omega_t}(x,a_t)}G_{\Omega_t}^{(m)}(x,y)dy.
\]
\end{lem}

\begin{proof}
The proof in \cite{ChelkakWan} works also for the directed triangular lattice, as besides the convergence results of the last section it only requires the identity above, and estimates on the massive Green function, which follow from convergence to Brownian motion.
One might at first be worried as the right-hand side of \eqref{eq:mystery} is not linear in $Z$, whereas the limit of the discrete Green function for the triangular lattice is  $\sqrt{3}$ times massive Green function (see Proposition~\ref{prop: green convergence}.
However, the factor $\sqrt{3}$ combines with the $\tfrac12$ in \eqref{eq:mystery} to make the sum over the triangular lattice to converge to a Lebesgue integral, see the remark after Proposition \ref{prop: green convergence}.
\end{proof}

As a corollary we obtain the following convergence of martingale observables. Fix a subsequential limit $(\gamma_t)_{t \ge 0}$ of massive LERW on the directed triangular lattice, which a priori we know to be a simple curve (by absolute continuity with standard LERW), and parametrise it by capacity. Let $\Omega_t = \Omega \setminus \gamma([0,t])$. Let $a_t = \gamma(t)$ denote the tip of the curve at time $t$, which is on the boundary of $\Omega_t$.

\begin{cor}\label{prop: convergence of martingales}
Fix $r>0$. Suppose $x^\delta
\in B(o^\delta, r/2)$.
For $t \le \log (1/r)$, let $n_{t}$ denote the first $n$ such that the capacity of $\gamma^\delta([0,n])$ viewed from $o^\delta$ exceeds $t$ (equivalently, the conformal radius of $o^\delta$ in $\Omega \setminus \gamma([0,n])$ is less than $e^{-t}$).
\[
M^{(m)}_{n_t}(x^\delta)
\rightarrow \frac{P^{(m)}_{\Omega_t}(x,a_t)}{P^{(m)}_{\Omega_t}(o,a_t)}=:
M^{(m)}_{\Omega_t}(x),
\]
almost surely along the underlying subsequential limit $\delta\to 0$.
\end{cor}

\subsection{Proof of the main statement}

We are now ready to prove convergence to massive SLE$_2$, as stated in Theorem \ref{T:massiveLERWtriangular}.

\begin{proof}[Proof of Theorem \ref{T:massiveLERWtriangular}]
As discussed in Section \ref{section:Absolutecont} the laws of the massive loop-erased random walks are tight and all subsequential limits are absolutely continuous with respect to classical SLE$_2$.
This justifies the application of Girsanov's theorem which in particular implies that the driving function $\xi_t$ is a semi-martingale under $\mathbb{P}^{(m)}$.

Moreover, the discrete martingales of \eqref{MMO} have continuous limits as shown in Proposition \ref{prop: convergence of martingales}.
Writing the martingale in the form
$$
M^{(m)}_n(x^\delta) =
\frac{Z^{(m)}_{\Omega_n^\delta}(x^\delta,a_n^\delta)}
{Z_{\Omega_n^\delta}(x^\delta,a_n^\delta)}
\left(\frac{Z^{(m)}_{\Omega_n^\delta}(o^\delta,a_n^\delta)}
{Z_{\Omega_n^\delta}(o^\delta,a_n^\delta)}\right)^{-1}
\frac{Z_{\Omega_n^\delta}(x^\delta,a_n^\delta)}
{Z_{\Omega_n^\delta}(o^\delta,a_n^\delta)}
$$
with $a_n^\delta = \gamma^\delta(n)$ and $\Omega_n = \Omega \setminus \gamma^\delta([0,n])$, we see that $M^{(m)}(x^\delta)$ is uniformly bounded:
the first term is trivially bounded by 1, the second is bounded by Proposition \ref{P:abscont} (and Koebe's one-quarter theorem), and the third one is bounded for $n \le n_t$ by (uniform) convergence to the (continuous) invariant Poisson kernel (here we use the strength of the result of Yadin and Yehudayoff  \cite{YadinYehudayoffLERWconvergence}) and conformal invariance of the latter.
Hence the limit  in Corollary \ref{prop: convergence of martingales} must also be a martingale (see Remark 2.3 and (2.14) in \cite{ChelkakWan} for the argument).

Standard It\^o calculations together with Hadamard's formula (as outlined in \cite{MakarovSmirnov} and written out in \cite{ChelkakWan}, see Section 4.3 and more specifically Lemma 4.9) for this family of martingales imply that the law of the driving function $\xi_t$ under $\mathbb{P}^{(m)}$ is uniquely determined.
In Sections \ref{subsec:resolvent} to \ref{subsec:identification} we will perform these calculations for a variable mass, one can also refer to them for the (very minor) differences between the chordal and the radial setting.
\end{proof}


\section{Convergence of massive LERW on general planar graphs}\label{S:LERW_general}

In the previous section it was proven that on a triangular and on a square lattice the loop erasure of a random walk with certain weights converges to SLE$_2$ with drift given by a re-weighting of massive SLE$_2$.
This also extended a result from \cite{ChelkakWan} which proved that the loop erasure of a massive random walk on the square lattice converges to massive SLE$_2$ as conjectured in \cite{MakarovSmirnov}.

We now want to extend this result to a general mass profile $\rho$ and to a more general planar graph.
To do this we will combine techniques from \cite{YadinYehudayoffLERWconvergence} and \cite{ChelkakWan}.
In particular we will first argue similarly to \cite{ChelkakWan} that the laws of the LERW are tight and any limit point is absolutely continuous with respect to the law of SLE$_2$.
Then we use the strategy employed by \cite{YadinYehudayoffLERWconvergence} to show that the discrete Poisson kernel ratio converges to the (continuous) invariant Poisson kernel, and finally use this convergence to identify the limiting law again as in \cite{ChelkakWan}.

We will use the following convergence of paths.
For two continuous curves $\alpha,\beta:[0,1]\to \C$, consider the norm $\inf_g\sup_{x\in[0,1]}|\alpha(x)-\beta\circ g(x)|$, where the infimum is over all continuous increasing bijections $g:[0,1]\to[0,1]$.
This is a norm on equivalence classes of continuous curves under reparametrization.
A law $\mu^\delta$ on continuous curves is said to converge weakly to a law $\mu$, if it converges weakly in the topology of this norm.

Recall that a Brownian motion with mass profile $\rho$ is a Brownian motion, which dies at rate $\rho(X_s)$ when at position $X_s$, i.e. it is a process which is absolutely continuous with respect to Brownian motion and has Radon Nikodym-derivative
\begin{equation}\label{E:massRN}
        \left.\frac{d\P_x^{(\rho)}}{d\P_x}\right\rvert_{\mathscr F_t}\,=\exp(-\int_0^t\rho(X_s)ds).
\end{equation}
Note that the total mass of $\P_x^{(\rho)}$ is less than one, so it is not a probability measure but a (finite) measure on paths.

Fix a domain $\Omega$ and two bounded, smooth functions $\rho, \tilde \rho:\Omega\to[0, \infty)$.
Let $\Gtree $ be a sequence of planar graphs embedded in the complex plane, converging to $\Omega$ in the \emph{Carathéodory} topology, with weights $w^\delta$ and a discrete mass function $\rho^\delta$ from the vertices of $\Gtree $ to $\R_+$.
Denote the partition function of the massive random walk with mass $\rho^\delta$ as $Z_{\Gtree }^{(\rho)}(x^\delta,y^\delta)$.
We need to assume the following properties of these objects:
\begin{enumerate}
    \item \label{as:timechange}
    Let $(X^\delta_t, t= 0,1, \ldots)$ be the massive random walk on $\Gtree$, started at a specified vertex $o^\delta\rightarrow o$ of $\Gtree$, with transition probabilities proportional to the directed weights $w^\delta$ and probability to die at each step given by $\rho^\delta(v)=\frac{\delta^2}{2}\widetilde{\rho}(v)+o(\delta^2)$, where the $o(\delta^{2})$ term needs to be uniform in $v$.
    The law of  $(X^\delta_{\delta^{-2} t}, t \ge 0)$ (interpolated continuously between time steps) converges weakly (in the above sense) as $\delta \to 0$, to the law (measure) $\P_o^{(\rho)}$ of a Brownian motion with mass profile $\rho$.
    \item \label{as:Absolutecontinuity} 
    Consider the non-massive random walk on with transition probabilities proportional to the directed weights $w^\delta$ in a domain $\Gtree\subset B(0,R)$ started at $o^\delta$ and conditioned to leave at $a^\delta$.
    Denote with $\sigma^\delta$ the number of steps before hitting the boundary $\partial \Gtree$. 
    There is a constant $c_0$ uniform in $R$ and $\Gtree$, such that
    \begin{equation}
        \E^{(0)}_{\Gtree,o^\delta\to a^\delta}(\sigma^\delta)\leq c_0R^2\delta^{-2}
    \end{equation}
    
    \item \label{as:crossing} The random walk satisfies a \textbf{uniform crossing} assumption:
    Let $\cR$ be the horizontal rectangle $[0,3]\times [0,1]$ and $\cR'$ be the vertical rectangle $[0,1]\times [0,3]$.
    Let $B_1 := B((1/2,1/2),1/4)$ be the \emph{starting ball} and $B_2:= B((5/2,1/2),1/4) $ be the \emph{target ball}. 
    Let $\cR_r = r \cR + z$ (resp. $\cR_r' = r \cR' + z$) for some $r>0$ and $z \in \Omega$, and suppose that $\cR_r$ (resp. $\cR_r'$) $\subset \Omega \subset R\D$. 
    Let $B_1^r , B_2^r$ be the corresponding scaled starting and target balls. 
    Let $\Cross_r$ denote the event that the walk hits $B_2^r$ before leaving the rectangle $\cR_r$ (respectively $\cR'_r$) or getting killed.
    We will say that the graphs ($\Gtree$) satisfying the uniform crossing estimate if there is a constant $c>0$ such that, uniformly over $z \in B_1^r$, uniformly over $r \le R$,
\begin{equation}\label{E:crossing_goal}
\P^{(0)}_z (\Cross_r) \ge c.
\end{equation}
\end{enumerate}

Assumption (i) is essentially an assumption about the fact that random walk converges to a Brownian motion, potentially up to a time-change (as in \cite{YadinYehudayoffLERWconvergence}). Our assumption basically requires that this time-change is not too rough, since before time change the rate of dying is $\tilde \rho(x)$ at a point $x$, whereas after this time-change the rate of dying is given by $ \rho (x)$, and both are assumed smooth and bounded.

The main theorem of this section is:

\begin{thm}\label{T:LERWgeneral_massTEXT}
Let $\Omega$ and $\rho, \tilde \rho$ be as above, and let $\Gtree$, $\rho^\delta$ and $w^\delta$ be such that the assumptions above are satisfied.
Let again $(X^\delta_t, t= 0,1, \ldots)$ be the random walk on $\Gtree$, started at a specified vertex $o^\delta\rightarrow o$ of $\Gtree$, with transition probabilities proportional to the directed weights $w^\delta$ and dying at each step with probability $\rho^\delta$.
Let $\sigma_\delta$ denote the first time at which  $X^\delta$ leaves $\Omega$ (with $\sigma_\delta = \infty$ if $X^\delta$ dies before leaving the domain) and consider the loop erasure  $\LE (X^\delta)$ of the walk up until this time. 
Then conditionally on $\sigma_\delta<\infty$ and $X^\delta_{\sigma_\delta} = a^\delta$, $\LE (X^\delta)$ converges weakly to a radial Loewner evolution $\gamma$, whose driving function $\zeta_t = e^{ i \xi_t}$ (when parametrised by capacity) satisfies the stochastic differential equation
\begin{equation}\label{eq:Lowenerdrift}
    \dd\xi_t=\sqrt{2}\dd B_t+\lambda_t\dd t, \quad \lambda_t=\frac{\partial}{\partial g_t(a_t)}\log\left(\frac{P^{(\rho)}_{\Omega_t}(o,a_t)}{P_{\Omega_t}(o,a_t)}\right),
\end{equation}
where $a_t=\gamma(t)$, $\Omega_t={\Omega}\setminus\gamma([0,t])$ is the slitted domain at time $t$, $g_t$ is the Loewner map from $\Omega_t$ to ${\DD}$ and $P^{(\rho)}_{\Omega_t}$ and $P_{\Omega_t}$ are the \textbf{invariant Poisson kernels} for the Brownian motion with mass profile $\rho$, and regular Brownian motion respectively, in $\Omega_t$.
\end{thm}

The definition of the invariant Poisson kernel $P^{(\rho)}_{\Omega_t}(o,a_t)$ appearing in the theorem is not \emph{a priori} obvious (its construction will be explained in Section \ref{S:PK}). We will obtain more explicit expressions for the drift term $\lambda_t$ in the course of the proof, which will show in particular that when the domains $\Omega_t$ are contiunous with respect to the Carath\'eodory topology (which must be the case a.s. here), $t\mapsto \lambda_t$ is itself continuous. Furthermore, we will see as a result of Lemma \ref{lem:novikov} that $\int_0^\infty \lambda_t^2\dd t  \le C$ for some constant $C>0$, thereby showing that $\lambda_t$ satisfies the Novikov condition and ensuring existence and pathwise uniqueness for solutions to the SDE  \eqref{eq:Lowenerdrift}.


\subsection{Invariant Poisson kernel for Brownian motion with mass}

\label{S:PK}

To describe the scaling limit of loop-erased random walk when the walk itself does not converge to a Brownian motion, but rather to a Brownian motion with mass, it is necessary to first define the invariant Poisson kernel of the latter, and describe a few of its properties. 

Recall that the Poisson kernel of Brownian motion in a domain $\Omega$, when it exists, is the density $h_\Omega(x, a)$ of harmonic measure in $\Omega$ viewed from $x \in \Omega$, with respect to the natural length measure $|da|$ on $\partial \Omega$, evaluated at the point $a\in \partial \Omega$. This quantity is conformally covariant (i.e., if $\phi: \Omega \to \Omega'$ is a conformal isomorphism and is analytic in a neighbourhood of the analytic domain $\Omega$, then $h_{\Omega'} (\phi(x), \phi(a)) = h_{\Omega } (x, a) |\phi'(a)|$.)  For a fixed $o\in \Omega$, the \emph{invariant} Poisson kernel is then the quantity
$$
\lambda_\Omega(x,a) = \lambda^{(0)}_\Omega(x,a)  = \frac{h_\Omega(x, a)}{h_\Omega(o, a)}.
$$
Since we took a ratio, this quantity is now conformally invariant, hence the name invariant Poisson kernel. 

This property of conformal invariance allows us to make sense of the invariant Poisson kernel even when the boundary of the domain is not smooth but $\Omega$ is simply connected: concretely, let $\Omega$ be a simply connected domain and fix $o \in \Omega$ and $a \in \partial \Omega$ (understood as a prime end or equivalently an element of the Martin boundary). Let  $\phi_\Omega $ be the unique conformal isomorphism from $\Omega$ to $\D$ sending $a$ to 1 and $o $ to $0$. Then we define $\lambda_\Omega (x,a) = \lambda_{\D} ( \phi(x),1)$. It is sometimes useful to write $\lambda_\Omega(x,a)$ as a ratio of a function of $x$ and a function of the distinguished point $o \in \Omega$ as follows: namely, if we set
\begin{equation}\label{defPK}
P_\Omega (x,a) = h_{\D} (\phi_\Omega(x),1) = \frac{1}{2\pi}\mathrm{Re}\left(\frac{1+\phi_\Omega(x)}{1-\phi_\Omega(x)}\right)=\frac{1}{2\pi}\frac{1-|\phi_\Omega(x)|^2}{|1-\phi_\Omega(x)|^2}
\end{equation}
(which implicitly also depends on the choice of $o$ and $a$) then $P_\Omega(o,a) = 1/ (2\pi)$ so one trivially has
$$
\lambda_\Omega (x, a) = \frac{P_\Omega(x,a)}{P_\Omega (o,a)}.
$$
Note that when $\Omega$ is a smooth domain, $P_\Omega (x,a)$ is proportional  (but in general not equal) to $h_\Omega(x, a)$. On the other hand, we  have
\begin{equation}
\frac{\P_x ( X_\sigma\in I) }{\P_{o} ( X_\sigma \in I)} \to \lambda_{\Omega} (x, a) = \frac{P_\Omega(x,a)}{P_\Omega (o,a)},
\label{e:PK_rough}
\end{equation}
as $I \subset \partial \Omega$ and $I \downarrow a$, where $\P_x$ is the law of Brownian motion $X$ starting from $x$, and $\sigma$ is the first time that $X$ leaves $\Omega$. To be more precise, $I$ should be thought of as a decreasing sequence of compacts in the Martin boundary of $\Omega$, whose intersection is $\{a\}$, where $a$ is also viewed as an element of the Martin boundary (or, equivalently, a prime end). Equivalently, we may parameterize $\Omega$ by the unit disc $\D$ via its Riemann mapping $\psi_\Omega = \phi_\Omega^{-1}:\D \to \Omega$; then $I = \psi_\Omega (\tilde I)$, where $\tilde I$ is an arc of $\partial \D$ shrinking to $\tilde a\in \partial \D$. When $\Omega$ has a locally connected boundary (which will a.s. hold in the cases where we apply the results below) then the Riemann map $\psi_\Omega$ may be extended from $\bar \D$ to $\bar \Omega$.  For future reference we also denote by $\P_{x\to a}= \P_{x\to a;\Omega}^{(0)}$ the conditional law $\P_x$, given $X_\sigma = a$. (This is obtained by mapping a Brownian motion $\tilde B$ in $\D$, conditioned so that $\tilde B_{\tilde \sigma} = \tilde a$, where $\psi_\Omega(\tilde a) = a,$ and $\tilde \sigma$ is the first exit time from $\D$ by $\tilde B$, and performing the appropriate time-change).

Now suppose $\rho: \bar \Omega \to \R$ is a given smooth real valued function on $\bar \Omega$, and consider the law $\P^{(\rho)}$ of the associated massive Brownian motion with profile $\rho$, i.e. a Brownian motion dying with probability $\rho(X_s)ds$ at each time step.
Without conformal invariance, some arguments are required to construct this invariant Poisson kernel for arbitrary simply connected domains $\Omega$. In fact, various constructions are possible, which we summarise:
\begin{itemize}

\item the approach of Yadin and Yehudayoff \cite{YadinYehudayoffLERWconvergence} can be used directly to show that the left hand side of \eqref{e:PK_rough} forms a Cauchy sequence, and it would be possible to obtain some mild regularity this way.

\item we could use the so-called resolvent identity to \emph{define} the invariant Poisson kernel (this will be discussed in much greater detail below, but going back to the work of Makarov and Smirnov \cite{MakarovSmirnov} and also used extensively by Chelkak and Wan \cite{ChelkakWan}). 

\item by multiplying the standard (i.e., non-massive) invariant Poisson kernel by the appropriate change of measure (``Girsanov'') terms. 

\end{itemize}

In fact all these approaches will play a role in the arguments below and part of the work will be to show these various definitions coincide with one another.
While in \cite{ChelkakWan} the second option is chosen, we have found it simplest to start from the Girsanov approach which gives us a continuous object to work with and for which some minimal regularity can be easily shown. From this we can connect to the discrete picture and separately show that it obeys the appropriate resolvent identity, see Proposition \ref{P:resolvent}.

\begin{thm}\label{T:PKDelta}
Suppose $\Omega$ is bounded and simply connected. As $I \subset \partial \Omega$ shrinks to $a\in \partial \Omega$ (thought of as a prime end or a point on the Martin boundary) \begin{equation}\label{E:PKDelta}
\frac{\P_x^{( \rho)} (\sigma< \infty,  X_\sigma\in I) }{\P^{( \rho)}_{o} ( \sigma<\infty, X_\sigma \in I)} \to \lambda^{(\rho)} (x, a) .
\end{equation}
This is the \textbf{massive invariant Poisson kernel} with profile $\rho$. Furthermore, we can write 
\begin{equation}
\lambda_{\Omega}^{(\rho)} (x,a) = \frac{P^{( \rho)}_{\Omega} (x,  a)}{P^{( \rho)}_{\Omega} (o,  a)}
\end{equation}
where 
\begin{equation}
P^{( \rho)}_{\Omega} (x,  a) = P^{( 0)}_{\Omega} (x,  a)
\displaystyle\E_{x\to a} \left[ \exp \left(  -\int_0^\sigma \rho (X_s) ds \right)\right]\,,
\label{e:PKformula}
\end{equation}
and where $P^{(0)} (x, a) = P(x,a)$ is the above (non-massive) invariant Poisson kernel.
\end{thm}

\begin{rem}
  Note that, despite the name, the massive invariant Poisson kernel is \emph{not} conformally invariant.  We hope this will not create any confusion. Also, unlike in the non-massive case, we point out that $P^{(\rho)} (o,a)$ has a nontrivial value. As before, $x\mapsto P_\Omega^{(\rho)}(x,a)$  coincides up to a constant with the massive invariant Poisson kernel $x\mapsto \lambda_\Omega^{(\rho)}(x,a)$ but the choice of normalisation coming from $P^{(\rho)}$ will turn out to be the right one to have the resolvent identity (Proposition \ref{P:resolvent}) below. 
\end{rem}

\begin{proof}[Proof of Theorem \ref{T:PKDelta}]
For $I \subset \partial \Omega$ (viewed as a subset of the Martin boundary), \begin{align}
    \P_x^{( \rho)} (\sigma< \infty, X_\sigma\in I)  
    & = \E_x^{( 0)} \left[1_{\{ X_\sigma\in I\}} \exp \left(- \int_0^\sigma \rho(X_s) ds \right) \right] \nonumber \\
    & = \P^{(0)}_x ( X_\sigma \in I) \E_x \left[ \left.\exp\left( - \int_0^\sigma \rho(X_s) ds  \right) \right| X_\sigma \in I\right] \nonumber \\
    & = 
    \P_x^{(0)} (X_\sigma \in I) \int_I \E_{x\to y}  \left[\exp\left( - \int_0^\sigma \rho(X_s) ds \right)\right] h^{\#}_I(x, dy) 
    \label{eq:PKGirsanov1} 
\end{align}
where $h^{\#}_I(x, dy)$ is the harmonic measure of Brownian motion restricted to $I$, viewed as a measure on the Martin boundary, and normalised so that it gives $I$ unit mass. 

Suppose now that $I\downarrow a$ in the above sense of Martin boundaries.  The integrand can be rewritten as 
$$ \E_{x\to y}\left[ \exp ( - \int_0^\sigma \rho(X_s) ds )\right] = \E_{\tilde x \to \tilde y} \left[ \exp ( - \int_0^{\tilde \sigma} \rho \circ \psi ( \tilde X_s) \cdot 
 | \psi'(\tilde X_s) |^2 d s) \right] , $$
 where $x = \psi_\Omega (\tilde x),  y =\psi_\Omega( \tilde y) $, $\tilde X$ is a Brownian motion in $\D$ (and the expectation $\E_{\tilde x \to \tilde y} $ corresponds to conditioning a Brownian motion in $\D$ starting from $\tilde x$ to leave $\D$ through $\tilde y$). Suppose $y \to y_0 \in \partial \Omega$ in the sense of the Martin boundary topology; that is, $y = \psi (\tilde y)$ and $\tilde y \to \tilde y_0 \in \partial \D$. To each $\tilde y \in \partial \D$ we can thus associate a law $(X_s)_{0\le s \le \sigma}$ which is a realisation of $\P_{x\to y}$ where $y = \psi( \tilde y)$. Since in $\D$, changing $\tilde y$ amounts to applying a certain M\"obius map (a rotation if $x =o$), one can easily check that we get a coupling of the realisations of $\P_{x\to y}$ such that $\int_0^\sigma \rho(X_s) ds $ converges almost surely (using the boundedness of $\rho$ and the a.s. finiteness of $\sigma$). Since $\rho \ge 0$, we also deduce by the dominated convergence theorem that $\E_{x\to y}[ \exp ( - \int_0^\sigma \rho(X_s) ds )]$ converges as $y \to y_0$, i.e., the left hand side is continuous. 

Thus the integrand in \eqref{eq:PKGirsanov1} is continuous with respect to the natural topology on the Martin boundary of $\Omega$. Using again dominated convergence, it follows that
 $\E_{x}[ \exp ( - \int_0^\sigma \rho(X_s) ds  )|X_{\sigma} \in I]$, converges to $\E_{x \to a}[ \exp ( - \int_0^\sigma \rho(X_s) ds  ) $ as $I \downarrow a$.
Taking the ratio of the right hand side of \eqref{eq:PKGirsanov1} with the same expression but for the starting point $o$, and using \eqref{e:PK_rough}, we conclude that the limit \eqref{E:PKDelta} exists, and is equal to the right hand side of \eqref{e:PKformula}, as desired.
\end{proof}

Note that from the formula \eqref{e:PKformula} a number of features of the invariant Poisson kernel are immediately obvious, such as its continuity with respect to $x$ or $a$, or continuity with respect to the domain in the Carath\'eodory sense.

\subsection{Convergence of discrete Poisson kernel ratio}\label{se:discretePconv}

Recall the discrete setup of Section \ref{S:LERW_general} of a random walks on $\Gtree$ a (general) embedded planar graph, killed with probability $\rho^\delta (v) = \tilde \rho(v) \delta^2/2 + o(\delta^2)$ whenever it is at the vertex $v\in \Gtree \subset \Omega$.
Recall also that $Z^{(\rho)}_{\Gtree} (x^\delta, a^\delta)$ denotes the total mass of random walk paths going from $x^\delta$ to $a^\delta$ without being killed: that is, $Z^{(\rho)}_{\Gtree} (x^\delta, a^\delta) = \P^{(\rho)}_{x^\delta} (\sigma^\delta< \infty, X^\delta_{\sigma^\delta} = a^\delta)$ (note that with a slight abuse of notation we write $Z^{(\rho)}$ when really it depends on $\rho^\delta$ rather than $\rho$). 

\begin{lem}\label{lem:YY}
    For $r>0$ and $\Gtree$ a sequence of subgraphs approximating $\Omega$ all containing a ball of radius $r$ around $o^\delta\to o$ and marked boundary points $a^\delta$, and $x^\delta\to x\in B(o,\frac12r)$ it holds that
    \[
    \frac{Z^{(\rho)}_{\Gtree} (x^\delta,a^\delta)}
    {Z^{(\rho)}_{\Gtree}(o^\delta , a^\delta)}
    \to \lambda_\Omega^{( \rho)} (x, a) = \frac{P^{(\rho)}_\Omega(x,a)}{P^{(\rho)}_\Omega(o,a)}
    \]
\end{lem}

Essentially this is an adapation of the arguments in \cite{YadinYehudayoffLERWconvergence}. We will content ourselves with  describing the instances where changes are needed. Because of this, we feel it is useful to first give a simplified overview of the arguments in \cite{YadinYehudayoffLERWconvergence}, as it may otherwise prove difficult to see why the instances below are indeed the only arguments that need to be changed. In order to go through this we first map $\Omega$ to the unit disc $\D$ (this is both because the proof of convergence to SLE requires mapping everything to a reference domain, and in order to avoid issues related to the distinction between prime ends of a domain and the actual boundary). Thus let $\phi = \psi^{-1}$ be the conformal isomorphism from $\Omega$ to $\D$ sending $o$ to $0$ and $a$ to $1$ (we drop the subscript $\Omega$ from now on).

The first observation of Yadin and Yehudayoff is that ``the exit probabilities are correct'': given a small macroscopic arc $\tilde I$ on $\partial \D$ and $ I = \psi (\tilde I) \subset \partial \Omega$, then the ratio of the probabilities $\P_x^\delta (X_{\sigma^\delta} \in I) /   \P_o^\delta (X_{\sigma^\delta} \in I) $ converges to what one would expect, namely $\P_x( B_\sigma \in I) / \P_o ( B_\sigma \in I)$. This is the content of their Lemma 4.8 and is a more or less obvious consequence of the assumption that random walk converges to Brownian motion, together with planarity. When the arc $I$ is small and close to $a$ (or rather $\tilde I\subset \partial \D$ is small and close to $\tilde a$), this ratio is itself close to the continuum invariant Poisson kernel $\lambda_\Omega(x,a)$ (essentially by definition of the latter, see \eqref{e:PK_rough}).

Next for a boundary point $a\in \partial\Omega$ (again, understood as a prime end) and an interior point $x$, set $\tilde a = \phi(a)$, $\tilde x = \phi(x)$ and $\tilde X = \phi(X)$. They fix a small boundary arc $\tilde I\subset \partial \D$ centered around $\tilde a$ and write
$$
H^\delta(x,a, \Omega) = \P_{\tilde x}^\delta ( \tilde X_{\sigma^\delta} = \tilde a | \tilde X_{\sigma^\delta} \in \tilde I )\P_{\tilde x}^\delta ( \tilde X_{\sigma^\delta} \in \tilde I),
$$
so that it suffices to prove that the ratio
\begin{equation}\label{YYgoal}
\frac{\P_{\tilde x}^\delta ( \tilde X_{\sigma^\delta} = \tilde a | \tilde X_{\sigma^\delta} \in \tilde I )}{\P_o^\delta ( \tilde X_{\sigma^\delta} = \tilde a | \tilde X_{\sigma^\delta} \in \tilde I )} \approx 1
\end{equation}
is close to 1, in the sense that 
\begin{equation}
    \label{E:PKgoal}
\limsup_{\tilde I \downarrow \tilde a} \limsup_{\delta \to 0}    \left| \frac{\P_{\tilde x}^\delta ( \tilde X_{\sigma^\delta} = \tilde a | \tilde X_{\sigma^\delta} \in \tilde I )}{\P_o^\delta ( \tilde X_{\sigma^\delta} = \tilde a | \tilde X_{\sigma^\delta} \in \tilde I )}  - 1\right| = 0 
\end{equation}

The key argument for this is a multiscale coupling, which is implicitly described in Propositions 5.4 -- 5.6. The idea is to consider exponentially growing scales $R_j, j = 1, \ldots, N$ (from microscopic to macroscopic) and points $\xi_j$ in the unit disc at distance of order $R_j$ from both $\tilde a$ and the unit circle, with $R_j \approx e^j r$, and $r$ being the width of the arc $\tilde I = \phi (I)$. At the smallest scale $j=1$, $\xi_j$ is thus at a distance of order $r$ from $\tilde a$ itself, while at the largest scale $j = N$, $\xi_j$ is at a macroscopic distance from $\tilde a$. They condition both walks starting from $x$ and $o$ respectively to leave $\Omega$ through $I$. At each successive scale, there is a positive chance that when the walks get to that scale, they will go and visit the \emph{same} predetermined small ball, chosen to be centered around $\xi_j$ and to have a radius proportional to $R_j$ times a very small constant. Once that is the case, the conditional chances of exiting through $a$ specifically rather than anywhere else in $I$ are necessarily essentially the same for both walks, which proves \eqref{YYgoal}. Essentially, Proposition 5.4 shows that the coupling succeeds with positive probability at each scale independently of previous attempts. Proposition 5.5 shows that the ratio in \eqref{YYgoal} is bounded even in the unlikely event that the coupling never succeeded, and Proposition 5.6 quantifies how close to 1 the ratio in \eqref{YYgoal} once there is a success.

At the discrete level, the only properties of the walks that are needed are planarity (which of course always holds for the random walks considered in this paper) as well as crossing estimates (i.e., \eqref{E:crossing_goal}) and simple consequences of it, such as Beurling estimates. These will be discussed briefly in Appendix \ref{SS:discrete}. At the continuum level the required estimates are described (without proof) in Section 3 of
\cite{YadinYehudayoffLERWconvergence}, mostly Proposition 3.3 to Lemma 3.10. One can see that with very few exceptions, these estimates are properties of Brownian motion which are concerned with typical events of Brownian motion that can additionally be required to hold in a short time scale. In such cases the change of measure between massive (or drifted) and ordinary Brownian motion is harmless, hence these properties also obviously hold true in our situation. The lone exception is Proposition 3.3 (recalled below as Lemma \ref{L:BMhit}, which concerns the probability to hit a very small ball); since this is not a typical event for Brownian motion, one needs to consider the effect of the change of measure and more specifically one needs to check that conditioning on the atypical event does not cause the change of measure to degenerate. This will be carried out in Appendix \ref{SS:continuum}. This concludes our discussion of the proof of Lemma \ref{lem:YY}.

\subsection{Density and absolute continuity with respect to classical SLE\texorpdfstring{$_2$}{2}}\label{S:absolutecontinuity}
In this section we will use assumption \ref{as:Absolutecontinuity} to show that the massive SLE$_2$ with profile $\rho$ is mutually absolutely continuous with respect to classical SLE$_2$ and the Radon Nikodym derivative is bounded.

\begin{prop}\label{prop:partitionbound}
There is a constant $c_0$ such that for every $R$ and every domain $\Gtree\subset B(0,R)$ and internal point $x^\delta$ and boundary point $a^\delta$,
\begin{equation}
    1\geq\frac{Z^{(\rho)}_{\Gtree} (x^\delta, a^\delta)}
    {Z^{(0)}_{\Gtree} (x^\delta, a^\delta)}
    \geq \exp(-c_0R^2\|\widetilde{\rho}\|_\infty)\,.
\end{equation}
\end{prop}
\begin{proof}
    The inequality on the left hand-side is obvious since $1-\rho^\delta \le 1$.
    For the equality on the right hand side note that the ratio can be written as
    \begin{align*}
        \E^{(0)}_{\Gtree,x^\delta\to a^\delta}
        \left(\prod_{s=0}^{\sigma^\delta-1}(1-\rho^\delta(X_s))\right)
        &\geq \E^{(0)}_{\Gtree,x^\delta\to a^\delta} \left[\left(1-\delta^2\|\widetilde{\rho}\|_\infty+o(\delta^2)\right)^{\sigma^\delta}\right]\\
        & \geq \left(1-\delta^2\|\widetilde{\rho}\|_\infty+o(\delta^2)\right)^{\E^{(0)}_{\Gtree,x^\delta\to a^\delta}(\sigma^\delta)}\\
        &\geq \exp(-c_0R^2\|\widetilde{\rho}\|_\infty)\,,
    \end{align*}
    where the second inequality follows from Jensen's inequality, and the last inequality is a direct consequence of assumption \ref{as:Absolutecontinuity}.
\end{proof}
Note that since $$\frac{Z^{(\rho)}_{\Gtree} (x^\delta, a^\delta)}
    {Z^{(\rho)}_{\Gtree} (o^\delta, a^\delta)} = \frac{Z^{(\rho)}_{\Gtree} (x^\delta, a^\delta)}
    {Z^{(0)}_{\Gtree} (x^\delta, a^\delta)} \frac{Z^{(0)}_{\Gtree} (o^\delta, a^\delta)}
    {Z^{(\rho)}_{\Gtree} (o^\delta, a^\delta)}
    \frac{Z^{(0)}_{\Gtree} (x^\delta, a^\delta)}
    {Z^{(0)}_{\Gtree} (o^\delta, a^\delta)}$$ we deduce from this lemma that the left hand side (which is our discrete massive martingale observable) is also bounded provided that $\Omega$ contains a ball of radius $r$ around $o$: indeed, the first two fractions are bounded by the previous Lemma, and the third one is the classical (non-massive) martingale observable, which is bounded as long as $\Omega$ contains a ball of radius $r$ as in \cite{ChelkakWan}, say.

Let $\P^{(\rho)}_{o^\delta \to a^\delta}$ denote the random walk starting from $o^\delta$ with mass $\rho$, conditioned so that $\sigma^\delta < \infty$ and $X^\delta_{\sigma^\delta} = a^\delta$.

\begin{prop}
    Let $$D^{(\rho)}_{\Gtree}(\gamma^\delta):=\frac{\P^{(\rho)}_{o^\delta \to a^\delta}(\mathrm{LE}(X^\delta)=\gamma^\delta)}{\P^{(0)}_{o^\delta \to a^\delta}(\mathrm{LE}(X^\delta)=\gamma^\delta)}$$ 
    be the Radon Nikodym derivative of the Loop erasure of  the $\rho$-massive random walk with respect to the loop-erasure of the regular random walk, conditioned to leave $\Gtree$ at $a^\delta$.
    Then 
    $$D^{(\rho)}_{\Gtree}(\gamma^\delta)\leq \exp(c_0\diam(\Omega )\|\rho\|_\infty)$$ 
    for each $\gamma^\delta$.
    Furthermore 
    \begin{equation}
        \label{eq:boundRN}
\E^{(0)}_{o^\delta \to a^\delta}[\log(D^{(\rho)}_{\Gtree}(\mathrm{LE}(X^\delta)))]\geq -c_0\diam(\Omega )^2\|\rho\|_\infty.
    \end{equation}
\end{prop}
\begin{proof}
    Note that by definition
    \begin{equation}\label{eq:LoopErasureRN-dervative}
    D^{(\rho)}_{\Gtree}(\gamma^\delta)
    =\frac{\sum_{ X^\delta:\mathrm{LE( X^\delta)=\gamma^\delta}}w^{(\rho)}( X^\delta)}
    {\sum_{ X^\delta:\mathrm{LE( X^\delta)=\gamma^\delta}}w^{(0)}( X^\delta)}
    \cdot\frac{Z^{(0)}_{\Gtree} (o^\delta, a^\delta)}{Z^{(\rho)}_{\Gtree} (o^\delta, a^\delta)}\,.
    \end{equation}
    The upper bound follows from the fact that the first fraction is less than $1$ and the bound in Proposition \ref{prop:partitionbound}.
    For the lower bound note that the second fraction is bigger than $1$ and the first one is equal to $\E^{(0)}_{\Gtree}[\prod_{s=0}^{\sigma^\delta-1}(1-\rho^\delta( X^\delta_s))|\mathrm{LE}( X^\delta)=\gamma^\delta]$.
    Taking the logarithm and applying Jensen gives the lower bound
    \[
    \log(D^{(\rho)}_{\Gtree}(\gamma^\delta)\geq\E^{(0)}_{o^\delta \to a^\delta}\left[\left.\log(\prod_{s=0}^{\sigma^\delta-1}(1-\rho^\delta( X^\delta_s)))\right|\mathrm{LE}(X^\delta)=\gamma^\delta\right]\,.
    \]
    Taking the expectation in $\gamma^\delta$ with respect to the distribution of the loop-erased random walk removes the conditioning and gives
    \begin{equation*}
        \E^{(0)}_{o^\delta \to a^\delta}[\log(D^{(\rho)}_{\Gtree}(\mathrm{LE}(X^\delta)))]
        \geq \E^{(0)}_{o^\delta \to a^\delta}[\log(\prod_{s=0}^{\sigma^\delta-1}(1-\rho^\delta( X^\delta_s)))]\geq -c_0\diam(\Omega)^2\|\widetilde{\rho}\|_\infty\,,
    \end{equation*}
    again by assumption \ref{as:Absolutecontinuity}.
\end{proof}

This last proposition implies that the laws of the Loop-erasure of the massive random walks are tight, and any limit point is mutually absolutely continuous with respect to SLE$_2$.
Furthermore, by \eqref{eq:boundRN} and Girsanov's theorem, the Loewner transform of a limit point is driven by a process of the form $\xi_t=\sqrt{2}B_t+2\lambda_t$ and therefore our goal is to identify $\lambda_t$.
See Section 2.6 of \cite{ChelkakWan} for more details.

\subsection{Resolvent identity of the massive invariant Poisson kernel}\label{subsec:resolvent}

Armed with the absolute continuity coming from the previous subsection we now come back to proving properties of the invariant Poisson kernel (and its spatial derivative in the next subsection). 

We fix a boundary point $y = a$, and consider the invariant Poisson kernel $P^{(\rho)}_{D}(x)=P^{(\rho)}(x,a)$ associated just with the mass profile $\rho$ from the previous section.
On a subdomain $\Dprime$ (which will later be $\Omega_t$) with a marked boundary point $a'$ (which will later be $a_t$), we consider $P_{\Dprime}^{(\rho)} (x) = P_{\Dprime}^{(\rho)} (x,a')$ (note that we therefore drop the dependence on $a'$ from our notations).

We aim to establish the following \textbf{resolvent identity} for $P^{(\rho)}_{\Dprime}(x)$, which plays a crucial role in the rest of the proof:

\begin{equation}\label{e:res1}
P_{\Dprime}^{(\rho)} (x) = P_{\Dprime}^{(0)}(x) - \int_{\Dprime} G_{\Dprime}^{(\rho)}(x,y) \rho(y) P_{\Dprime}^{(\rho)}(y) dy,
\end{equation}
where $G_{\Dprime}^{(\rho)}(x,y) = G_{\Dprime}^{(\rho)}(x,y)$ is an appropriate Green function, more precisely the Green function of the Brownian motion with mass profile $\rho$ with Dirichlet boundary conditions in $\Dprime$: that is, 
$$
G_{\Dprime}^{(\rho)}(x,y) = \int_0^\infty p_t^{(\rho)} (x,y) dt
$$
where for $t\ge 0$,
\begin{equation}\label{eq:massGreen}
p_t^{(\rho)}(x,y): = p_t(x,y) \E_{x\to y;t} \left[\exp (-\int_0^t \rho (X_s) ds ) 1_{\{X[0,t] \subset \Dprime\}} \right]     
\end{equation}
and $p_t(x,y)$ denotes the (full plane) transition probabilities for standard Brownian motion.

A trivial but essential property of the massive Green function is that it is bounded by the standard (non-massive) Green function, i.e.
\[
G^{(\rho)}_{\Dprime}(x,y)\leq G^{(0)}_{\Dprime}(x,y)
\]
for all $x,y\in \Dprime$.
This allows us to estimate many integrals simply with their non-massive counterparts.

We are now in a position to prove the resolvent identity \eqref{e:res1}.

\begin{prop}\label{P:resolvent}
  We have
  \begin{equation}\label{e:res}
P_{\Dprime}^{(\rho)} (x) = P_{\Dprime}(x) - \int_{\Dprime} G_{\Dprime}^{(\rho)}(x,y) \rho(y) P_{\Dprime}(y) dy\,.
\end{equation}
\end{prop}

\begin{proof}
    It is clear from the expression of the potential kernel $P^{(\rho)}(x)$ in Theorem \ref{T:PKDelta} that this is $C^2$ in $\Dprime$; the definition as a limit of hitting probabilities shows that it is harmonic with respect to the generator of the massive Brownian motion, i.e. $\cL_\rho := \frac12\Delta-\rho$.
    
    Let us consider the function
    $$
    f(x) = P_{\Dprime}^{(\rho)} (x) - P_{\Dprime}(x). 
    $$
    Our goal is to show that $f (x) = \int_{\Dprime} G_{\Dprime}^{(\rho)} (x,y) \rho(y) P_{\Dprime} (y) dy $ for all $x \in \Dprime$.

    Note that from \eqref{e:PKformula} $f$ is clearly continuous (in fact twice differentiable) in $\Dprime$. Furthermore if $x = x_n $ tends to a point $x' \in \partial \Dprime$ with $x '\neq \aprime$, both terms in $f(x_n)$ tend to zero. We now claim that $f(x)$ is ``negligible compared to the probability of leaving by $\aprime$'' as $x \to \aprime$. Let us explain what we mean by this. Recall the map $\phi:\Dprime  \to \D$ which is the conformal isomorphism sending $o$ to 0 and $\aprime$ to $1$. For small $r>0$, let $A^\D_r$ be the set of points in $\D$ at distance $r$ from $1$ and let $A_r = \phi^{-1} ( A^\D_r)$. Let $\Dprime_r$ be the connected complement of $\Dprime \setminus A_r$ not adjacent to $\aprime$.
    
    \begin{lem}\label{L:exitarc}
For $x \in A_r$ we have $f(x) = o(1/r)$ uniformly. On the other hand, for fixed $x \in \Dprime$, as $r \to 0$ (assume without loss of generality that $r$ is small enough that $x \in\Omega_r$), if $\sigmaprime_r = \inf \{ t \ge 0: X_t \notin \Dprime_r\}$ then there is a constant $C = C(x,\Omega)$ 
such that  
$$
\P^{(0)}_x ( X_{\sigmaprime_r}\in A_r) \le Cr, \quad r >0. 
$$    \end{lem}


\begin{proof}
 Recall that $P_{\Dprime}^{(\rho)}(x)  = P_{\Dprime}(x) \E_{x \to \aprime} [ \exp (- \int_0^{\sigmaprime} \rho(X_s) ds ) ].$
 Note that $P_{\Dprime}(x)$ is exactly conformally invariant by our definition \eqref{defPK}) and so $P_{\Dprime}(x) \le O(1/r) $ uniformly on $A_r$ (using the exact value of the invariant Poisson kernel in the unit disc and the fact that by definition $A_r $ is mapped to $A^\D_r$ by $\phi$ which lies at distance $r$ from 1). Furthermore, uniformly for $x \in A_r$, the expectation $\E_x ( \exp (-\int_0^{\sigmaprime} \rho(X_s)ds ) ) \to 1$ as $x \to \aprime$ (i.e., as $r\to 0$) by dominated convergence since $\sigmaprime\to 0$ in probability (indeed, $\Dprime$ is simply connected hence has a regular boundary) and $\rho (x) \ge 0$. Thus 
 $$
 f(x) = P_{\Dprime}(x) \left(\E_x [ \exp (-\int_0^{\sigma} \rho(X_s)ds)]  -1\right) = o(1/r) 
 $$
as $r\to 0$, as desired. 

For the probability to leave through $A_r$, we simply note that 
using conformal invariance of harmonic measure for Brownian motion, we have that 
$$
\P_x^{(0)} ( X_{\sigmaprime_r}\in A_r ) = \P_{\phi(x)}^{(0)} ( X_{\sigma^\D_r} \in A_r^\D ) \le C(x,\Omega) r
$$
which concludes the proof.
\end{proof}

For the proof of Proposition \ref{P:resolvent} we will need to find a suitable martingale. Recall that $f(x) = P_{\Dprime}^{(\rho)}(x) - P_{\Dprime}(x).$ Note that, since $P_{\Dprime}^{(\rho)}$ is $\cL_{\rho}$ harmonic,
\begin{align*}
\cL_\rho f & = \cL_\rho P_{\Dprime}^{(\rho)} - \cL_\rho P_{\Dprime}\\
& = 0 - ( \frac12 \Delta - \rho) P_{\Dprime}\\
& = \rho P_{\Dprime}. 
\end{align*}
We will call $g(x) =  - \rho(x) P_{\Dprime}(x)$, so that the above identity reads
\begin{equation}\label{eq:harmofg}
\cL_\rho f = - g. 
\end{equation}

\begin{lem}
    For $t \le \sigmaprime$, let $I_t = \int_0^t \rho(X_s) ds$, and let
    $$
    M_t = f(X_t) e^{-I_t} + \int_0^t g(X_s) e^{-I_s}ds. 
    $$
    Then for each $r>0$, $(M_{t \wedge \sigmaprime_r}, t \ge 0)$ is a continuous local martingale under $\P^{(0)}_x$.
\end{lem}

\begin{proof}
    We apply It\^o's formula:
    \begin{align*}
        dM_t &= -(de^{-I_t}) f(X_t) + e^{-I_t} (df (X_t)) + g(X_t)e^{-I_t} dt\\
        & = \text{mart.} + e^{-I_t} \left[ -\rho(X_t) f(X_t) dt + \frac12 \Delta f(X_t) dt + g(X_t) dt 
        \right]\\
        & = \text{mart.} + e^{-I_t} [ \cL_\rho f(X_t) + g(X_t)] dt 
    \end{align*}
    where \text{mart.} is a term denoting a continuous local martingale. Hence, since $\cL_\rho f + g = 0$ by \eqref{eq:harmofg}, $M_t$ is indeed a continuous local martingale up to time $\sigmaprime_r$. 
\end{proof}

   We are now ready to derive a proof of the resolvent identity, i.e., Proposition \ref{P:resolvent}. To do this we apply the optional stopping for $M_t$ at time $t \wedge \sigmaprime_r$ until which $M$ remains bounded, so the application is justified: indeed $f$ is continuous on $\Omega_r$ and has zero boundary conditions on $\partial\Omega_r$ except on $A_r$ where it is uniformly bounded by $o(1/r)$ by  Lemma \ref{L:exitarc}. A similar justification applies to $g$ also. 
    We find, since $M_0 = f(x),$
    \begin{equation}\label{eq:OST}
    f(x) = \E_x \Bigg[ f(X_{t\wedge \sigmaprime_r} ) e^{-I_{t\wedge \sigmaprime_r}} \Bigg] + \E_x \left[ \int_0^{t\wedge \sigmaprime_r} g(X_s) e^{- I_s} ds\right]
    \end{equation}

We will now let $t\to \infty$ and \emph{then} $r\to 0$ in both terms separately to obtain the resolvent identity. We start with the first term, for which we claim the limit as $t\to \infty$ and then $r\to 0$ is simply zero. Indeed, 
    since $r>0$ is fixed and $f$ is bounded $\bar \Dprime_r$, by the dominated convergence theorem we get
$$
\lim_{t\to\infty}\E_x \Bigg[ f(X_{t\wedge \sigmaprime_r} ) e^{-I_{t\wedge \sigmaprime_r}} \Bigg] \le \E^{(0)}_x [ f( X_{\sigmaprime_r})].
$$
      Furthermore recall that $f(x) = 0$ for $x \in \partial\Dprime$ with $x \neq \aprime$, so the only contribution comes from the event where $X_{\sigmaprime_r} \in A_r$: 
      thus
    \begin{align*}
        |\E_x^{(0)} [ f( X_{\sigmaprime_r})] | 
        & \le  \P_x( X_{\sigmaprime_r} \in A_r) \sup_{x \in A_r} |f(x)|\\
        & = O(r) o(1/r) \to 0
    \end{align*}
    by Lemma \ref{L:exitarc}.

    It remains to deal with the second term, which is of the form 
    $$
    \int_{\Dprime_r} g(y) G^{(\rho)}_{t,r}(x,y) dy,
    $$
    where $G^{(\rho)}_{t,r}(x,y)$ is the Green function for Brownian motion up to time $t$, weighted by $\exp ( -I_t)$, and stopped when leaving $\Dprime_r$, i.e.,  
    $$
    G^{(\rho)}_{t,r} (x,y) = \int_0^t p_{s,r}^{(\rho)} (x,y) ds; \quad p^{(\rho)}_{s,r}(x,y) = p_s(x,y) \E_{x\to y;s } [e^{-I_s} 1_{\{ X[0,s] \subset \Dprime_r\}}]
    $$
Letting $t \to \infty$, there is no problem (by monotone convergence) in showing that $G_{t,r}^{(\rho)}$ converges pointwise to $G_r^{(\rho)}(x,y)$, the massive Green function in the domain $\Dprime_r$. Subsequently letting $r\to 0$, there is for the same reason no problem in showing that this converges monotonically to $G^{(
\rho)}_{\Dprime}(x,y)$. Thus, for all $t < \infty$ and $r>0$,
$$
G^{(\rho)}_{t,r} (x,y) \le G^{(\rho)}_{\Dprime} (x,y) \le  G^{(0)}_{\Dprime}(x,y)\,.
$$

Furthermore, observe that since  $|g(y)| = \rho(y) P_{\Dprime} (y)$ and $\rho $ is bounded,
\begin{equation}\label{420}
\int_{\Dprime} G^{(0)}_{\Dprime}(x,y) |g(y)|dy  \le C \int_{\Dprime} G^{(0)}_{\Dprime}(x,y) P_{\Dprime}(y) dy<\infty.
\end{equation}
To see the finiteness of the right hand side, observe that both terms in the integral on the right hand side are conformally invariant, and that after mapping by the conformal isomorphism $\phi$,  $P_{\D} (\phi(y)) \sim 2/ |\phi(y) - 1|$ as $y \to \aprime $, while $G_{\D}^{(0)} (\phi(x), \phi(y)) \sim |1- \phi(y)|$, so that the integrand is bounded in the neighbourhood of $y = \aprime$. Elsewhere, $P_{\Dprime} (y)$ is bounded, and the Green function $G_{\Dprime} (x,y)$, while having a singularity at $y = x$, is clearly integrable over $\Dprime$.

Consequently the assumptions of the dominated convergence theorem are satisfied, and we deduce that 
$$
\lim_{r\to 0}\lim_{t\to \infty} \E_x \left[ \int_0^{t\wedge \sigmaprime_r} g(X_s) e^{- I_s} ds\right] = \int_{\Dprime} G^{(\rho)}_{\Dprime}(x,y) g(y) dy.
$$
Plugging into \eqref{eq:OST}, we therefore obtain:
$$
f(x) = \int_{\Dprime} G^{(\rho)}_{\Dprime}(x,y) g(y) dy,
$$
which is the desired identity. This concludes the proof of the resolvent identity (Proposition \ref{P:resolvent}).
\end{proof}

There is also a resolvent identity for the Green function $G^{(\rho)}$ itself.

\begin{prop}\label{P:Gresolvent}
  We have
  \begin{equation}\label{e:Gres}
G_{\Dprime}^{(\rho)} (x,z) = G_{\Dprime}(x,z) - \int_{\Dprime} G_{\Dprime}(x,y) \rho(y) G_{\Dprime}^{(\rho)}(y,z) dy.
\end{equation}
\end{prop}
\begin{proof}
    Since $G^\rho$ is also $\cL_\rho$-harmonic in both variables, and $G$ is harmonic, the proof proceeds essentially along the same lines. The difference is that we need to replace Lemma \ref{L:exitarc} by the following control over $f(z) := G^{(\rho)}_{\Dprime} (x,z) - G^{(0)}_{\Dprime}(x,z)$ near $z = x$:

 \begin{lem}
    \label{lem:Gcontrol} Let $f$ be as above. Then as $z \to x$,
    $$
    f(z) = o ( \log |x-z|^{-1}).
    $$
 \end{lem}   

    \begin{proof}
Since $G^{(\rho)}_\Omega(x,z) \le G^{(0)}_\Omega(x,z)$ we only need a lower bound on the massive Green function. This is easily obtained: for any $\eps>0$, 
\begin{align*}
    G^{(\rho)}_\Omega(x,z) & = \int_0^\infty p^{(\rho)}_t(x,z) dt \\
    & \ge \int_0^\eps p^{(0)}_t(x,z)\E_{x\to z;t}(e^{-I_t}) dt  \\
    & \ge e^{-\eps \|\rho\|_\infty} \left( G^{(0)}( x, z) - \int_\eps^\infty p_t^{(0)}(x,z) dt \right)\\
    & \ge (1- \eps \|\rho\|_\infty) \left(G^{(0)}(x,z) - O(1) \right)
\end{align*}
where we used the easy consequence of Beurling's estimate that for a simply connected domain $\Omega$, $p_t(x,z) \le t^{-1 - \eta}$ for some $\eta>0$. The lemma follows.
    \end{proof}
The rest of the proof of Proposition \ref{P:Gresolvent} proceeds exactly as in the proof of Proposition \ref{P:resolvent}.
    \end{proof}

From this we can deduce a massive version of Hadamard's formula (see \cite[Lemma 4.7]{ChelkakWan} for the case of constant mass). Let $(K_t)_{t\ge 0}$ be a growing family of compact $\Omega$-hulls in $\Omega$, growing from $a$ to the inner point $o\in \Omega$ and having the locality property (see e.g. \cite{Berestycki2014LecturesEvolution} for a definition of these terms), generated by a continuous curve $\gamma_t$ growing in $\Omega$ from $a$ to $o$. Let $\Omega_t = \Omega\setminus K_t$, which is a monotone decreasing family of subdomains of $\Omega$, and let $a_t$ be the point on the (Martin) boundary of $\Omega_t$ corresponding to $\gamma_t$. We will assume that $\gamma[0, \infty)$ has Lebesgue measure zero. 

Let $G_t^{(\rho)} = G_{\Omega_t}^{(\rho)}$ be the massive Green function in $\Omega_t$ and let $P_t^{(\rho)} = P_{\Omega_t}^{(\rho)}$ be the massive invariant Poisson kernel (defined in Theorem \ref{T:PKDelta}, associated with the boundary point $a_t$). Since $\Omega_t$ is monotone decreasing, it is obvious that for each fixed $z\neq x$ in $\Omega$, $t\mapsto G_t^{(\rho)} (x,z)$ is monotone decreasing until the first time $t$ such that either one of $x$ or $z$ is in $K_t$. The massive Hadamard identity expresses the derivative of $G_t^{(\rho)}$ in terms of product of massive invariant Poisson kernels. Intuitively, this is because the paths from $x$ to $z$ that are lost between times $t$ and $t + dt$ can be decomposed into two portions, one from $x$ and one from $z$, which go via the tip of the curve $a_t$. (In the case of constant mass, this is stated without proof within the proof of Proposition 3.1 by Makarov and Smirnov \cite{MakarovSmirnov}; the argument below is close to the proof given by Chelkak and Wan \cite[Lemma 4.7]{ChelkakWan}).

\begin{prop}\label{l:massHadamard}
For each fixed $x \neq z$, such that $x, z \in \Omega_t$, the massive Green function $s\mapsto G^{(\rho)}_s(x,z)$ is differentiable at $s =t$, and
\begin{equation}
    \partial_tG_t^{(\rho)}(x,z)=-\pi P_t^{(\rho)}(x)P_t^{(\rho)}(z),
\end{equation}
    where $P_t^{(\rho)}$ is defined above.
\end{prop}
\begin{proof}
    Since the mass $\rho$ is nonnegative (or more precisely by \eqref{eq:massGreen}), for $s<t$, 
    $$0\le G^{(\rho)}_s (x,z) - G^{(\rho)}_t(x,z) \le  G^{(0)}_s (x,z) - G^{(0)}_t(x,z)$$
    so the increments of $G^{(\rho)}$ are bounded by those of $G^{(0)}$. Note that $s\mapsto G^{(0)}_s (x,z)$ is differentiable by the classical (non-massive) Hadarmard formula (see, e.g., \cite{SS}) and $s\mapsto G_s^{(\rho)}$ is differentiable a.e. by monotonicity. 
    
    Therefore $\partial_tG_t^{(\rho)}(x,z)$ exists for all $x,z\in \Omega_t$ and it also holds that
    \begin{equation}\label{e:deriv-comp}
        0\leq -\partial_tG_t^{(\rho)}(x,y)\leq -\partial_tG_t^{(0)}(x,y)=\pi P_t(x)P_t(y)<\infty.
    \end{equation}
    Recall from the Green function resolvent identity that 
    $$
    G_t^{(\rho)} (x,z) = G^{(0)}_t(x,z) - \int_{\Omega_t} G^{(0)}_t(x,y) \rho(y) G_t^{(\rho)}(y,z) dy 
    $$
    Since $\gamma[0, t]$ has Lebesgue measure equal to zero we can replace the domain of integration in the above integral by $\Omega$. 

Differentiating this identity (\ref{e:Gres}), using the classical Hadamard formula and the resolvent identity for the massive invariant Poisson kernel (Proposition \ref{P:resolvent}), we obtain:
\begin{align}
    \partial_tG_t^{(\rho)}(x,z) &=-\pi P_t(x)P_t(z)
    +\pi\int_{\Omega} P_t(x)P_t(y) \rho(y) G_t^{(\rho)}(y,z) dy \nonumber \\
& \quad \quad \quad     -\int_{\Omega} G_{\Dprime}(x,y) \rho(y) \partial_t G_t^{(\rho)}(y,z) dy \nonumber \\
    &=-\pi P_t(x)P_t^{(\rho)}(y) - \int_{\Omega} G_{\Dprime}(x,y) \rho(y) \partial_t G_t^{(\rho)}(y,z) dy.\label{423}
\end{align}
Differentiation under the integral is justified because both $G_t$ and $G_t^{(\rho)}$ are decreasing in $t$ so their product is also monotone, and we can then use the positive case of the Fubini theorem (i.e., the Tonelli theorem) as well as the fundamental theorem of calculus to conclude.

Consider now the integral operators $\mathfrak{G}_t$ and $\mathfrak G_t^{(\rho)}$ acting on an arbitrary function $h:\Omega_t\to\mathbb R$ by
\begin{align}
    (\mathfrak{G}_t h)(x):=\int_{\Omega}G_t(x,y)\rho(y)h(y)dy\text{ and}\\
    (\mathfrak{G}_t^{(\rho)} h)(x):=\int_{\Omega}G_t^{(\rho)}(x,y)\rho(y)h(y)dy\,,
\end{align}
whenever the integrals above converge. 
Using these operators we can rewrite \eqref{423} as
\begin{equation}\label{426}
    (\mathrm{Id}+\mathfrak{G}_t)(\partial_t G_t^{(\rho)}(\cdot,z))=-\pi P_t(\cdot)P_t^{(\rho)}(z).
\end{equation}
Again by the resolvent identity (\ref{e:Gres}) we will see that 
\begin{equation}
    (\mathrm{Id}-\mathfrak{G}_t^{(\rho)})(\mathrm{Id}+\mathfrak{G}_t)h=h,
\end{equation}
whenever all involved integrals are absolutely convergent.
Indeed the left hand side, evaluated at $x$, expands as
\begin{align*}
    &h(x)-\int_{\Omega}G_t^{(\rho)}(x,y)\rho(y)h(y)\dd y
    +\int_{\Omega}G_t(x,y)\rho(y)h(y)\dd y
    \\
    &\quad\quad-
    \iint_{\Omega\times\Omega }G_t^{(\rho)}(x,z)\rho(z)G_t(z,y)\rho(y)h(y)\dd y \dd z\\
    &=h(x)+\int_{\Omega} \rho(y)h(y)\left(G_t(x,y)-G_t^{(\rho)}(x,y)-
    \int_{\Omega} G_t^{(\rho)}(x,z) \rho(z) G_t(z,y)\dd z\right)\dd y\\&=h(x),
\end{align*} 
since the bracketed term in the integral vanishes due to the Green function resolvent identity (\ref{e:Gres}) and reversibility of $G_t(z,y)$.

 For $h(x) = \partial_tG_t^{(\rho)}(x,z)$ the absolute convergence of the integrals is justified by, respectively: (\ref{e:deriv-comp}) and the finiteness of the integral in \eqref{420}, and an application of (4.6) in \cite{ChelkakWan} (also recalled below explicitly in \eqref{eq:CWboundPQ3}) together with the fact that $\int_\Omega G_t(z,y) dy$ is bounded uniformly in $z\in \Omega$ by a constant depending only on the diameter of $\Omega$.
 
Applying \eqref{426} and Proposition \ref{P:resolvent} yields
\begin{align*}
    \partial_tG_t^{(\rho)}(\cdot,z)&=
    (\mathrm{Id}-\mathfrak{G}_t^{(\rho)})(\mathrm{Id}+\mathfrak{G}_t)\partial_tG_t^{(\rho)}(\cdot,z)\\
    &=(\mathrm{Id}-\mathfrak{G}_t^{(\rho)})(-\pi P_t(\cdot)P_t^{(\rho)}(z))\\
    &=-\pi P_t^{(\rho)}(\cdot)P_t^{(\rho)}(z).
\end{align*}
This proves the statement.
\end{proof}
\subsection{Derivative resolvent identity}

We continue with the setup introduced above Proposition \ref{l:massHadamard}. We now assume further that $K_t$ is a growing family of hulls generated by a continuous curve $\gamma_t$, and $\Omega_t=\Omega\setminus K_t$.
Below we will in fact need to assume that this curve is either non-massive SLE$_2$ or something absolutely continuous with respect to it.

We define the (radial) derivative kernel $Q_t$ in $\Omega_t$ by setting for $y \in\Omega_t$,
$$
Q_{t}(y) = Q_{t}(y, a_t) = \mathrm{Im} \left( \frac{2\phi_t (y)}{(1- \phi_t(y))^2}\right),
$$
where that $\phi_t: \Omega_t \to \D$ is the unique conformal isomorphism mapping $o$ to $0$ and $a_t$ to 1. (This is the radial analogue of (4.1) in \cite{ChelkakWan}).
The reason for introducing this radial derivative kernel $Q_{t}$ is that, writing $P_t$ for the potential kernel ratio $P_{\Omega_t}$,
$$
dP_t(x) = Q_t(x) d\xi_t.
$$
In the chordal case this is an easy application of It\^o's formula which can be seen for instance from Proposition 7.7 in \cite{Berestycki2014LecturesEvolution}.
In the radial case this calculation is slightly more involved but essentially similar.
To calculate $\dd P^{(\rho)}_t(x)$ we need to define the massive version of $Q_t$.
For this it is simpler to \emph{define} $Q^{(\rho)}_{t}$ via its associated resolvent identity: namely,  we set
\begin{equation}\label{eq:deriv_resolvent}
Q_{t}^{(\rho)}(x) = Q_{t}(x) - \int_{\Omega} G_{t}^{(\rho)} (x,y) \rho(y) Q_{t}(y) dy. 
\end{equation}
To make this definition we need to check that the integral appearing on the right hand side is finite; we will check this is the case for almost every time (at the moment we still do not need to assume that $K_t$ is the hull of an SLE so the estimate is deterministic). Since $G^{(\rho)}_{t} (x,y) \le  G_{t}(x,y) $ it suffices to prove the same with $G_{t}^{(\rho)}$ replaced with the ordinary Green function $G_{t}$ in $\Omega_t$.
This is done in \cite[Corollary 4.6]{ChelkakWan} for the chordal case.
What remains to be checked is that these arguments can be carried over to the radial case, but for completeness we will also repeat the rest of the argument.

\begin{prop}\label{P:derivative}
For any fixed Loewner chain we have the following estimate. For all $x \in \Omega$, for almost every time $t\ge 0$,  $\int_{\Omega} G_{t} (x,y) Q_{t}(y) <\infty $. In particular this holds almost surely for the Loewner chain driven by $(\xi_t)_{t\ge 0}$.
\end{prop}

\begin{proof}
Using \cite[Lemma 4.1]{ChelkakWan} and expressing our $Q$ in the upper-half plane via a suitable M\"obius map (and conformal invariance of $P$), it is easy to check that
\begin{equation}\label{eq:CWboundsPQ}
\left| \frac{P_t(y)}{P_t(x)} - 1\right| G_{t} (x,y) \le C, \quad \left| \frac{Q_t (y)}{P_t(y)} - \frac{Q_t(x)}{P_t(x) } \right| \frac{G_{t} (x,y)}{P_t(x)} \le C
\end{equation}
for some uniform constant $C>0$ independent of anything. In particular,  
\begin{equation}
        \label{eq:CWboundPQ3}
    P_t(y) G_t(y,x) \le P_t(x) G_t(y,x) + C P_t(x)
\end{equation}
and
\begin{equation}
    \label{eq:CWboundPQ2}
    |Q_t (y) | G_{t} (x,y) \le C P_t(x) P_t(y) + |Q_t(x) | G_{t} (x,y) +  C |Q_t(x)|. 
\end{equation}
When we integrate over $y \in\Omega$, the third term does not depend on $y$ and therefore is integrable, the second term depends on $y$ only through $G_{t} (x,y)$ but using the fact that $G_{t} (x,y) \le G_\Omega(x,y)$ which has only a logarithmic singularity at $x$, it is easy to see that this term too is integrable for all $t\ge 0$. The problematic term is the first term. The proposition follows if we can prove
\begin{equation}\label{eq:propgoalPQ}
\int_\Omega P_t(y) dy < \infty
\end{equation}
for almost every $t\ge 0$, almost surely. In fact, we will check 
\begin{equation}
    \label{eq:propgoalPQ2}
    \int_0^\infty \left[ \int_{\Omega_t} P_t (y) dy \right]^2 dt  \le C(\Omega)< \infty
\end{equation}
where $C(\Omega)$ depends only on $\Omega$ (this is the analogue of Corollary 4.6 (i) in \cite{ChelkakWan}). This will obviously imply \eqref{eq:propgoalPQ} and thus Proposition \ref{P:derivative}. This however follows immediately from the classical Hadamard formula since 
$$
\left[\int_{\Omega_t} P_t(y)dy\right]^2 = \iint_{\Omega} P_t(x) P_t(y) dx dy = - \frac1{\pi} \iint_\Omega \partial_t G_t(x,y) dx dy.
$$
Thus integrating over $t>0$ we get 
$$
\int_{t>0} \left[\int_{\Omega_t} P_t(y)dy\right]^2 dt = \frac1{\pi}\iint_\Omega G_\Omega(x,y) - G_{\infty} (x,y) dx dy  \le \frac1{\pi} \iint G_\Omega(x,y) dx dy < \infty,
$$
since $\Omega$ is bounded.
\end{proof}

For our subsequent use of the radial derivative kernel in It\^o's formula we need a strengthening of \eqref{eq:propgoalPQ2}, which is the analogue of Corollary 4.6 (ii) in \cite{ChelkakWan}. This is the key estimate, and requires us to make one additional assumption compared to the general setup introduced above Proposition \ref{l:massHadamard}, We will stop assuming that $(\gamma_t)_{t\ge 0}$ is deterministic and arbitrary, and instead assume it is random, absolutely continuous on compact intervals of time $[0,T]$ with respect to SLE$_\kappa$ for some $\kappa\le 4$. (In our application in this article, $\gamma$ will be either SLE$_2$ or the inhomogeneous massive SLE$_2$, so these assumptions will be satisfied). 

\begin{lem}\label{l:quadint}
    Almost surely, for any fixed $T>0$, 
    $$
    \int_0^T \int_{\Omega_t} P_t(x)^2 \dd x \dd t < \infty\,.
    $$
\end{lem}

In view of the nature of the singularity of $P_t(x)$ near $a_t$, such a result might seem surprising initially.

\begin{proof}
    Fix $x \in \Omega$ and suppose $t < \sigma_x$. Let $y$ be sufficiently close to $x$ that $y \in \Omega_t$, but $y \neq x$. We know that we may write the Green function 
    $$
    G_t(x,y) = -\frac1{\pi} \log |x-y| + h_t(y),
    $$
    where $h_t(y)$ is harmonic in $y \in \Omega_t$ (including at $y =x$) and $h_t(x) = 1/(\pi) \log \crad (x, \Omega_t)$ (see, e.g., Theorem 1.23 in \cite{BP}). The left hand side is (for instance by Hadamard's formula) differentiable in $t$; since in the right hand side only $h_t(y)$ depends on $t$ we get that $t\mapsto h_t(y)$ is differentiable and
    $$
    \partial_t h_t(y) = \partial_t G_t(x,y) = - \pi P_t(x) P_t(y).
    $$
    We want to take $y\to x$ and this requires an exchange of derivation of limit. This can be done using the fact that $h_t(x)$ is harmonic and thus satisfies the mean value property:
    $$
    \frac{h_{t+\delta} (x) - h_t(x)}{\delta} = \int_y \frac{h_{t+\delta} (y) - h_t(y)}{\delta} s(\dd y)
    $$
    where $s(\dd y)$ is the uniform law on some given circle centred at $x$ of arbitrary sufficiently small positive radius. Pointwise,  as $\delta \to 0$, the terms in the integral converge to $P_t(x) P_t(y)$ by the above. The assumptions of the dominated convergence theorem are satisfied by the mean value theorem and the fact that $t\mapsto P_t(x)P_t(y)$ is continuous at a given time $t$ so long as $y\in \Omega_t$. We deduce, using harmonicity of $P_t(y)$:
    $$
    \frac1{\pi}\partial_t  \log \crad (x,\Omega_t) = -\pi P_t(x)^2,
    $$
    for any $x \in \Omega_t$. Therefore, 
    $$
    \int_0^T P_t(x)^2 \dd t = \frac1{\pi^2} \log \frac{\crad (x,\Omega)}{\crad (x, \Omega_T)},
    $$
    for any $x \in \Omega_T$ (note the unimportant difference of a factor of $\pi/2$ with respect to the proof of Corollary 4.6(ii) in \cite{ChelkakWan}, which comes from a different choice of normalisation of the Laplacian and what appears to be a typo).

    This can be integrated over $x \in \Omega_T$ and even $x\in \Omega$ when we set the integrand to be infinity on $K_T$; since $K_T$ has Lebesgue measure a.s. equal to zero this does not make a difference. We get
    $$
    \int_{\Omega_T} \int_0^T P_t(x)^2 \dd t \dd x= \frac1{\pi^2} \int_{\Omega_T} \log \frac{\crad (x, \Omega)}{\crad (x, \Omega_T)} dx.
    $$
    Using Fubini's theorem (since the integrand is positive) we can exchange the space and time integration on the left hand side. Taking expectations, we further obtain: 
    $$
    \E \left[ \int_0^T \int_{\Omega_T} P_t(x)^2 \dd x \right] = \frac1{\pi^2} \int_\Omega \E \left[ 1_{x\in \Omega_T} \log \frac{\crad(x, \Omega) }{\crad (x, \Omega_T)} \right] \dd x
    $$
    In the left hand side there is no difference if we replace $\Omega_T$ by $\Omega_t$ (since the difference has zero Lebesgue measure a.s.) and in the right hand side for the same reason we can ignore the indicator. The result follows since the expectation on the right hand side is finite. Indeed, much stronger bounds are known than (negative) logarithmic moments for $\crad (x, \Omega_T)$: it is known that $\P( \dist (w, K_T) \le \eps) \le C(T) \eps^{1-\kappa/8}$ with $\kappa\le 4$ (see, e.g., Proposition 4 in \cite{Beffara} for the chordal case, but the argument easily generalises to the radial case). This gives  polynomial hence the desired logarithmic moments using the Koebe one-quarter theorem.  
\end{proof}

This result implies a lemma (``stochastic Fubini'') corresponding to Lemma 4.8 in \cite{ChelkakWan}:
\begin{lem}\label{l:stochFubini}
    The process $t\mapsto \int_{\Omega}G^{(\rho)}_t(x,y)Q_t(y)\dd y$ is a local semi martingale in the filtration of the driving function $(\xi_t)_{t\ge 0}$.
    Moreover almost surely, for all $T>0$ the following identity is satisfied:
    \begin{equation}
        \int_\Omega\int_0^TG^{(\rho)}_t(x,y)Q_t(y)\dd\xi_t\dd y=\int_0^T\int_\Omega G^{(\rho)}_t(x,y)Q_t(y)\dd y\dd\xi_t
    \end{equation}
\end{lem}
This follows from Proposition~\ref{P:derivative} and Lemma~\ref{l:quadint} in the same way as in \cite{ChelkakWan}.

\subsection{Identification of LERW limit: proof of Theorem \ref{T:LERWgeneral_massTEXT}}
\label{subsec:identification}

In this section we complete the proof of Theorem \ref{T:LERWgeneral_massTEXT} by showing that the limit of loop-erased random walk on the triangular lattice, in the case where the walk itself converges to inhomogeneous massive Brownian motion, exists and is given by a Lowener evolution whose driving function $\xi$ satisfies \eqref{E:Loewner_off}. 

Let us summarise the situation at this stage. As already mentioned at the end of 
Section \ref{S:absolutecontinuity}, we know that subsequential limits of the loop-erasure exist (i.e. the laws of the loop-erasure are tight), and it suffices to identify any subsequential limit uniquely. We also know, again from the same discussion, that any subsequential limit is absolutely continuous with respect to radial SLE$_2$, and may be described by a radial Loewner evolution whose driving function $\xi$ satisfies 
\begin{equation}\label{eq:drivingSDE}
\dd\xi_t = \sqrt{2}dB_t + 2 \lambda_t \dd t.    
\end{equation}
Our goal is thus to identify $\lambda_t$ (we will show that $\lambda_t = \frac{Q_t^{(\rho)}(o)}{P_t^{(\rho)}(o)}$) and show that the above SDE has a unique weak solution (we will in fact get strong pathwise uniqueness).

\begin{proof}[Proof of Theorem \ref{T:LERWgeneral_massTEXT}] We know by a classical argument, see \cite[Remark 3.6]{Lawler2001ConformalTrees}, that for every vertex $x^\delta$, we get a \textbf{discrete martingale observable} $M_n^{(\rho)} (x^\delta)$ defined by
\begin{equation}
M_n^{(\rho)}(x^\delta) = \frac{Z^{(\rho)}_{\Omega^\delta \setminus \gamma^\delta [0,n]} (x^\delta, \gamma^\delta(n))}{Z^{(\rho)}_{\Omega^\delta \setminus \gamma^\delta [0,n]} (o^\delta, \gamma^\delta(n))}.
\end{equation}

Applying Lemma \ref{lem:YY} (which, as already explained, extends to our setup one of the main results of Yadin and Yehudayoff \cite{YadinYehudayoffLERWconvergence}), we see that if we take $x^\delta \to x \in \Omega$ and parameterize $\gamma^\delta[0,n]$ by capacity (which requires taking $n = n^\delta(t)$ for any given $t>0$) then, assuming $z \in \Omega_t$, we have as $\delta \to 0$, for each $t>0$,
\begin{equation}\label{eq:convPKmart}
\frac{Z^{(\rho)}_{\Omega^\delta\setminus \gamma^\delta[0,n]}(x^\delta, \gamma^\delta(n))}{Z^{(\rho)}_{\Omega^\delta\setminus \gamma^\delta[0,n]}(o^\delta, \gamma^\delta(n))} \to \frac{P^{(\rho)}_{\Omega_t} (x, a_t)}{P^{(\rho)}_{\Omega_t} (o, a_t)}
\end{equation}
where $a_t$ denotes the tip of $\gamma_t$, viewed as a prime end in $\Omega_t$.
This is the analogue of Proposition 3.16 in \cite{ChelkakWan}. The right hand side is continuous in $t\ge 0$, as remarked at the end of Section~\ref{S:PK}. Furthermore, as argued in Section 2.4 of \cite{ChelkakWan}, the discrete martingales $M_n^{(\rho)}(x^\delta)$ yield continuous martingales in any subsequential limit (see in particular Remark 2.3 in \cite{ChelkakWan}), hence we deduce 
 that for every $x \in \Omega$, every $r>0$,
 $$
 M_t^\rho(x) : = \frac{P^{(\rho)}_{\Omega_t} (x, a_t)}{P^{(\rho)}_{\Omega_t} (o, a_t)} ; t \wedge \sigma_r
 $$
is a martingale, where for every $r>0$ the stopping time $\sigma_r$ is defined as $\inf\{ t>0: |\gamma_t - b | \wedge |\gamma_t - x | \le r\}$.

Now we explain how these martingales can be used to identify the drift $\lambda_t$ uniquely. To see this, first recall that $P_t^{(\rho)}:=P_{\Omega_t}^{(\rho)}$ satisfies the resolvent identity (Proposition \ref{P:resolvent}), namely,
$$
P_t^{(\rho)}(x, a_t) = P_t(z,a_t) + \int_\Omega G_t^{(\rho)} (x,y) \rho(y) P_t(y, a_t) \dd y. 
$$
Now we know that in the critical ($\rho=0$) case, one has as a direct application of Loewner's equation
$$
\dd P_t (x, a_t) = Q_t (x) \dd\xi_t.
$$
Since this is an a.s. identity, this same identity remains true by absolute continuity for our subsequential limit. Together with the Hadamard formula (Lemma \ref{l:massHadamard}), the resolvent identity for $P_t^{(\rho)}$ (Proposition \ref{P:resolvent}) implies
\begin{align}
\dd P_t^{(\rho)} (x)&= Q_t(x)\dd \xi_t+\int_\Omega G^{(\rho)}_t(x,y)\rho(y)Q_t(y)\dd\xi_t \dd y-\pi P^{(\rho)}_t(x)\int_\Omega P^{(\rho)}_t(y)\rho(y)P_t(y)\dd y\dd t\nonumber\\
&=Q^{(\rho)}_t(x)\dd \xi_t-\pi P^{(\rho)}_t(x)\int_\Omega P^{(\rho)}_t(y)\rho(y)P_t(y)\dd y\dd t,\label{eq:notinverse}
\end{align}
where we used Lemma \ref{l:stochFubini} to exchange $\dd \xi_t$ and $\dd y$.

Since we know that $\frac{P_t^{(\rho)}(x)}{P_t^{(\rho)}(o)}$ is a bounded martingale for any $x\in B(o,\tfrac12 r)$ until time $\sigma_r$, we consider:
\begin{align}
    \dd \frac{P_t^{(\rho)}(x)}{P_t^{(\rho)}(o)}=&
    P_t^{(\rho)}(o)^{-1}\dd P_t^{(\rho)}(x) + P_t^{(\rho)}(x)\dd \ (P_t^{(\rho)}(o)^{-1})+\dd\ \langle P_t^{(\rho)}(o)^{-1},P_t^{(\rho)}(x)\rangle_t \nonumber \\
    =&P_t^{(\rho)}(x)\left(\dd\ (P_t^{(\rho)}(o)^{-1})-\pi P_t^{(\rho)}(o)^{-1}\int_{\Omega_t}\rho(y)P_t(y)P^{(\rho)}_t(y)\dd y \dd t\right) \label{mart1}\\
    &+Q_t^{(\rho)}(x)\left(P_t^{(\rho)}(o)^{-1}\dd \xi_t +\dd\ \langle \xi_t,P_t^{(\rho)}(o)^{-1}\rangle\right).\label{mart2}
\end{align}
Since this is a martingale for any $x$, and $P^{(\rho)}_t$ and $Q^{(\rho)}_t$ are clearly linearly independent functions of $x$, both of \eqref{mart1} and \eqref{mart2} (and thus each bracket in these two lines) must be local martingales and thus have vanishing finite variation parts.
By standard stochastic calculus arguments (applying It\^o's formula to describe $\dd P_t^{(\rho)}(o)^{-1}$ from \eqref{eq:notinverse}), the finite variation part of the second bracket is
\begin{equation}
    \frac{2}{(P_t^{(\rho)}(o))}\lambda_t-2\frac{Q^{(\rho)}_t(o)}{(P_t^{(\rho)}(o))^2},
\end{equation}
this implies that $\lambda_t=\frac{Q_t^{(\rho)}(o)}{P_t^{(\rho)}(o)}$, as desired.

The following lemma together with continuity of $\lambda_t$ implies the uniqueness of solutions to the SDE \eqref{eq:drivingSDE} by Novikov's condition.
\begin{lem}\label{lem:novikov}
There is a constant $C=C(\|\rho\|_\infty,\diam(\Omega ))<\infty$ such that the drift $\lambda_t$ almost surely satisfies
\begin{equation*}
    \int_0^\infty |\lambda_t|^2\leq C\,.
\end{equation*}    
\end{lem}
\begin{proof}
We start by noting that 
\begin{align*}
    P^{(\rho)}_t(o)&=\E_{o\to a_t;\Omega_t}[\exp(-\int_0^\sigma\rho(X_s)\dd s]\geq \exp(-c_0\|\rho\|_\infty\diam(\Omega )^2)\\
    & \ge \exp [ - \|\rho\|_\infty \E_{o\to a_t; \Omega_t} (\sigma)  ].
    \end{align*}
    We claim that $\E_{o\to a_t; \Omega_t} (\sigma)  \le c_0 \diam(\Omega)^2$. To see this, note that (for instance using the Doob transform description of Brownian motion conditioned to leave $\Omega_t$ by $a_t$),
    $$
    \E_{o\to a_t; \Omega_t} (\sigma) = \int_{\Omega_t} G_{\Omega_t}(o, y) \frac{P_t(y,a_t)}{P_t(o,a)} \dd y. 
    $$
Moreover, one can deduce from \eqref{eq:CWboundsPQ} and conformal invariance that 
$$
G_{t}(o,y) \frac{P_t(y,a_t)}{P_t(o,a_t)} \le c_0 ( 1+ G_{t}(o,y)) \le c_0 (1+ G_\Omega(o ,y) )
$$
so that 
$$
\E_{o\to a_t; \Omega_t} (\sigma)\le c_o \int_\Omega  (1 + G_\Omega(o,y) )\dd y  \le c'_0 \diam(\Omega)^2
$$
as claimed. 

Furthermore, by the resolvent equation for $Q^{(\rho)}_t$ \eqref{eq:deriv_resolvent} and $Q_t(o)=0$ by definition, we have
    \begin{equation}
        Q^{(\rho)}_t(o)=-\int_\Omega G_t^{(\rho)}(o,y)\rho(y)Q_t(y)\dd y\,.
    \end{equation}

 Combining these two estimates together and using  \eqref{eq:CWboundPQ2}, we get
    \begin{equation*}
        \int_0^\infty\left|\frac{Q^{(\rho)}_t(o)}{P^{(\rho_t)}(o)}\right|^2\dd t\lesssim \int_0^\infty \left(\int_\Omega  P_t(z)\dd x\right)^2\dd t \leq \frac{1}{\pi}\int_\Omega\int_\Omega  G_\Omega (z,w)\dd z\dd w\,,
    \end{equation*}
    as shown in the proof of \eqref{eq:propgoalPQ2}, and where the hidden constant depends only on $\diam(\Omega )$ and $\|\rho\|_\infty$ and the final integral is bounded by a constant only depending on $\mathrm{Diam}(\Omega )$.
\end{proof}
This concludes the proof of Theorem \ref{T:LERWgeneral_massTEXT}. \end{proof}

\section{Scaling limit of the LERW with drift and conformal covariance}\label{sec:drift}

In this section we explain how 
Theorem \ref{T:LERWgeneral_massTEXT} implies Theorem  \ref{T:LERWgeneral_drift}.
We also discuss why this implies convergence of the height function in the associated dimer models and why these satisfy conformal covariance (Theorem \ref{T:covariance}). 
To do this we will rely on our discrete Girsanov theorem (more precisely Corollary \ref{cor:girsanovpotential}). We will in particular need to check that the assumptions of Theorem \ref{T:LERWgeneral_massTEXT} are satisfied not only on the directed triangular lattice, but also on the 
 \textbf{image} of this lattice under a conformal map.

\begin{rem}
    While in principle possible, trying to work directly with quantities associated to the random walk with drift poses serious difficulty since the formal analogues of many statements (e.g. \ref{eq:deriv_resolvent}) have well-posedness issues stemming from worse regularity properties of the operator $\frac{1}{2}\Delta+\alpha\cdot\nabla$ compared to $\frac12\Delta-\rho$.
    \end{rem}

Let us begin with the proof for the triangular lattice.
\begin{proof}[Proof of Theorem \ref{T:LERWgeneral_drift}.]
    We want to apply Theorem \ref{T:LERWgeneral_massTEXT}.
    Recall that $\P^{(\ph)}$ denotes the law of a random walk on $\Gtree$ with drift $\alpha(v) = \alpha^\delta(v)$ where $\alpha_k^\delta(v) = \ph( v+ \delta \tau^k) - \ph(v) $ (the transition probabilities of the walk are described in \eqref{eq:driftweights}).  
    By Corollary \ref{cor:girsanovpotential} the law $\P^{(\ph)}$, conditioned so that $X^\delta_{\sigma^\delta} = a^\delta$, has the same law as the random walk $\P^{(\rho^\delta)}$ with mass $$\rho^\delta=\Delta^{\delta \mathbb T}\ph+\frac{1}{3}\beta^2,$$ 
    also conditioned so that $X^\delta_\sigma = a^\delta$ (and in particular to survive until doing so).

    We need to check that this random walk $\P^{(\rho^\delta)}$ satisfies the conditions of the theorem \ref{T:LERWgeneral_massTEXT}.
    First note that $\rho^\delta(v)=\delta^2\rho(v)/2+o(\delta^2)$ uniformly in $v$, where $\rho(z) = \tfrac12 (\Delta \ph(z) + \|\nabla \ph(z)\|^2 )$ as in the statement of Theorem \ref{T:LERWgeneral_drift}.

    As already noted in Section \ref{sec:notation}, this implies that $(X^\delta_{2t\delta^{-2}})_{t\geq 0}$ converges weakly, uniformly on compacts, to the law $\P^{(\rho)}$ of Brownian motion killed at the instantaneous rate $\rho(x)$ when in position $x \in \Omega$.
    
Secondly we need to check that there is a constant $c_0$ such that
    \begin{equation}
        \E^{(0)}_{\Gtree,o^\delta\to a^\delta}(\sigma^\delta)\leq c_0\delta^{-2}R^2\,.
    \end{equation}
    This is an estimate purely for the simple random walk on the triangular lattice and follows as in \cite[Corrollary 2.8]{ChelkakWan}.

    It remains to check the uniform crossing assumption \eqref{E:crossing_goal}.
    In order to prove that $\P^{(\rho^\delta)}_{z} ( \Cross_r )\ge c$ for some uniform constant $c>0$ we will in fact consider the restricted event $G = \Cross_r  \cap\{\sigma \le \delta^{-2}\}$, where $\sigma$ is the stopping time at which the walk first leaves the relevant rectangle.
    \begin{align}
        \P^{(\rho^\delta)}_{x} ( \Cross_r ) 
        &\geq \P^{(\rho^\delta)}_{x} (G)=\E^{(0)}_{z} ( 1_G\prod_{s=0}^{\sigma}(1-\rho^\delta(X_s) )\\
        &\geq \P^{(0)}_{z} (G)(1-\|\rho^\delta\|_\infty)^{\delta^{-2}}\\
        &=(1+o(1))\exp(-c\|\rho\|_\infty)\P^{(0)}_{x} (G).
    \end{align}
    The statement follows since $\P^{(0)}_{z} (G)$ is uniformly bounded below.  Thus the assumptions of Theorem \ref{T:LERWgeneral_massTEXT} are fulfilled and Theorem \ref{T:LERWgeneral_drift} follows.
\end{proof}

Now to prove \ref{T:covariance} we will first show conformal covariance for massive SLE$_2$.
\begin{thm}[Conformal covariance for SLE$_2$ with mass profile]\label{T:masscovariance}
    Let $\Omega $ be a simply connected domain, and $\rho:\Omega \to[0,\infty)$ be a bounded and continuous mass profile.
    Let $T:\Omega \to \tilde \Omega $ be a conformal map such that $|T'|$ is uniformly bounded away from $0$ and $\infty$ on $\Omega $.
    Then the image of radial massive SLE$_2$ from $a\in \partial \Omega $ to $o\in \Omega $, with mass profile $\rho$, under $T$ is given by radial massive SLE$_2$ from $T(a)$ (seen as an element of the Martin boundary) to $T(o)$ with mass profile $|(T^{-1})'(\cdot)|^2\rho(T^{-1}(\cdot))$.
\end{thm}
\begin{proof}
    The strategy of this proof is as follows:
    Let $X^\delta$ be the random walk on the directed triangular lattice with mass profile $\rho^\delta$ approximating $\rho$ as in the previous theorem.
    Consider the random walk $T(X^\delta)$, which is a random walk on the \textbf{image} of the directed triangular lattice $T(\Gtree)$ under $T$. Note that this is also a planar graph. We will aim to apply Theorem \ref{T:LERWgeneral_massTEXT} to this walk and so need to check that the conditions of the theorems are also fulfilled.

    Noting that condition \ref{as:Absolutecontinuity} on the expected time to leave the domain, 
        does not depend on the embedding of the graph, so it follows directly from what we proved above. 
        It remains to check the other two assumptions.
    The fact that it converges to a time changed massive Brownian motion follows from the standard conformal invariance of Brownian motion.
    Indeed if $B_t$ is standard Brownian motion, then
    so is $\tilde{B}_t:=T(B_{\xi^{-1}(t)})$, where $\xi(t)=\int_0^t|T'(B_s)|^{2}\dd s$.
    Let $\sigma_\Omega $ be the time at which $B_t$ leaves $\Omega $ and $\tilde \sigma_{\tilde \Omega }$ be the time at which $\tilde{B}_t$ leaves $\tilde \Omega $.
    By definition, $\xi(\sigma_\Omega )=\tilde\sigma_{\tilde\Omega}$.
    Now consider the Radon--Nikodym derivative of a massive Brownian motion in $\Omega $ with profile $\rho$ with respect to standard Brownian motion and rewrite to in terms of $\tilde{B}$:
    \begin{align}
        \exp\left(-\int_0^{\sigma_\Omega }\rho(B_s)\dd s\right)
        &=\exp\left(-\int_0^{\sigma_\Omega }\rho(T^{-1}(\tilde{B}_{\xi(s)})\dd s\right)\\
        &=\exp\left(-\int_0^{\tilde\sigma_{\tilde\Omega}}\rho(T^{-1}(\tilde{B}_s))|T'(T^{-1}(\tilde{B}_s)|^2\dd s\right)\\
        &=\exp\left(-\int_0^{\tilde\sigma_{\tilde\Omega}}\rho(T^{-1}(\tilde{B}_s))|((T^{-1})'(\tilde{B}_s)|^{-2}\dd s\right)\,.
    \end{align}
    Since the condition in Theorem~\ref{T:LERWgeneral_massTEXT} is convergence of paths up to time reparametrization, this shows that assumption \ref{as:timechange} holds.
    The uniform crossing estimate \eqref{as:crossing} follows the fact that this assumption is invariant under conformal maps using the Koebe one-quarter theorem and the uniform control over $|T'|$ in our assumption.
\end{proof}

To prove Theorem \ref{T:covariance} we first see how Theorem \ref{T:LERWgeneral_drift} implies convergence of the dimer height function when the weights are given by  \eqref{eq:weights_hex}.
\begin{prop}
\label{P:convergenceheight} Consider the directed triangular lattices with weights \eqref{eq:weights_hex}. Let $h^{(\alpha),\delta}$ denote the height function of the biperiodic dimer on the dimer graph $G^\delta$ (a piece of the hexagonal lattice). Then $h^{(\alpha),\delta}$ converges, in the sense that if $f$ is a test function, then
$$
(h^{(\alpha),\delta}, f) \to ( h^{(\alpha)}, f)
$$
converges in law and in the sense of moments. Here $h^{(\alpha),\delta}$ is identified with a function defined on all $\Omega$ which is constant on each face of $G^\delta$, and the inner product above is simply the $L^2$ inner product of square integrable functions.
\end{prop}

\begin{proof}
The convergence of the loop-erased random walk in Theorem~\ref{T:LERWgeneral_drift}, applied iteratively using Wilson's algorithm, implies the convergence of the uniform spanning tree $\cT$ with weights \eqref{eq:weights_hex} in the Schramm topology (\cite{Schramm1999ScalingTrees}).
Recall that this tree is identical to the tree one obtains from applying the Temperley bijection to the biperiodic dimer model with weights \eqref{eq:weights_hex}. 
We apply a general theorem (Theorem 8.1 in \cite{BLR_torus}) in order to deduce convergence of the height function. 
The theorem, which follows the approach originating in \cite{BLR_DimersGeometry}, is particularly simple to apply on simply connected domains, which is our situation. 
The assumptions of that theorem in this simplified situations are as follows:

\begin{itemize}

\item There exists $c>0$ such that the following holds. For any vertex $v \in v ( \Gtree)$, for any interior point $z \in \Omega$, if $r = |v -z| \wedge \text{dist} (v, \partial \Omega) \wedge\text{dist} (z, \partial \Omega)  $ and  if $\gamma$ is the loop-erasure of the random walk starting from $v$ and killed when it leaves $\Omega$, then for any $0 < \epsilon < 1$,
\begin{equation}
\label{E:LERWsf}
\P^{(\ph)}_v ( \gamma \cap  B(z, r \epsilon ) \neq \emptyset) \le \epsilon^c,
\end{equation}
in other words $\gamma$ is polynomially unlikely to enter a small ball near $z$.

\item There exists $C, c> 0$ and for every $k\ge 1$ there is a constant $M_k$ such that the following holds. For any $v \in v (\Gtree)$, let $\gamma$ denote the loop-erasure of the random walk starting from $v$ and killed when it leaves $\Omega$, parameterized from $v$ to $\partial \Omega$. For all $r > 0$, let $\theta_r$ denote the first time it leaves $B(v, r)$ and $\sigma_r$ the last time it is in $B(v, er)$. For $s<t$, let $ W( \gamma[ s, t])$ denote the intrinsic winding of the path $\gamma([s,t])$ (that is, on a graph where all edges are straight, the sum of the turning angles of $\gamma$ during that interval of time). Then for every $k \ge 1$,
\begin{equation}\label{E:windingtail}
\E^{(\ph)}_v [ \sup_{\theta_r \le s \le t \le \sigma_r } |W ( \gamma[s,t]) |^k] \le M_k,
\end{equation}
in other words the winding of the path $\gamma$ at any scale $r$ is of order one.
\end{itemize}
The proofs in \cite{BLR_DimersGeometry} of both these facts for the random walk on $\Gtree$ relies on nothing but the uniform crossing estimate of \eqref{E:crossing_goal}; in fact Proposition 4.4 of \cite{BLR_DimersGeometry} and Proposition 4.12 of \cite{BLR_DimersGeometry} are stated for general random walks on embedded planar graphs subject to the uniform crossing estimate (convergence to Brownian motion is also assumed throughout that section, but plainly that assumption is only used to identify the law of the limit of loop-erased random walk). Hence Proposition 4.4 of \cite{BLR_DimersGeometry} applies and yields \eqref{E:LERWsf}; and Proposition 4.12 of \cite{BLR_DimersGeometry} also applies and yields uniform stretched exponential tails hence \eqref{E:windingtail}. This completes the proof of Proposition \ref{P:convergenceheight}.

It is also possible to deduce \eqref{E:LERWsf} and \eqref{E:windingtail} from the Proposition 4.4 of \cite{BLR_DimersGeometry} and Proposition 4.12 of \cite{BLR_DimersGeometry} (applied to the usual driftless random walk on the square lattice) and the fact that the Radon-Nikodym derivative in Corollary~\ref{cor:girsanovpotential} is uniformly bounded by $\exp(2\sup_{x\in\Omega }|\ph(x)|)$.

Let us see how this may be used to finish the proof of \eqref{E:LERWsf} and \eqref{E:windingtail}. Consider for instance \eqref{E:LERWsf}.
\begin{align*}
\P^{(\ph)}_v ( \gamma \cap  B(z, r \epsilon ) \neq \emptyset) & = \E^{(0)}_v  [ 1_{ \{  \gamma \cap  B(z, r \epsilon ) \neq \emptyset \} } e^{\ph(X_\sigma)-\ph(X_0) - \tfrac12 A_\sigma} ] \\
& \le
\P^{(0)}_v [ \gamma \cap  B(z, r \epsilon ) \neq \emptyset]\exp(2\sup_{x\in\Omega }|\ph(x)|)
\end{align*}
so using Proposition 4.4 of \cite{BLR_DimersGeometry} we obtain \eqref{E:LERWsf}. The same argument also implies \eqref{E:windingtail}.

This concludes the proof of Proposition \ref{P:convergenceheight}.
\end{proof}

\begin{proof}[Proof of Theorem \ref{T:covariance}] We are now ready to finish the proof of Theorem \ref{T:covariance}. All that remains to prove is the conformal covariance of the limiting height function $h^{(\alpha); \Omega}$ (here we write explicitly the dependence on the domain $\Omega$ in order to avoid confusions). Let $\tilde \Omega$ be another bounded simply connected domain and let $T:  \Omega \to \tilde\Omega$ be a conformal map with bounded derivative. Recall that we wish to show
  $$
  h^{(\alpha);\Omega } \circ T^{-1}= h^{(\tilde \alpha); \tilde \Omega}
  $$
  where at a point $w \in \tilde \Omega$,
\begin{equation}\label{E:driftchange}
\tilde \alpha (w) = \overline{(T^{-1})'(w)} \cdot \alpha (T^{-1} ( w)).
\end{equation}
The idea is to use the same approach as in Theorem~\ref{T:masscovariance}, i.e. using both the convergence as in Proposition~\ref{P:convergenceheight} and the same type of result on the lattice obtained by the image of $G^\delta$ under $T$.
Indeed since the connection with the massive random walk (i.e. Theorem~\ref{cor:girsanovpotential}) does not depend on the embedding, the analogue of Theorem~\ref{T:LERWgeneral_drift} for the random walk on the deformed triangular lattice is an immediate consequence of Theorem~\ref{T:masscovariance}.
The scaling limit of the corresponding random walk is necessarily the image by $T$ of a Brownian motion with drift $\alpha$ in $\Omega$. Applying It\^o's formula and the Cauchy--Riemann equations, one checks that $\alpha$ and $\tilde \alpha$ are related via \eqref{E:driftchange}.

Likewise \eqref{E:LERWsf} and \eqref{E:windingtail} are trivially verified in $\tilde \Gtree$ because they are verified in $\Gtree$ and $T$ has bounded derivative. The dimer model associated to $T( G^\delta)$ is the image by $T$ of the dimer model on $G^\delta$ and has a height function which necessarily converges to $ h^{(\alpha); \Omega } \circ T^{-1}$ in $\tilde\Omega$. On the other hand, the law of the limiting Temperleyan tree is uniquely determined by the law of its branches, which by Theorem \ref{T:LERWgeneral_drift} are off-critical radial SLE$_2$ with limiting drift vector field $\tilde{\alpha}$, as described in \eqref{E:Loewner_off}.
We conclude that, in law,
$$
h^{(\alpha); \Omega } \circ T^{-1} = h^{(\tilde\alpha), \tilde\Omega},
$$
as desired.
\end{proof}

\appendix

\section{Continuum hitting probabilities}

\label{SS:continuum}

The following well-known proposition is recalled as Proposition 3.3 of \cite{YadinYehudayoffLERWconvergence} and can be proved using the fact that for two dimensional Brownian motion $\log(|B_t|)$ is a local martingale and the inequality $\log (1-r) \le -r$.
\begin{lem}\label{L:BMhit}
Let $\DD$ be the unit disc and let $x\in\DD$ be different from $0.$ Let $0<\epsilon<|x|$. Let $\sigma$ be the exit time of $X_t$ from the unit disc $\DD$.
Then
\begin{equation}
    \mathbb{P}_x(\exists t\in[0,\sigma]:|X_t|<\epsilon)\geq \frac{1-|x|}{\log(1/ \epsilon)},
\end{equation}
\end{lem}

We need to replace this with a suitable analogue for massive Brownian motion.
\begin{lem}\label{L:hitmass}
Suppose $\Omega = \DD$ is the unit disc. There exists a constant $c>0$ such that the following holds. Let $x\in \Omega$ be different from $0$ and $0<\epsilon<|x|$. Let $\sigma$ be the exit time of $X_t$ from the disc.
Then
\begin{equation}
    \mathbb{P}^{(\rho)}_x(\exists t\in[0,\sigma \wedge \sigma_*]:|X_t|<\epsilon)\geq c\frac{1-|x|}{\log(1/ \epsilon)}.
\end{equation}
\end{lem}

\begin{proof}
Suppose without loss of generality that $\eps = e^{-N}$ for some $N \ge 1$. Writing down the Radon--Nikodym derivative with respect to ordinary Brownian motion, and letting $\sigma_\epsilon$ being the first time the trajectory enters $B(0, \epsilon)$, we get
\begin{align*}
    \mathbb{P}_x^{(\rho)}(\sigma_\epsilon < \sigma)
    &=
      \mathbb{E}^{(0)}_x \left(1_{\{ \sigma_\epsilon < \sigma \} } \exp(-\int_0^{\sigma_\epsilon}\rho (X_s)ds)\right)\\
  & \ge  \mathbb{E}^{(0)}_x\Big(\exp(-\sigma_\epsilon\|\rho\|_\infty)\Big| \sigma_\epsilon < \sigma\Big) \mathbb{P}_x(\sigma_\epsilon < \sigma).
\end{align*}

Thus it remains to show
\begin{equation}\label{E:goalmas}
\mathbb{E}_x\Big(\exp(-M^2 \sigma_\epsilon )\Big| \sigma_\epsilon < \sigma\Big) \ge c,
\end{equation}
for some constant $c$, where $M^2 = \| \rho\|_\infty$. A priori, the difficulty is that conditioning the Brownian motion to hit a very small ball might cause the process to waste a lot of time and thus make it highly likely to be killed (or equivalently make the exponential term very small). We will see this is not the case; essentially, when we condition planar Brownian motion to hit zero before leaving the unit disc, it does so in an a.s. finite time.

Let $\sigma_0=\inf \{ t >0: |B_t |= e^k \text{ for some $k \in \Z$} \}$, and define inductively a sequence of stopping times $\sigma_n$ by setting
$$
\sigma_{n+1}=\inf\{t>\sigma_n:  |B_t| = e^{k} \text{ for some $k \in \Z$ with} |B_t | \neq |B_{\sigma_n}| \}.
$$
In words, the sequence $\sigma_n$ corresponds to the sequence of times at which $|B_t|$ is of the form $e^k$ for some distinct $k$.

Let $M_n = \log_r (|B_{\sigma_n}|)$. Because $\log|x|$ is a harmonic function on $\mathbb R^2$ and rotational invariance of Brownian motion, it is easy to see that
 $M_n$ is nothing but simple random walk on $\mathbb Z$ with a possibly random initial value $M_0$ which however differs from $\log |x|$ by at most 1. Let $\theta_\epsilon$ denote the first $n$ such that $M_n \le -N $ (recall that we have assumed $\epsilon = e^{-N}$, so $\theta_\epsilon$ corresponds to Brownian motion entering $B(0, \epsilon)$). Let $\theta $ be the smallest $n$ such that $M_n \ge 0$ (which corresponds to Brownian motion leaving the unit disc).

Now let us describe the effect of conditioning on $\sigma_\epsilon < \sigma$ (or equivalently $\sigma_\epsilon < \sigma$). The conditional transition probabilities
are well-known and easy to compute (this can be viewed as an elementary version of Doob's h-transform). Writing $\tilde \P$ for the conditional probability measure given $\theta_\epsilon < \theta$,
we obtain for $-N+1 \le k \le -1$,
 \begin{equation}\label{E:conditionalwalk}
\tilde  \P (M_{n+1}=k\pm 1 |
  M_n=k)=\frac12(1\mp \frac1{|k|}).
 \end{equation}
Note that this description is actually independent of $N$ (or equivalently $\epsilon$).
The formalism of electrical networks is useful to describe the conditional walk defined by \eqref{E:conditionalwalk} (which, up to the sign, is essentially a discrete version of a three-dimensional Bessel process, and is in particular transient).
To put it in this framework, note that \eqref{E:conditionalwalk} coincides with the walk on the network with conductances $c(k,k-1)=\binom{|k| +1}2$. Indeed in that case the corresponding stationary measure is then
$$\pi(k) = \binom{|k| +1}2 + \binom{|k| }2 = k^2\,,
$$
so that $c(k, k-1) / \pi(k)$ coincides with \eqref{E:conditionalwalk} as desired.
The corresponding unit current voltage $v(k)=\frac2{|k|}$ (if we set zero voltage at $-\infty$ and unit voltage at 1), which means that the expected number of visits to $k$ is exactly $2|k|$ if we let the conditioned walk \eqref{E:conditionalwalk} live forever. We deduce that
\begin{equation}\label{E:localtimecond}
\tilde \E ( \# \{n \le \theta: M_n = k \} ) \le 2 |k|.
\end{equation}
(This can also be computed directly using elementary computations based on the gambler's ruin probability, and considering the probability from $k$ that the conditioned walk ever returns to $k$).

Now let us decompose
\begin{equation}\label{E:decomposition}
    \sigma_\epsilon - \sigma_0=\sum_{n=0}^{\theta_\epsilon-1} (\sigma_{n+1}-\sigma_{n} )= \sum_{j=1}^{N-1} \sum_{m=1}^\infty 1_{\{N_j^m < \theta_\epsilon\}} (\sigma_{ N_j^m +1} - \sigma_{N_j^m})
\end{equation}
where for $ 1\le j \le N-1$ and $m \ge 1$, $n = N_j^m$ is the time of the $m$th visit to level $-j$ by the martingale $M_n$.
We will check that the conditional expectation of the left hand side, given $\theta_\eps < \theta$, remains finite as $\eps \to 0$.

Let $\cF$ denote the $\sigma$-algebra generated by all the random variables of the form $X_{\sigma_n}, 0 \le n \ne N$. Note that the event $\theta_\eps < \theta$ is measurable with respect to $\cF$, and that given $\cF$, the trajectory of $(X_t, 0 \le \sigma_\eps)$ may be split in pieces of the form $X[ \sigma_n, \sigma_{n+1}]$, which are \emph{independent} of one another, and where each piece may be described as a Brownian motion starting from $X_{\sigma_n}$ conditioned to exit a certain annulus $A_n = B(0, e^{M_n +1}) \setminus B(0, e^{M_n -1})$ through $X_{\sigma_{n+1}}$.  Now, if $A$ is any annulus of the form $B(0, e^{k+1} ) \setminus B(0, e^{k-1})$ and $y \in A$ is any interior point, $z \in \partial A$ is any point on the boundary of the annulus $A$, then it is not hard to see for some constant $C>0$, by Brownian scaling,
\begin{equation}\label{E:scaling}
 \E_y (\sigma_A |X_{\sigma_A} = z)  \le C e^{2k}\,,
\end{equation}
where $\sigma_A$ is the time at which $X$ leaves $A$,  and this estimate is uniform in $y \in A, z \in \partial A$, and $k \in \Z$. Consequently,
\begin{equation}\label{E:uniformexit}
\E( \sigma_{n+1} - \sigma_n | \cF) \le C e^{2 M_n}.
\end{equation}
This implies that $\tilde \E( \sigma_0) \le C < \infty$. Furthermore, using \eqref{E:decomposition}
\begin{align*}
    \tilde{\mathbb{E}}(\sigma_\epsilon - \sigma_0) & = \sum_{j=1}^{N-1} \sum_{m=1}^\infty  \tilde  \E \left[1_{\{N_j^m < \theta_\epsilon\}} (\sigma_{ N_j^m +1} - \sigma_{N_j^m}) \right]\\
    & = \sum_{j=1}^{N-1} \sum_{m=1}^\infty  \tilde  \E \left[ 1_{\{N_j^m < \theta_\epsilon\}}  \tilde \E [ (\sigma_{ N_j^m +1} - \sigma_{N_j^m} ) | \cF] \right] \\
    & \le \sum_{j=1}^{N-1} \sum_{m=1}^\infty  \tilde  \E \left[ 1_{\{N_j^m < \theta_\epsilon\}}   C e^{-2j} \right]\\
    & \le C  \sum_{j=1}^{N-1}  e^{-2j}  \tilde \E ( \# \{n \le \theta: M_n = j \} ) \\
    & \le C \sum_{j=1}^{N-1} j^2 e^{-2j}.
\end{align*}
Here we used \eqref{E:uniformexit} in the third line, and \eqref{E:localtimecond} in the last line. The right hand side is uniformly bounded in $N$ (or equivalently $\eps$). We deduce that $\E ( \sigma_\epsilon | \sigma_\epsilon < \sigma ) \le C$ for some constant $C$ independent of $x$. Therefore, using Jensen's inequality and convexity of $x \mapsto e^{-x}$, we get
$$
\mathbb{E}_x\Big(\exp(-M^2 \sigma_\epsilon )\Big| \sigma_\epsilon < \sigma\Big) \ge  \exp ( - M^2 \E_x( \sigma_\epsilon | \sigma_\epsilon < \sigma) ) \ge \exp ( - M^2 C),
$$
which proves \eqref{E:goalmas}. This concludes the proof of Lemma \ref{L:hitmass}.
\end{proof}

\section{Discrete crossing, Beurling estimates}

\label{SS:discrete}

Here we supply the missing discrete estimates required for the proof of Lemma \ref{lem:YY}, required to complete the proof of  \ref{T:LERWgeneral_massTEXT} and thus Theorem \ref{T:LERWgeneral_drift}. The first one concerns disconnection events: for $z \in\Omega $, and $r>0$ such that $B (z, 10 r) \subset\Omega $, let us write $ x[0,t] \circlearrowleft^{(r)}z$ for the event
that the path $x[0,t]$ disconnects $B(z, r)$ from $B(z, 5r)^c$  (or, equivalently, makes a non-contractible loop in the corresponding annulus); this is the notation from \cite{YadinYehudayoffLERWconvergence}. The next lemma corresponds to Proposition 3.4 in \cite{YadinYehudayoffLERWconvergence}, where it is stated for Brownian motion.
We will need it for the random walk. 

\begin{lem}\label{lem:circle}
For every $R$ there exists a $z$ such that the following holds:
Let $0<r\leq R$ and let $z\in\mathbb{C}$. Let $T$ be the exit time of $X(\cdot)$ from $B(z,r)$. Then for every $x^\delta\in B(z,r/2)$, $$\mathbb{P}^{(\rho)}_{x^\delta}(X(0,T)\circlearrowleft^{(r)}z)\geq c.$$
\end{lem}
\begin{proof}
Encircling a point at scale $r$ contains the intersection of ten box-crossing events (see Figure \ref{F:loops}).
We conclude using our crossing assumption \eqref{E:crossing_goal}.
\end{proof}

\begin{figure}
\begin{center}
\includegraphics[scale=.8]{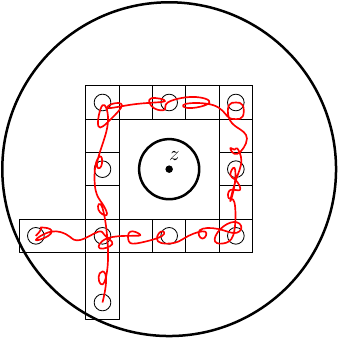}
\label{F:loops}
\end{center}
\caption{Making a loop by crossing rectangles.}
\end{figure}

The last missing piece is a Beurling estimate (corresponding to Proposition 4.1. in \cite{YadinYehudayoffLERWconvergence}), which shows that a walk starting close to the boundary of a domain is very likely to leave this domain in a short time, without going far from its starting point. Actually what is needed is the version of this estimate in which we want to ensure the random walk will hit a given curve which is close to its starting point; of course, this makes no difference.
Such an estimate is well-known in the critical case where the walk converges to Brownian motion.
This remains true in the off-critical regime thanks to the following observation: while of course the off-critical Brownian motions are not scale invariant, this effect disappears at small scales. In fact, making loops at any scale above that separating the curve from the starting point guarantees an intersection, and so we can get a uniform bound using the previous observations. Also, since we assume that the original domain $\Omega $ is bounded, we do not need to consider arbitrarily large scales and can therefore obtain uniform bounds for all domains which have diameter less than some constant $R$.

The desired estimate is formulated in \cite{YadinYehudayoffLERWconvergence} after applying a conformal map to the unit disc (let $\phi$ denote the unit conformal map from $\Omega $ to $\DD$ such that $\phi(o) = 0$ and $\phi'(o)>0$).
This is initially a little worrying, since we did not assume uniform crossing after applying the conformal map $\phi$ but instead only in $\Omega $ itself.
(Note that this uniform crossing estimate could in fact fail to hold for $\phi(\Gtree)$ if the domain $\Omega $ is not very nice). Thankfully, we will see that thanks to Koebe's one quarter theorem we can get the required estimate. 
\begin{lem}\label{lem:hitcurve}
For all $\alpha,R>0$, there exists an $\eta>0$ such that for all $\tilde \epsilon>0$, for all simply connected domains $\Omega $ such that $0\in \Omega \subset B(0,R)$, and for all $\tilde a\in(1-\tilde \epsilon)\mathbb \Omega $, there exists a $\delta_0$ such that the following holds for all $\delta < \delta_0$:

Let $y\in v(\Gtree)\cap\phi^{-1}(\rho(\tilde a, \eta\tilde \epsilon)) \in \Omega $. Let $X^\delta$ denote random walk on $\Gtree$ starting from $y$. Then, for every continuous curve $g$ starting in $B(\tilde a,\eta\tilde \epsilon)$ and ending outside of $B(\tilde a, \tilde \epsilon)$, $$
\P_y^\delta ( \phi (X[0,T]) \cap [g] = \emptyset)
\le \alpha\,,
$$
where $[g]$ is the range of $g$ and $T$ is the time at which $\phi(X)$ leaves $B( \tilde a, \tilde \epsilon)$.
\end{lem}
\begin{proof}
Let $\tilde \epsilon>0$ and let $\tilde a \in (1- \tilde \eps) \DD$. Let $ a = \phi^{-1} (\tilde a) \in\Omega $, and let $\eps = |(\phi^{-1})'(\tilde a) | \tilde \eps$; note that we have no control over the actual size of $\eps$ since it depends on the conformal map near $\tilde a$. Nevertheless, applying the Koebe $1/4$-theorem (twice), it is easy to see that the image of curve $g$ under $\phi^{-1}$ starts from a ball of radius $4\eta \eps$ around $a$, and ends outside of a ball of radius $\eps/4$ around $a$. For $\phi(X^\delta[0,T])$ to avoid $g$, $X^\delta[0,T]$ must therefore avoid making loops at \emph{all} scales between $4 \eta \eps$ and $\eps/4$ (this corresponds to a \emph{fixed} number of scales, even though $\eps$ itself is variable).
Furthermore, using the strong Markov property, all the events $\circlearrowleft^{(r)} a$ occur with fixed positive probability (by Lemma \ref{lem:circle}) and independently of one another.
By choosing $\eta$ small enough, this probability can therefore be made smaller than $\alpha$, uniformly over all the parameters.
\end{proof}

Together these results conclude the convergence of the discrete Poisson kernel ratio and therefore the proof of Lemma \ref{lem:YY}.

\bibliographystyle{alpha}
\bibliography{references}

\end{document}